\documentclass[10pt, oneside]{article}
\usepackage{mathrsfs}
\usepackage{amsfonts}
\usepackage{amssymb} 
\usepackage{amsmath,amsfonts,amssymb,amsthm}
\usepackage{caption}
\usepackage{subfigure}
\usepackage{cite}
\usepackage{color}
\usepackage{graphicx}
\usepackage{subfigure}
\usepackage{float}
\usepackage{paralist}
\usepackage{indentfirst}
\usepackage{bm}
\usepackage{authblk}
\usepackage{tikz}
\usetikzlibrary{shapes,arrows}
\usetikzlibrary{shapes.geometric, arrows}
\usepackage[colorlinks,
linkcolor=blue,
citecolor=blue
]{hyperref}
\usepackage[T1]{fontenc}

\definecolor{red}{rgb}{1.00,0.00,0.00}
\definecolor{blue}{rgb}{0.00,0.00,0.63}
\definecolor{black}{rgb}{0.00,0.00,0.00}
\definecolor{purple}{rgb}{0.00,1.00,0.00}
\definecolor{pink}{rgb}{0.95,0.01,0.08}

\newtheorem{theorem}{Theorem}[section]

\newtheorem{corollary}{Corollary}[section]

\newtheorem{lemma}{Lemma}[section]

\newtheorem{proposition}{Proposition}[section]
\newtheorem{remark}{Remark}[section]

\allowdisplaybreaks
\def\div{ \hbox{\rm div}\,  }
\def\ddj{\dot{\Delta}_{j}}

\def\var{\varepsilon}

\def \andf {\quad\!\hbox{and}\!\quad}

\def\dZ_1{\delta\!Z_1}

\def\ddj{\dot\Delta_j}

\numberwithin{equation}{section}

\providecommand{\keywords}[1]
{
  \textbf{\textit{Keywords---}} #1
}
 \providecommand{\Subjclass}[1]
{
  \textbf{\textit{2020 Mathematics Subject Classification---}} #1
}

\begin{document}


\title{Quantitative derivation of a two-phase porous media system from the one-velocity Baer-Nunziato and Kapila systems
}

\author{ Timoth\'{e}e Crin-Barat\footnote{ Corresponding author: timotheecrinbarat@gmail.com} , Ling-Yun Shou, Jin Tan}

\date{}
 	\maketitle

\begin{abstract}
We derive a novel two-phase flow system in porous media as a relaxation limit of compressible multi-fluid systems.
Considering a one-velocity Baer-Nunziato system with friction forces, we first justify its pressure-relaxation limit toward a Kapila model in a uniform manner with respect to the time-relaxation parameter associated with the friction forces. Then, we show that the diffusely rescaled solutions of the damped Kapila system converge to the solutions of the new two-phase porous media system as the time-relaxation parameter tends to zero. In addition, we also prove the convergence of the Baer-Nunziato system to the same two-phase porous media system as both relaxation parameters tend to zero. For each relaxation limit, we exhibit sharp rates of convergence in a critical regularity setting.

 Our proof is based on an elaborate low-frequency and high-frequency analysis via the  Littlewood-Paley decomposition and includes three main ingredients: a refined spectral analysis of the linearized problem to determine the frequency threshold explicitly in terms of the time-relaxation parameter, the introduction of an effective flux in the low-frequency region to overcome the loss of parameters due to the  {\em overdamping phenomenon}, and renormalized energy estimates in the high-frequency region to cancel higher-order nonlinear terms. To justify the convergence rates, we discover several {\em auxiliary unknowns} allowing us to recover crucial $\mathcal{O}(\varepsilon)$ bounds.

\end{abstract}


\Subjclass{35Q35; 35B40; 76N10; 76T17}

\section{Introduction}

\subsection{Models and motivations}
 Multi-phase flows have been used to simulate a wide range of physical mixing phenomenon, from engineering to biological systems (cf. \cite{baer1,bresch3,ishii2,wallis1} and the references therein). In the present paper, we investigate an inviscid compressible one-velocity Baer-Nunziato system  with two different pressure laws in presence of drag forces,  which was discussed in the recent work \cite{BH1} of Bresch and Hillairet:
\begin{equation}\tag{BN}
\left\{
\begin{aligned}
&\partial_{t}\alpha_{+} +u \cdot\nabla \alpha_{+} =\frac{\alpha_{+} \alpha_{-} }{\var}\big( P_{+}(\rho_{+} )-P_{-}(\rho_{-} ) \big),\\
&\partial_{t}(\alpha_{\pm} \rho_{\pm} )+\div (\alpha_{\pm}  \rho_{\pm}  u )=0,\\
&\partial_{t}(\rho  u )+\div (\rho  u \otimes u )+\nabla P +\frac{\rho  u }{\tau}=0,\quad \quad x\in\mathbb{R}^{d},\quad t>0,
\end{aligned}\right.\label{BN}
\end{equation}
where the unknowns $\alpha_{\pm} =\alpha_{\pm} (t,x)\in [0,1]$, $\rho_{\pm} =\rho_{\pm} (t,x)\geq0$ and $u =u(t,x)\in\mathbb{R}^{d}$ stand for the volume fractions, the densities and the common velocity of two fluids (denoted by $+$ and $-$), respectively, which satisfy
\begin{equation}\nonumber
\begin{aligned}
&\alpha_++\alpha_-=1,\quad\quad \rho=\alpha_{+}\rho_{+}+\alpha_{-}\rho_{-},\quad\quad P=\alpha_{+}P_{+}(\rho_{+})+\alpha_{-}P_{-}(\rho_{-}).
\end{aligned}
\end{equation}
The two positive constants $\var$ and  $\tau$ are  (small) relaxation parameters associated to the pressure-relaxation and time-relaxation limits.
Finally, the two  pressures  $P_{+}$ and $P_{-}$ take the gamma-law forms
\begin{align}\label{AssumPressure}
&P_{\pm}(s)=A_{\pm}s^{\gamma_{\pm}} \quad {\rm with ~constants} \quad A_{\pm}>0,\quad\quad 1\leq \gamma_{-}< \gamma_{+}.
\end{align}

The  \textit{Baer-Nunziato} terminology refers to the pressure-relaxation mechanism in the equations of volume fractions. 
   Numerically, such relaxation procedure can simplify its resolution as it reduces the number of constraints by introducing new unknowns: \textit{two pressures instead of one}. The readers can see  \cite{BH1} and references therein for more discussions on this pressure-relaxation process. Very recently, the one-dimensional version of System \eqref{BN}  was rigorously derived by Bresch, Burtea and Lagouti\'{e}re in \cite{bresch1,breschCos1}.

There is an extensive literature on the mathematical analysis of multi-fluid systems. For example, in the  one-velocity case, the global existence of weak solutions has been studied in \cite{novotny2, bresch4,vasseur1,li2,yao3,wen0}, and the global well-posedness and optimal time-decay rates of strong solutions has been established in the framework of Sobolev spaces \cite{guo1,yao1,zhang1,zhangying1} and critical Besov spaces \cite{hao1,burtea1,lhlshou2}, etc. We also refer to \cite{evje1,bresch2,bresch5,kracmar1,BreschHuangLi} on the study of multi-fluid systems in the two-velocity case. Complete reviews on multi-fluid systems are presented in \cite{breschhand,wen1}.
Concerning the study of relaxation problems associated to systems of conservation laws, it
can be traced back to the work \cite{chen1994} by Chen, Levermore and Liu. Recently, Giovangigli and Yong in \cite{Gio1,Gio2} studied a relaxation problem arising in the dynamics of perfect gases out of thermodynamic equilibrium.




At the formal level, the solution $(\alpha_{\pm}^{\var, \tau},\rho_{\pm}^{\var, \tau},u^{\var, \tau})$ of System \eqref{BN} tends, as $\varepsilon\to0$, to some limit $(\alpha_{\pm}^{  \tau},\rho_{\pm}^{ \tau},u^{ \tau})$ that satisfies the so-called one-velocity Kapila system (cf. \cite{kapila1}):
\begin{equation}\label{K}\tag{K}
\left\{
\begin{aligned}
&\partial_{t}(\alpha_{\pm}^\tau \rho_{\pm}^\tau )+\div (\alpha_{\pm}^\tau  \rho_{\pm}^\tau  u^\tau )=0,\\
&\partial_{t}(\rho^\tau  u^\tau )+\div (\rho ^\tau u^\tau \otimes u^\tau )+\nabla P^\tau +\frac{\rho^\tau  u^\tau }{\tau}=0,\\
&P^\tau=P_{+}^\tau(\rho_{+})=P_{-}^\tau(\rho_{-}),
\end{aligned}
\right.
\end{equation}
with $\alpha_{+}^\tau+\alpha_{-}^\tau=1$ and $\rho^\tau=\alpha^\tau_{+}\rho^\tau_{+}+\alpha^\tau_{-}\rho_{-}^\tau$. System \eqref{K} can be rewritten as classical two-phase fluid models of drift-flux type, see \cite{wen0,ishii2,evje11} and the references therein. For existence of finite energy weak solutions to System \eqref{K} with viscosities, refer to the recent works \cite{novotny2, bresch4,li2,wen0}. 
 

Then, we further investigate the time-relaxation limit of System \eqref{K} as $\tau\to 0.$ Inspired by the works \cite{cou1,junca1,xu00} concerning the relaxation problems for the compressible Euler system with damping, we introduce a large time-scale $\mathcal{O}(1/\tau)$ and define the following charge of variables
\begin{align}
    (\beta_{\pm}^{\tau}, \varrho_{\pm}^{\tau}, v^{\tau})(s,x):=\left(\alpha_{\pm}^{\tau}, \rho_{\pm}^{\tau}, \frac{u^{\tau}}{\tau}\right)\left(\frac{s}{\tau},x\right).\label{scalingK}
\end{align} 
Under the diffusive scaling \eqref{scalingK},  System \eqref{K} becomes
\begin{equation}\tag{${\rm K_{ {\tau}} }$}
\left\{
\begin{aligned}
&\partial_{s}(\beta_{\pm}^{\tau}\varrho_{\pm}^{\tau})+\div (\beta_{\pm}^{\tau} \varrho_{\pm}^{\tau} v^{\tau})=0,\\
&\tau^2\partial_{s}(\varrho^{\tau} v^{\tau})+\tau^2\div (\varrho^{\tau} v^{\tau}\otimes v^{\tau})+\nabla \Pi^{\tau}+\varrho^{\tau} v^{\tau}=0,
\\& \beta_+^\tau+\beta_-^\tau=1,
\end{aligned}
\right.\label{Keta}
\end{equation}
with $\varrho^{\tau}=\beta_{+}^{\tau}\rho_{+}^{\tau}+\beta_{-}^{\tau}\varrho_{-}^{\tau}$ and $\Pi^{\tau}=P_{+}(\varrho_{+}^{\tau})=P_{-}(\varrho_{-}^{\tau})$. As $\tau\rightarrow0$, one then expects that $(\beta_{\pm}^{\tau}, \varrho_{\pm}^{\tau},v^{\tau})$ converges to some limit $(\beta_{\pm}, \varrho_{\pm},v)$ which is the solution of a new two-phase system
\begin{equation}
\left\{
\begin{aligned}
&\partial_{s}(\beta_{\pm}\varrho_{\pm})+\div (\beta_{\pm} \varrho_{\pm} v)=0,\\
&\nabla \Pi+\varrho v=0,
\\& \beta_++\beta_-=1,
\end{aligned}
\right.\label{PM1}
\end{equation}
 with $\varrho=\beta_{+}\varrho_{+}+\beta_{-}\varrho_{-}$ and $\Pi=P_{+}(\varrho_{+})=P_{-}(\varrho_{-})$. Inserting Darcy's law $\eqref{PM1}_{2}$ into $\eqref{PM1}_{1}$, we derive the following two-phase system in porous media:
\begin{equation}\tag{PM}
\left\{
\begin{aligned}
&\partial_{s} \beta_++v\cdot\nabla\beta_+=\dfrac{(\gamma_+-\gamma_-)\beta_+\beta_-}{\gamma_+\beta_-+\gamma_-\beta_+}\,\div \left(\frac{\nabla \Pi}{\beta_{+}\varrho_{+}+\beta_{-}\varrho_{-}}\right),\\
&\partial_{s}\Pi+v\cdot\nabla \Pi=\frac{\gamma_{+}\gamma_{-}\Pi}{\gamma_{+}\beta_{-}+\gamma_{-}\beta_{+}}\div \left(\frac{\nabla \Pi}{\beta_{+}\varrho_{+}+\beta_{-}\varrho_{-}}\right),
\\& \beta_++\beta_-=1,\\
&\Pi=P_{+}(\varrho_{+})=P_{-}(\varrho_{-}).
\end{aligned}
\right.\label{PM}
\end{equation}

   \bigbreak
   


The present paper is a follow-up to the paper \cite{burtea1} by Burtea, Crin-Barat and Tan where the authors justified the pressure-relaxation limit  for the viscous version of System \eqref{BN} to System \eqref{K}. 
In \cite{burtea1}, the smallness  condition on initial data employed to justify their global well-posedness result depends on $\min\{\tau, \frac{1}{\tau}\}$ (due to the overdamping phenomenon that will be explained below) and therefore does not allow to further investigate the limit when $\tau\to0$. 
\medbreak
\textit{The main results of this article are the quantitative justification of the pressure-relaxation limit from System \eqref{BN} to System \eqref{K} as $\var\rightarrow0$ uniformly in $\tau$ and the time-relaxation limit from System \eqref{Keta} to System \eqref{PM}} as $\tau\rightarrow0$. Consequently,  a new two-phase flow system in porous media \eqref{PM} is rigorously derived from Systems \eqref{Keta} and \eqref{BNvar}, which implies that Kapila and Baer-Nunziato systems considered in our paper can be viewed, for $\varepsilon$ and $\tau $ small enough,  as hyperbolic approximations of \eqref{PM}.
\smallbreak


For both relaxation limits, we will focus on global-in-time strong solutions being small perturbations of constant equilibrium states. In other words,
we consider   solutions $(\alpha_{\pm}^{\var, \tau},\rho_{\pm}^{\var, \tau},u^{\var, \tau})$  to System \eqref{BN} (resp.  $(\alpha_{\pm}^{  \tau},\rho_{\pm}^{  \tau},u^{\tau})$ to System \eqref{K}) with positive densities and volume fractions which, as $|x|\rightarrow\infty$, tend to  some  thermodynamically stable equilibrium state $(\bar{\alpha}_{\pm},\bar{\rho}_{\pm}, 0)$ fulfilling
\begin{equation}\label{far}
    0<\bar \alpha_{\pm}<1,\quad\quad \bar\alpha_{+}+\bar\alpha_{-}=1,\quad\quad \bar\rho_{\pm}>0,\quad\quad P_{+}(\bar{\rho}_{+})=P_{-}(\bar{\rho}_{-}).
\end{equation}

For convenience, we also define the corresponding equilibrium state for the total density and the total pressure as
\begin{equation}\label{far1}
\bar{\rho} :=\bar{\alpha}_{+} \bar{\rho}_{+} +\bar{\alpha}_{-} \bar{\rho}_{-} ,\quad\quad \bar{P}:=P_{+}(\bar{\rho}_{+} )=P_{-}(\bar{\rho}_{-} ).
\end{equation}




To achieve our goals, we prove uniform in $\var$ and $\tau$ (such that $\var\leq \tau$) a priori estimate for System \eqref{BN} which improves the analysis performed in \cite{burtea1} that did not provide uniform-in-$\tau$ estimate. Such estimate allows us to justify a global well-posedness for a class of non-symmetric partially dissipative hyperbolic systems with rough coefficients in the context of overdamping phenomenon, which is not covered by the recent lecture of Danchin \cite{danchin5}. 
Indeed, our proof generalizes the techniques developed in \cite{c0, c1,jinxin,BCBP} which cannot be directly applied to System \eqref{BN} due to the complex forms of the total pressure and the lack of symmetry. 

\medbreak
 On the other hand, it is natural to ask what happens for System \eqref{BN} as $\tau$ tends to 0 first. To investigate this process, we introduce a diffusive scaling similar to \eqref{scalingK} as follows
\begin{align}
(\beta_{\pm}^{\var, \tau}, \varrho_{\pm}^{\var, \tau}, v^{\var, \tau})(s,x):=\left(\alpha_{\pm}^{\var, \tau},\rho_{\pm}^{\var, \tau}, \frac{u^{\var, \tau}}{\tau}\right)\left(\frac{s}{\tau},x\right).\label{scalingBN}
\end{align}
Under such scaling \eqref{scalingBN}, System \eqref{BN} becomes
\begin{equation}\tag{${\rm BN_{\tau} }$}
\left\{
\begin{aligned}
&\partial_{s}\beta^{\var, \tau}_{+}+v^{\var, \tau}\cdot\nabla \beta^{\var, \tau}_{+}=-\frac{\beta^{\var, \tau}_{+}\beta^{\var, \tau}_{-}}{\var\tau}\big(P_{+}(\varrho^{\var, \tau}_{+})-P_{-}(\varrho^{\var, \tau}_{-}) \big),\\
&\partial_{s}(\beta^{\var, \tau}_{\pm}\varrho^{\var, \tau}_{\pm})+\div (\beta^{\var, \tau}_{\pm} \varrho^{\var, \tau}_{\pm} v^{\var, \tau})=0,\\
&\tau^2\partial_{s}(\varrho^{\var, \tau} v^{\var, \tau})+\tau^2\div (\varrho^{\var, \tau} v^{\var, \tau}\otimes v^{\var, \tau})+\nabla \Pi^{\var, \tau}+\varrho^{\var, \tau} v^{\var, \tau}=0,\label{BNvar}
\end{aligned}\right.
\end{equation}
with $\beta^{\var,\tau}_{+}+\beta_{-}^{\var,\tau}=1$, $\varrho^{\var, \tau}=\beta^{\var, \tau}_{+}\varrho^{\var, \tau}_{+}+\beta^{\var, \tau}_{-}\varrho^{\var, \tau}_{-}$ and $\Pi^{\var, \tau}=\beta^{\var, \tau}_{+}P_{+}(\varrho^{\var, \tau}_{+})+\beta^{\var, \tau}_{-}P_{-}(\varrho^{\var, \tau}_{-})$. The crucial observation is that the parameter $\tau$ now also appears under the pressure-relaxation term in the equation of the volume fractions. This reveals that as $\tau\to0$, the two pressures in System \eqref{BNvar} should converge to a common pressure, and the solutions of System \eqref{BN} should converge to the solutions of System \eqref{K} regardless of $\var$. Additionally, in the sequel of the paper we are only able to justify the limit in the case $\var\leq\tau$ which corresponds to the situation that the time-scale of the pressure-relaxation is small than the time-scale of the diffusive relaxation. The condition $\var\leq\tau$ appears in the spectral analysis of the system (see Section \ref{sec:spectral}) and  is essential for us  to close the uniform a-priori estimate in both low and high frequencies (See Sections \ref{subsectionlow}-\ref{subsectionhigh} for the details). But in a formal way, the condition $\var\leq\tau$ seems not necessary in the limit process $\tau\to 0$, so the case $\var>\tau$ remains an interesting  open problem.
The Figure \ref{fig:tri-limits} summarizes the limit processes that we tackle in this article.

\tikzstyle{system} = [rectangle,rounded corners, minimum width=2.5cm,minimum height=1cm,text centered, draw=black]
\tikzstyle{arrow} = [thick,->,>=stealth]
 
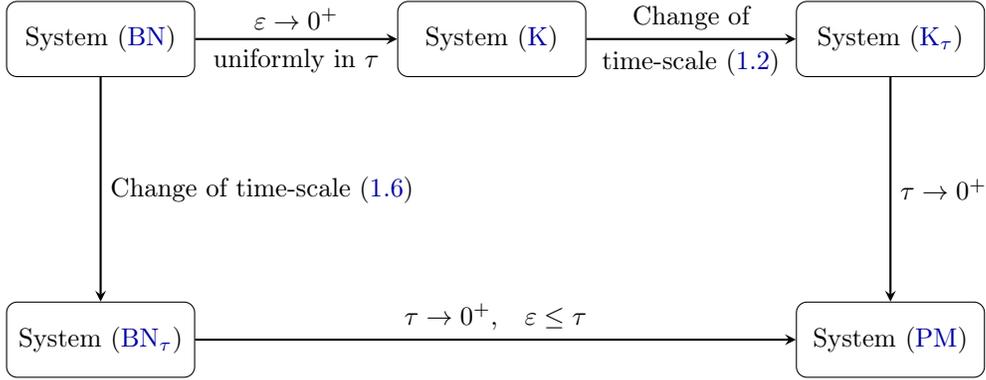
\begin{figure}[!ht]
	\centering
	\begin{tikzpicture}[node distance=2cm]
\node (BN) [system] at (0, 0) {System \eqref{BN}};
\node (K) [system] at (5.2, 0) {System \eqref{K}};
\node (K-tau) [system] at (10.5, 0) {System \eqref{Keta}};
\node (PME) [system] at (10.5, -4) {System \eqref{PM}};
\node (BN-tau) [system] at (0, -4) {System \eqref{BNvar}};

\draw [arrow] (BN) --node[anchor=south]{$\varepsilon\to 0^+$}node[anchor=north]{uniformly in $\tau$} (K);
\draw [arrow] (K) --node[anchor=south]{Change of}node[anchor=north]{time-scale \eqref{scalingK} } (K-tau);
\draw [arrow] (K-tau) --node[anchor=west]{$\tau\to 0^+$} (PME);
\draw [arrow] (BN) -- node[anchor=west]{Change of time-scale \eqref{scalingBN} } (BN-tau);
\draw [arrow] (BN-tau) --node[anchor=south]{$\tau\to 0^+,$ ~ $\var\leq\tau$} (PME);
\end{tikzpicture}
	\caption{Relaxation limits diagram.}
	\label{fig:tri-limits}
\end{figure}

\subsection{Outline of the paper}

The rest of the paper is organized as follows. Our main results are stated in Section \ref{subsectionresult}. In Section \ref{sec:spectral}, we first recall a reformulation of System \eqref{BN} from \cite{burtea1} and present an explicit spectral analysis for the associated linear system, then the difficulties and strategies of proof are discussed. Section \ref{sectionbesov} is devoted to some notations and properties of Besov spaces and Littlewood-Paley decomposition, and the regularity estimates for some linear problems are stated. In Section \ref{section3}, we establish uniform a priori estimate for the linearized problem.  Next, in Section \ref{section4}, we prove the global existence and uniqueness results of solutions for Systems \eqref{BN}, \eqref{K} and \eqref{PM}, respectively. Section \ref{sectionrelaxation} is devoted to the justification of the relaxation limits with explicit convergence rates.

\vspace{2mm}

\textbf{Notations.} We end this section by presenting  a few notations. 
As usual,  we denote by $C$ (and sometimes with subscripts) harmless  positive constants that may change from line to line, and  $A\lesssim B$ ($A\gtrsim B$) means that both $A\leq C B$ ($A\geq C B$),  while $A\sim B$ means that $A\lesssim  B$ 
and $A\gtrsim B.$ 
For $X$ a Banach space, $p\in[1, \infty]$ and 
$T>0$, the notation $L^p(0, T; X)$ or $L^p_T(X)$ designates the set of measurable functions $f: [0, T]\to X$ with $t\mapsto\|f(t)\|_X$ in $L^p(0, T)$, endowed with the norm $\|\cdot\|_{L^p_{T}(X)} :=\|\|\cdot\|_X\|_{L^p(0, T)}.$ We agree that $\mathcal C_b([0, T];X)$ denotes the set of continuous and  bounded (uniformly in $T$) functions from $[0, T]$ to $X$.  Sometimes, we  use the notation $L^p(X)$ to designate the space
$L^p(\mathbb{R}_+;X)$ and $\|\cdot\|_{L^p(X)}$ for the associated norm. 
 We will keep the same notations for multi-component functions, namely  for
 $f:[0,T]\to X^m$ with $m\in\mathbb{N}.$ $\mathcal{F}$ and $\mathcal{F}^{-1}$ stand for the Fourier transform and its inverse, and define the operator $\Lambda^{\sigma}:=\mathcal{F}^{-1}(|\xi|^{\sigma}\mathcal{F}(\cdot))$. Finally, let ${o}(1)$ denote a generic constant which can be sufficiently small. 

\subsection{Main results}\label{subsectionresult}

Our first theorem concerns the uniform, in both relaxation parameters $\var$ and $\tau$, global well-posedness of System \eqref{BN} in a critical regularity framework.

\begin{theorem} \label{theorem11}
Let $d\geq2$ and $0<\var\leq\tau\leq 1.$  Given the constants $\bar\alpha_{\pm}, \bar\rho_{\pm}$ verifying \eqref{far}-\eqref{far1}, assume that the initial data $(\alpha_{\pm,0},\rho_{\pm,0},u_{0})$ satisfies $(\alpha_{\pm,0}-\bar\alpha_{\pm},\rho_{\pm,0}-\bar\rho_{\pm},u_{0})\in \dot{B}^{\frac{d}{2}-1}\cap\dot{B}^{\frac{d}{2}+1}$. There exists a positive constant $c_0$ independent of $\tau$ and $\var$ such that if 
\begin{align}
&\|(\alpha_{\pm,0}-\bar{\alpha}_{\pm}, \rho_{\pm,0}-\bar{\rho}_{\pm},u_{0})\|_{\dot{B}^{\frac{d}{2}-1}\cap\dot{B}^{\frac{d}{2}+1}}\leq c_{0},\label{a1}
\end{align}
then the Cauchy problem of System \eqref{BN} with the initial data  $(\alpha_{\pm,0},\rho_{\pm,0},u_{0})$ has a unique global  solution $(\alpha^{\var, \tau}_{\pm},\rho^{\var, \tau}_{\pm}, u^{\var, \tau})$ satisfying
\begin{equation}\label{r1}
\left\{
\begin{aligned}
&(\alpha^{\var, \tau}_{\pm}-\bar{\alpha}_{\pm}, \rho^{\var, \tau}_{\pm}-\bar{\rho}_{\pm},u^{\var, \tau})\in \mathcal{C}_{b}(\mathbb{R}_{+};\dot{B}^{\frac{d}{2}-1}\cap\dot{B}^{\frac{d}{2}+1}), \\
& P_{+}(\rho_{+}^{\var, \tau})-P_{-}(\rho_{-}^{\var, \tau})\in L^1(\mathbb{R}_{+};\dot{B}^{\frac{d}{2}-1}\cap\dot{B}^{\frac{d}{2}+1}),\\
&P^{\var, \tau}-\bar{P}\in L^1(\mathbb{R}_{+};\dot{B}^{\frac{d}{2}+1})\cap {L}^2(\mathbb{R}_{+};\dot{B}^{\frac{d}{2}}\cap\dot{B}^{\frac{d}{2}+1}),\\
& u^{\var, \tau}\in L^1(\mathbb{R}_{+};\dot{B}^{\frac{d}{2}}\cap\dot{B}^{\frac{d}{2}+1})\cap  {L}^2(\mathbb{R}_{+};\dot{B}^{\frac{d}{2}-1}\cap\dot{B}^{\frac{d}{2}+1}).
\end{aligned}
\right.
\end{equation}
 
Moreover, the following uniform estimate holds{\rm:}
\begin{equation}\label{uniform1}
\begin{aligned}
&\|(\alpha^{\var, \tau}_{\pm}-\bar{\alpha}_{\pm}, \rho^{\var, \tau}_{\pm}-\bar{\rho}_{\pm},u^{\var, \tau})\|_{ {L}^{\infty} (\dot{B}^{\frac{d}{2}-1}\cap\dot{B}^{\frac{d}{2}+1})}+\|(\partial_{t}\alpha^{\var, \tau}_{\pm}, \partial_{t}\rho^{\var, \tau}_{\pm},\partial_{t}u^{\var, \tau})\|_{ {L}^{1}(\dot{B}^{\frac{d}{2}})}\\
&\quad\quad+\frac{1}{\var}\|P_{+}(\rho_{+}^{\var, \tau})-P_{-}(\rho_{-}^{\var, \tau})\|_{L^1(\dot{B}^{\frac{d}{2}-1}\cap\dot{B}^{\frac{d}{2}})}+\frac{1}{\sqrt{\var}}\|P_{+}(\rho_{+}^{\var, \tau})-P_{-}(\rho_{-}^{\var, \tau})\|_{ {L}^2 (\dot{B}^{\frac{d}{2}-1}\cap\dot{B}^{\frac{d}{2}+1})}\\
&\quad\quad+\tau \|P^{\var, \tau}-\bar{P}\|_{L^1(\dot{B}^{\frac{d}{2}+1})}+\sqrt{\tau} \|P^{\var, \tau}-\bar{P}\|_{ {L}^2(\dot{B}^{\frac{d}{2}}\cap\dot{B}^{\frac{d}{2}+1})}\\
&\quad\quad+\|u^{\var, \tau}\|_{L^1 (\dot{B}^{\frac{d}{2}}\cap\dot{B}^{\frac{d}{2}+1})}+\frac{1}{\sqrt{\tau}}\|u^{\var, \tau}\|_{ {L}^2(\dot{B}^{\frac{d}{2}-1}\cap\dot{B}^{\frac{d}{2}+1})}\\
&\quad\quad+\left\|\frac{ \rho^{\var,\tau}u^{\var, \tau}}{\tau}+\nabla P^{\var, \tau}\right\|_{L^1(\dot{B}^{\frac{d}{2}-1}\cap\dot{B}^{\frac{d}{2}})}\leq C\|(\alpha_{\pm,0}-\bar{\alpha}_{\pm}, \rho_{\pm,0}-\bar{\rho}_{\pm},u_{0})\|_{\dot{B}^{\frac{d}{2}-1}\cap\dot{B}^{\frac{d}{2}+1}},
\end{aligned}
\end{equation}
where $C>0$ is a generic constant.  
\end{theorem}

\begin{remark}
It should be emphasized that the regularity and decay-in-$\tau$ properties of the effective flux $\dfrac{ \rho^{\var,\tau}u^{\var, \tau}}{\tau}+\nabla P^{\var, \tau}$ is better than the one verified by  the solution $(\alpha^{\var, \tau}_{\pm},\rho^{\var, \tau}_{\pm}, u^{\var, \tau})$. This is consistent with Darcy's law and plays key role in the justification of the time-relaxation limit.
\end{remark}



By classical compactness arguments and the uniform estimate \eqref{uniform1}, we obtain the following global well-posedness theorems for Systems \eqref{K} and \eqref{PM} in the  critical regularity framework. 

\begin{theorem}\label{theoremK}
Let $d\geq2$ and $0<\tau\leq 1.$ Given the constants $\bar\alpha_{\pm}, \bar\rho_{\pm}$ verifying \eqref{far}-\eqref{far1}, assume that the initial data $(\alpha_{\pm,0},\rho_{\pm,0},u_{0})$ satisfies $(\alpha_{\pm,0}-\bar\alpha_{\pm},\rho_{\pm,0}-\bar\rho_{\pm},u_{0})\in \dot{B}^{\frac{d}{2}-1}\cap\dot{B}^{\frac{d}{2}+1}$. There exists a positive constant $c_1$ independent of $\tau$ such that if 
\begin{align}
&\|(\alpha_{\pm,0}-\bar{\alpha}_{\pm}, \rho_{\pm,0}-\bar{\rho}_{\pm},u_{0})\|_{\dot{B}^{\frac{d}{2}-1}\cap\dot{B}^{\frac{d}{2}+1}}\leq c_{1},\label{a1K}
\end{align}
then the Cauchy problem of System \eqref{K} with the initial data  $(\alpha_{\pm,0},\rho_{\pm,0},u_{0})$ admits a unique global  solution $(\alpha^{\tau}_{\pm},\rho^{\tau}_{\pm}, u^{\tau})$ satisfying
\begin{equation}\label{r1K}
\left\{
\begin{aligned}
&(\alpha^{\tau}_{\pm}-\bar{\alpha}_{\pm}, \rho^{\tau}_{\pm}-\bar{\rho}_{\pm},u^{\tau})\in \mathcal{C}_{b}(\mathbb{R}_{+};\dot{B}^{\frac{d}{2}-1}\cap\dot{B}^{\frac{d}{2}+1}), \\
&P^{\tau}-\bar{P}\in L^1(\mathbb{R}_{+};\dot{B}^{\frac{d}{2}+1})\cap {L}^2(\mathbb{R}_{+};\dot{B}^{\frac{d}{2}}\cap\dot{B}^{\frac{d}{2}+1}),\\
& u^{\tau}\in L^1(\mathbb{R}_{+};\dot{B}^{\frac{d}{2}}\cap\dot{B}^{\frac{d}{2}+1})\cap  {L}^2(\mathbb{R}_{+};\dot{B}^{\frac{d}{2}-1}\cap\dot{B}^{\frac{d}{2}+1}).
\end{aligned}
\right.
\end{equation}
 
Moreover, the following uniform estimate holds{\rm:}
\begin{equation}\label{uniform1K}
\begin{aligned}
&\|(\alpha^{\tau}_{\pm}-\bar{\alpha}_{\pm}, \rho^{\tau}_{\pm}-\bar{\rho}_{\pm},u^{\tau})\|_{ {L}^{\infty} (\dot{B}^{\frac{d}{2}-1}\cap\dot{B}^{\frac{d}{2}+1})}+\|(\partial_{t}\alpha^{\tau}_{\pm}, \partial_{t}\rho^{\tau}_{\pm},\partial_{t}u^{\tau})\|_{ {L}^{1}(\dot{B}^{\frac{d}{2}})}+\tau \|P^{\tau}-\bar{P}\|_{L^1(\dot{B}^{\frac{d}{2}+1})}\\
&\quad+\sqrt{\tau} \|P^{\tau}-\bar{P}\|_{ {L}^2(\dot{B}^{\frac{d}{2}}\cap\dot{B}^{\frac{d}{2}+1})}+\|u^{\tau}\|_{L^1 (\dot{B}^{\frac{d}{2}}\cap\dot{B}^{\frac{d}{2}+1})}+\frac{1}{\sqrt{\tau}}\|u^{\tau}\|_{ {L}^2(\dot{B}^{\frac{d}{2}-1}\cap\dot{B}^{\frac{d}{2}+1})}\\
&\quad\quad+\left\|\frac{ \rho^{\var,\tau}u^{\tau}}{\tau}+\nabla P^{\tau}\right\|_{L^1(\dot{B}^{\frac{d}{2}-1}\cap\dot{B}^{\frac{d}{2}})}\leq C\|(\alpha_{\pm,0}-\bar{\alpha}_{\pm}, \rho_{\pm,0}-\bar{\rho}_{\pm},u_{0})\|_{\dot{B}^{\frac{d}{2}-1}\cap\dot{B}^{\frac{d}{2}+1}},
\end{aligned}
\end{equation}
where $C>0$ is a generic constant.  
\end{theorem}


\begin{theorem}\label{theoremPM}
Let $d\geq2$. Given the constants $\bar\alpha_{\pm}, \bar\rho_{\pm}$ verifying \eqref{far}-\eqref{far1}, assume that the initial data $(\beta_{\pm,0}, \varrho_{\pm,0})$ satisfies $(\beta_{\pm,0}-\bar\alpha_{\pm},\varrho_{\pm,0}-\bar\rho_{\pm})\in \dot{B}^{\frac{d}{2}-1}\cap\dot{B}^{\frac{d}{2}+1}$ and
\begin{align}
&\|(\beta_{\pm,0}-\bar{\alpha}_{\pm},\varrho_{\pm,0}-\bar{\rho}_{\pm})\|_{\dot{B}^{\frac{d}{2}-1}\cap\dot{B}^{\frac{d}{2}+1}}\leq c_{2},\label{a1PM}
\end{align}
for a positive constant $c_2$, then the Cauchy problem of System \eqref{PM} with the initial data  $(\beta_{\pm,0},\varrho_{\pm,0})$ admits a unique global  solution $(\beta_{\pm},\varrho_{\pm})$, which satisfies
\begin{equation}\label{r1PM}
\left\{
\begin{aligned}
&\beta_{\pm}-\bar{\alpha}_{\pm}\in \mathcal{C}_{b}(\mathbb{R}_{+};\dot{B}^{\frac{d}{2}-1}\cap\dot{B}^{\frac{d}{2}+1}),\\
&\varrho_{\pm}-\bar{\rho}_{\pm}\in \mathcal{C}_{b}(\mathbb{R}_{+};\dot{B}^{\frac{d}{2}-1}\cap\dot{B}^{\frac{d}{2}+1})\cap L^1(\mathbb{R}_{+};\dot{B}^{\frac{d}{2}+1}).
\end{aligned}
\right.
\end{equation}
 
Moreover, the following uniform estimate holds{\rm:}
\begin{align}\label{uniform1PM}
\|(\beta_{\pm}-\bar{\alpha}_{\pm}, \varrho_{\pm}-\bar{\rho}_{\pm})\|_{ {L}^{\infty} (\dot{B}^{\frac{d}{2}-1}\cap\dot{B}^{\frac{d}{2}+1})}+\|(\partial_{t}\beta_{\pm},& \partial_{t}\varrho_{\pm})\|_{ {L}^{1}(\dot{B}^{\frac{d}{2}})}+\|\varrho_{\pm}-\bar{\rho}_{\pm}\|_{ {L}^{1} (\dot{B}^{\frac{d}{2}+1}\cap\dot{B}^{\frac{d}{2}+3})}\\
&\leq C\|(\beta_{\pm,0}-\bar{\alpha}_{\pm}, \varrho_{\pm,0}-\bar{\rho}_{\pm})\|_{\dot{B}^{\frac{d}{2}-1}\cap\dot{B}^{\frac{d}{2}+1}},
\end{align}
where $C>0$ is a generic constant.  
\end{theorem}



Next, we present the rigorous justifications of the pressure-relaxation limit for System \eqref{BN} to System \eqref{K} as $\var\rightarrow0$ uniform with respect to $\tau$, and further the time-relaxation limit for System \eqref{Keta} to System \eqref{PM} as $\tau\rightarrow0,$  with explicit convergence rates.

\begin{theorem}\label{theorem12}
Let $d\geq2$ and $0<\var\leq\tau\leq 1.$  Given the constants $\bar\alpha_{\pm}, \bar\rho_{\pm}$ verifying \eqref{far}-\eqref{far1}, let
$(\alpha^{\var, \tau}_{\pm},\rho^{\var, \tau}_{\pm}, u^{\var, \tau})$, $(\alpha^{\tau}_{\pm},\rho^{\tau}_{\pm}, u^{\tau})$ and  $(\beta_{\pm},\varrho_{\pm})$ be the global solutions to the Cauchy problems of Systems \eqref{BN}, \eqref{K} and \eqref{PM} obtained from Theorems \ref{theorem11}-\ref{theoremPM} associated to their corresponding initial data $(\alpha^{\var, \tau}_{\pm,0},\rho^{\var, \tau}_{\pm,0}, u_{0}^{\var, \tau})$, $(\alpha^{\tau}_{\pm,0},\rho^{\tau}_{\pm,0}, u^{\tau}_{0})$ and $(\beta_{\pm,0},\varrho_{\pm,0})$, respectively.



\begin{itemize}
    \item Let the initial quantities $P^{\var,\tau}_{0}-P^{\tau}_{0}$ and $Y_{0}^{\var,\tau}-Y_{0}^{\tau}$ be denoted by \eqref{deltaP0} and \eqref{deltaY0}, respectively.  If $d\geq3$ and
  \begin{equation}\label{errora1}
      \begin{aligned}
      &\|\big(P_{+}(\rho_{+,0}^{\var,\tau})-P_{-}(\rho_{-,0}^{\var,\tau}), Y_{0}^{\var,\tau}-Y_{0}^{\tau}, P^{\var,\tau}_{0}-P^{\tau}_{0},u_{0}^{\var, \tau}-u_{0}^{ \tau} \big)\|_{\dot{B}^{\frac{d}{2}-2}\cap\dot{B}^{\frac{d}{2}-1}} \leq \sqrt{ \var\tau},
      \end{aligned}
  \end{equation}  
  then there exists a universal constant $C_{1}$ such that the following estimate holds{\rm:}
\begin{equation}\label{error1}
\begin{aligned}
 &\|(\alpha_{\pm}^{\var, \tau}-\alpha_{\pm}^{\tau},\rho_{\pm}^{\var, \tau}-\rho_{\pm}^{\tau},u^{\var,\tau}-u^{\tau})\|_{ {L}^{\infty} (\dot{B}^{\frac{d}{2}-2}\cap\dot{B}^{\frac{d}{2}-1})}\\
 &\quad\quad+\sqrt{\tau}\|\rho_{\pm}^{\var, \tau}-\rho_{\pm}^{\tau}\|_{ {L}^2(\dot{B}^{\frac{d}{2}-1})}+\frac{1}{\sqrt{\tau}}\|u^{\var, \tau}-u^{\tau}\|_{ {L}^2(\dot{B}^{\frac{d}{2}-2}\cap\dot{B}^{\frac{d}{2}-1})}\\
&\quad\quad+\|u^{\var,\tau}-u^{\tau}\|_{L^1(\dot{B}^{\frac{d}{2}-1})}\leq C_{1}\sqrt{\var\tau}.
\end{aligned}
\end{equation}
 \item 
 Furthermore, define $(\beta_{\pm}^{\tau}, \varrho_{\pm}^{\tau},v^{\tau})$ by the diffusive scaling \eqref{scalingK} and $v$ by Darcy's law $\eqref{PM1}_{2}$. Let the initial quantity $Z_{0}^{\tau}-Z_{0}$ be denoted by \eqref{deltaZ0}. If 
     \begin{equation}
     \begin{aligned}
    & \hspace{-2mm}
    \|Z_{0}^{\tau}-Z_{0}\|_{\dot{B}^{\frac{d}{2}-1}\cap\dot{B}^{\frac{d}{2}}}
    +\|\varrho_{\pm,0}^{\tau}-\varrho_{\pm,0}\|_{\dot{B}^{\frac{d}{2}-1}}\leq \tau,\label{errora2}
     \end{aligned}
\end{equation}
then there exists a universal constant $C_{2}$ such that the following estimate holds{\rm:}
 \begin{equation} \label{error2}
 \begin{aligned}
 &\|(\beta_{\pm}^{\tau}-\beta_{\pm},\varrho_{\pm}^{\tau}-\varrho_{\pm})\|_{ {L}^{\infty} (\dot{B}^{\frac{d}{2}-1})}\quad\quad+\|\varrho_{\pm}^{\tau}-\varrho_{\pm}\|_{L^1 (\dot{B}^{\frac{d}{2}+1})}+\|v^{\tau}-v\|_{L^1 (\dot{B}^{\frac{d}{2}})}\leq C_{2}\tau.
 \end{aligned}
\end{equation}
\end{itemize}
\end{theorem}

Finally, Theorem \ref{theorem12} implies the relaxation limit for System \eqref{BNvar} to System \eqref{PM} as both $\var,\tau\rightarrow0$.

\begin{corollary}
Let $d\geq3$, $0<\var\leq\tau\leq 1,$ and
$(\beta^{\var, \tau}_{\pm},\varrho^{\var, \tau}_{\pm}, v^{\var, \tau})$ be defined by \eqref{scalingBN}. Then under the assumptions 
of Theorem \ref{theorem12}, there is a generic constant $C_3$ such that
\begin{equation}\nonumber
\begin{aligned}
   & \|(\beta^{\var, \tau}_{\pm}-\beta_{\pm},\varrho_{\pm}^{\var,\tau}-\varrho_{\pm} )\|_{L^{\infty}(\dot{B}^{\frac{d}{2}-1})}\leq C_3( \sqrt{\var\tau}+\tau).
\end{aligned}
\end{equation}
\end{corollary}

\subsection{Difficulties  and strategies}\label{sec:spectral}

The first difficulty concerning the study of System \eqref{BN} are its lack of dissipativity and symmetrizability.
Indeed, the linearization of \eqref{BN} admits the eigenvalue $0$ and therefore does not satisfy the well-known ``Shizuta-Kawashima'' stability condition for partially dissipative hyperbolic systems (cf. \cite{shi1}). Additionally, System \eqref{BN} cannot be written in a conservative form and the entropy  naturally associated to \eqref{BN} is not positive definite, therefore the notion of entropic variables does not make sense in this case. Therefore, the first crucial step in our analysis is to partially symmetrize System \eqref{BN}, by hands. We refer to \cite{BreschHuangLi,ForestierGavrilyuk} for the treatment of non-conservative systems in similar contexts. In our setting, as explained in \cite{burtea1}, we define the new unknowns
\begin{equation}\label{newunknows}
\left\{
\begin{aligned}
&y^{\var, \tau}:=\frac{\alpha_{+}^{\var, \tau}\rho_{+}^{\var, \tau}}{\alpha_{+}^{\var, \tau}\rho_{+}^{\var, \tau}+\alpha_{-}^{\var, \tau}\rho_{-}^{\var, \tau}}-\frac{\bar{\alpha}_{+}\bar{\rho}_{+}}{\bar{\alpha}_{+}\bar{\rho}_{+}+\bar{\alpha}_{-}\bar{\rho}_{-}},\\
& w^{\var, \tau}:=\frac{\alpha_{+}^{\var, \tau}\alpha_{-}^{\var, \tau}}{\gamma_{+}\alpha_{-}^{\var, \tau}+\gamma_{-}\alpha_{+}^{\var, \tau}} \big( P_{+}(\rho_{+}^{\var, \tau})-P_{-}(\rho_{-}^{\var, \tau}) \big),\\
&r^{\var, \tau}:=P^{\var, \tau}-\bar{P}-(\gamma_{+}-\gamma_{-})w^{\var, \tau},
\end{aligned}
\right.
\end{equation}
and the corresponding initial data
\begin{equation}\label{y0w0r0}
\left\{
\begin{aligned}
&y_{0}:=\frac{\alpha_{+,0}\rho_{+,0}}{\alpha_{+,0}\rho_{+,0}+\alpha_{-,0}\rho_{-,0}}-\frac{\bar{\alpha}_{+}\bar{\rho}_{+}}{\bar{\alpha}_{+}\bar{\rho}_{+}+\bar{\alpha}_{-}\bar{\rho}_{-}},\\
& w_{0}:=\frac{\alpha_{+,0}\alpha_{-,0}}{\gamma_{+}\alpha_{-,0}+\gamma_{-}\alpha_{+,0}} \big(P_{+}(\rho_{+,0})-P_{-}(\rho_{-,0}) \big),\\
& r_{0}:=\alpha_{+,0}P_{+}(\rho_{+,0})+\alpha_{-,0}P_{-}(\rho_{-,0})-\bar{P}-(\gamma_{+}-\gamma_{-})w_{0},
\end{aligned}
\right.
\end{equation}
so that the Cauchy problem of System \eqref{BN} subject to the initial data $(\alpha_{\pm,0}, \rho_{\pm,0}, u_{0})$ is reformulated as
\begin{equation}\label{re}
\left\{
\begin{aligned}
&\partial_{t}y^{\var, \tau}+u^{\var, \tau}\cdot\nabla y^{\var, \tau}=0,\\
&\partial_{t}w^{\var, \tau}+u^{\var, \tau}\cdot\nabla w^{\var, \tau}+ (\bar{F}_{1}+G^{\var, \tau}_1  )\div u^{\var, \tau}+ (\bar{F}_2+G^{\var, \tau}_2) \dfrac{w^{\var, \tau}}{\var}=0,\\
&\partial_{t}r^{\var, \tau}+u^{\var, \tau}\cdot\nabla r^{\var, \tau}+ (\bar{F}_3+G^{\var, \tau}_3  )\div u^{\var, \tau}=F_{4}^{\var, \tau}\frac{(w^{\var, \tau})^2}{\var},\\
&\partial_{t}u^{\var, \tau}+u^{\var, \tau}\cdot\nabla u^{\var, \tau}+\frac{u^{\var, \tau}}{\tau}+ ( \bar{F}_{0}+G^{\var, \tau}_0  )\nabla  r^{\var, \tau}+(\gamma_{+}-\gamma_{-}) (\bar{F}_{0}+G^{\var, \tau}_0 ) \nabla w^{\var, \tau}   =0,\\
&(y^{\var, \tau},w^{\var, \tau},r^{\var, \tau},u^{\var, \tau})(0,x)=(y_{0},w_{0},r_{0},u_{0})(x),
\end{aligned}
\right.
\end{equation}
where $F^{\var, \tau}_{i}=F^{\var, \tau}_{i}(y,w,r)$ $(i=0,1,2,3,4)$ are the nonlinear terms
\begin{equation}\label{Fi}
\left\{
\begin{aligned}
\hspace{-20mm}&F^{\var, \tau}_{0}:=\frac{1}{\alpha_{+}^{\var, \tau}\rho_{+}^{\var, \tau}+\alpha_{-}^{\var, \tau}\rho_{-}^{\var, \tau}},\\
&F^{\var, \tau}_{1}:=\frac{(\gamma_{+}-\gamma_{-})\alpha_{+}^{\var, \tau}\alpha^{\var, \tau}_{-}}{\gamma_{+}\alpha^{\var, \tau}_{-}+\gamma_{-}\alpha^{\var, \tau}_{+}}(\bar{P}+r^{\var, \tau} )+\frac{\gamma_{+}^2\alpha^{\var, \tau}_{-}+\gamma_{-}^2\alpha^{\var, \tau}_{+}}{\gamma_{+}\alpha^{\var, \tau}_{-}+\gamma_{-}\alpha^{\var, \tau}_{+}} w^{\var, \tau},\\
&F^{\var, \tau}_{2}:=(\gamma_{+}\alpha^{\var, \tau}_{-}+\gamma_{-}\alpha^{\var, \tau}_{+})(\bar{P}+r^{\var, \tau})-\frac{(\gamma_{+}-\gamma_{+}^2)(\alpha^{\var, \tau}_{-})^2-(\gamma_{-}-\gamma_{-}^2)(\alpha^{\var, \tau}_{+})^2}{\alpha^{\var, \tau}_{+}\alpha^{\var, \tau}_{-}}w^{\var, \tau},\\
&F^{\var, \tau}_{3}:=\frac{\gamma_{+}\gamma_{-}}{\gamma_{+}\alpha^{\var, \tau}_{-}+\gamma_{-}\alpha^{\var, \tau}_{+}}P^{\var, \tau},\\
&F^{\var, \tau}_{4}:=\frac{\gamma_{+}\gamma_{-}}{\alpha^{\var, \tau}_{+}\alpha^{\var, \tau}_{-}}(1-\gamma_{+}\alpha^{\var, \tau}_{-}-\gamma_{-}\alpha^{\var, \tau}_{+}),
\end{aligned}
\right.
\end{equation}
$\bar{F}_{i}$ $(i=0,1,2,3)$ are the constants
\begin{equation}\label{barFi}
\left\{
\begin{aligned}
&\bar{F}_{0}:=\frac{1}{\bar{\alpha}_{+}\bar{\rho}_{+}+\bar{\alpha}_{-}\bar{\rho}_{-} }>0,\\
&\bar{F}_{1}:=\frac{(\gamma_{+}-\gamma_{-})\bar{\alpha}_{+}\bar{\alpha}_{-}}{\gamma_{+}\bar{\alpha}_{-}+\gamma_{-}\bar{\alpha}_{+}}\bar{P}>0,\\
&\bar{F}_{2}:=(\gamma_{+}\bar{\alpha}_{-}+\gamma_{-}\bar{\alpha}_{+})\bar{P}>0,\\
&\bar{F}_{3}:=\frac{\gamma_{+}\gamma_{-}}{\gamma_{+}\bar{\alpha}_{-}+\gamma_{-}\bar{\alpha}_{+}}\bar{P}>0,\\
\end{aligned}
\right.
\end{equation}
and $G^{\var, \tau}_{i}=G^{\var, \tau}_{i}(y,w,r)$ $(i=0,1,2,3)$ are the coefficients 
\begin{equation}
\begin{aligned}
G^{\var, \tau}_{i}:=F^{\var, \tau}_{i}-\bar{F}_{i}.\label{Gi}
\end{aligned}
\end{equation}
In this formulation, the equation $\eqref{re}_{1}$ is purely transport and the linear part of subsystem $\eqref{re}_{2}$-$\eqref{re}_{4}$ is partially dissipative and satisfies the ``Shizuta-Kawashima'' stability condition. Thus, we will estimate the undamped unknown $y^{\var, \tau}$ and the dissipative components $(w^{\var,\tau}, r^{\var, \tau}, u^{\var, \tau})$ separately. We emphasize here that due to the double parameters $\var, \tau$ and the lack of time-integrability of $G^{\var, \tau}_{i}$, the dissipative structures of subsystem $\eqref{re}_{2}$-$\eqref{re}_{4}$ does not fit into the general theorems that can be found in \cite{shi1,xu1,xu00,c0,c1,danchin5}, and a new analysis is needed to be developed to obtain the uniform estimates with respect to the two relaxation parameters $\var, \tau$.

In order to understand the behaviors of the solution to \eqref{re} with respect to $\var, \tau$, we perform a spectral analysis of the linear system for \eqref{re}. For simplicity we set $\bar{F}_{i}=1$ $(i=0,2,3)$ and $\bar{F}_{1}=\gamma_{+}-\gamma_{-}$. In terms of Hodge decomposition, we denote the compressible part $m=\Lambda^{-1}\div u$ and the incompressible part $\Omega=\Lambda^{-1}\nabla\times u$ and rewrite the linear system of \eqref{re} as
 \begin{equation}\nonumber
\begin{aligned}
& \partial_{t}
\left(\begin{matrix}
   w  \\
   r  \\
   m\\
  \end{matrix}\right)
  =\mathbb{A}\left(\begin{matrix}
  w \\
  r\\
  m
     \end{matrix}\right),\quad  \mathbb{A}:=\left(\begin{matrix}
-\dfrac{1}{\var}                                 &    0                                    &      -(\gamma_{+}-\gamma_{-}) \Lambda\\
0                                                                 &    0                                    &      -\Lambda \\            
(\gamma_{+}-\gamma_{-})\Lambda &   \Lambda           &   -\dfrac{1}{\tau}
  \end{matrix}\right),
 \quad\quad \partial_{t} \Omega +\frac{1}{\tau}\Omega=0.
   \end{aligned}
\end{equation}
The eigenvalues of the matrix $\widehat{\mathbb{A}}(\xi)$ satisfy
\begin{equation}\nonumber
\begin{aligned}
&|\widehat{\mathbb{A}}(\xi)-\lambda \mathbb{I}_{3\times 3} |=\lambda^3+\left(\frac{1}{\tau}+\frac{1}{\var}  \right)\lambda^2+\left[\frac{1}{\var\tau}+ \big( |\gamma_{+}-\gamma_{-}|^2+1 \big) |\xi|^2\right] \lambda+\frac{1}{\var} |\xi|^2=0.
\end{aligned}
\end{equation}
Under the condition $0<\var<<\tau$, the behaviors of $\lambda_{i}$ $(i=1,2,3)$ can be analyzed as follows:
\begin{itemize}
    \item In the low-frequency region $|\xi|<<\frac{1}{\tau}$, by Taylor's expansion near $|\tau\xi|<<1$ as in \cite{mats1}, all the eigenvalues are real, and we have $\lambda_{1}= -\frac{1}{\var}+\frac{1}{\tau}\mathcal{O}(|\tau \xi|^2)$,  $\lambda_{2}=-\tau|\xi|^2+\frac{1}{\tau}\mathcal{O}(|\tau \xi|^3)$ and $\lambda_{3}= -\frac{1}{\tau}+\frac{1}{\tau}\mathcal{O}(|\tau \xi|^2)$.
    
   \item In the medium-frequency region $\frac{1}{\tau}<<|\xi|<<\frac{1}{\var}$, according to Cardano's formula, $\lambda_{1}$ is real and $\lambda_{i}$ $(i=2,3)$ are conjugated complex, and Re $\lambda_{i}\lesssim -\frac{1}{\tau}$ holds for all $i=1,2,3$.
   
   \item In the high-frequency region $|\xi|>>\frac{1}{\var}$, by Taylor's expansion near $|\var\xi|^{-1}<<1$, the real eigenvalue $\lambda_{1}$ and the conjugated complex eigenvalues $\lambda_{i}$ $(i=2,3)$ satisfy $\lambda_{1}=-\frac{1}{|\gamma_{+}-\gamma_{-}|^2+1}\frac{1}{\var}+\frac{1}{\var}\mathcal{O}\left(\frac{1}{| \var\xi|^2}\right)$ and $\lambda_{2,3}=-\frac{1}{2\tau}-\frac{|\gamma_{+}-\gamma_{-}|^2}{|\gamma_{+}-\gamma_{-}|^2+1}\frac{1}{2\var}\pm \sqrt{|\gamma_{+}-\gamma_{-}|^2+1} |\xi| i+(\frac{1}{\tau}+\frac{|\gamma_{+}-\gamma_{-}|^2}{\var})\mathcal{O}\left(\frac{1}{|\var\xi|}\right)$.
\end{itemize}
The above spectral analysis suggests us to separate the whole frequencies into two parts $|\xi|\lesssim\frac{1}{\tau}$ and $|\xi|\gtrsim \frac{1}{\tau}$ so as to capture the qualitative properties of solutions for System \eqref{re}. Indeed, the time-decay rates (determined by $\lambda_{2}$) achieve the fastest rate in the low-frequency region $|\xi|\lesssim\frac{1}{\tau}$. Moreover this region recover the whole frequency-space when $\tau\rightarrow0$, as expected from the well-known overdamping phenomenon which will be mentioned below. 
To this end, the threshold $J_{\tau}$ between these two regions is used in the definition of the hybrid Besov spaces in next section.

It should be noted that $\lambda_2$ and $\lambda_3$  exhibit similar behaviors to the eigenvalues of the compressible Euler equations with damping. Indeed, to study System \eqref{re}, one considers the following simplified system of damped Euler type with rough coefficients:   
\begin{equation}\label{L-E}
\left\{\begin{aligned}
&\partial_{t}r^{\tau}+ (1+ G^\tau_3)\div u^{ \tau}=0,\\
&\partial_{t}u^{\tau}+(1+ G^\tau_0)\nabla  r^{ \tau} +\frac{u^{ \tau}}{\tau} =0.
\end{aligned} 
\right. 
\end{equation}
The well-known spectral analysis for the linear Euler part of System \eqref{L-E} implies that the frequency space shall be separated into the low-frequency region $|\xi|\lesssim \frac{1}{\tau}$ and the high-frequency region $|\xi|\gtrsim \frac{1}{\tau}$ to recover the uniform estimates and optimal regularity of solutions. Formally, this implies that as $\tau\to0$, the low-frequency region covers the whole frequency space and is therefore be dominant at the limit. 
We observe here the classical \textit{overdamping phenomenon}:{ as the friction coefficient $\frac{1}{\tau}$ gets larger, the decay rates of $r^\tau$ do not necessarily increase and on the contrary follow }$\min\{\tau, \frac{1}{\tau}\}$, cf. Figure \ref{fig:overdamping}. For more discussion on the overdamping phenomenon, see Zuazua's sildes \cite{SlidesZuazua}.
\begin{figure}[!ht]
	\centering	
 \resizebox{0.69\columnwidth}{!}{\begin{tikzpicture}
\node at (0, 3) [left] {$\omega^*=|\xi|$ };
\draw[->,thick] (0,0) -- (8,0) node[right] {$\frac{1}{\tau}=$ damping};
\draw[->,thick] (0,0) -- (0,4) node[above] {$\omega$: decay rate};

\draw[domain=0:2] plot (\x,1.5*\x) node[below right] {};
\draw[domain=2:8] plot (\x,{6/\x}) node[below right] {};
\draw[dotted] (0,3)--(2, 3); 
\draw[dotted] (2,3)--(2, 0) node[below] {$\frac{1}{\tau*}=2|\xi|$ };
\node[above left] at (1.3,1.8){$\omega=\frac{1}{2\tau}$}; 
  \node[above right] at (4, 1.5) {$\omega=\frac{2\tau|\xi|^2}{1+\sqrt{1-4\tau^2|\xi|^2}}$};
\end{tikzpicture}}
	\caption{A graph of overdamping phenomenon for System \eqref{L-E}.}
	\label{fig:overdamping} 
\end{figure}
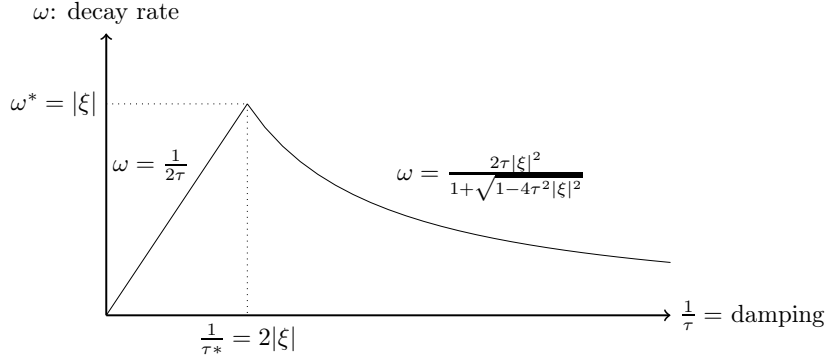

Recently, in \cite{c0,c1}, the issue concerning the 
 relaxation limit from compressible Euler system with damping toward the porous media equation has been rigorously justified  in  critical space  $\dot{B}^{\frac{d}{2}}\cap \dot{B}^{\frac{d}{2}+1}$.  The readers also can refer to the work \cite{c2}  about the relaxation limit for  a hyperbolic-parabolic chemotaxis system to a parabolic-elliptic Keller-Segel model.   The regularity index $\frac{d}{2}+1$ is called critical for initial data of general hyperbolic systems since $\dot{B}^{\frac{d}{2}+1}$ is embedded in the set of globally Lipschitz functions. Indeed, It has been observed by many authors that controlling the Lipschitz regularity of solutions for general hyperbolic systems can prevent blow-up in finite time, see e.g., \cite{xu1,Majda}. We also refer to  \cite{li010,linares1} about the ill-posedness results for hyperbolic systems in $H^{s}$ with $s<\frac{d}{2}+1$.




\medbreak
Nevertheless, the   methods developed in \cite{c0,c1,c2} are not applicable in the current situation to derive estimates which are uniform with respect to the relaxation parameter $\tau$. This is mainly due to the \emph{complex form of the total pressure $P^{\var,\tau}$} in the velocity equation $\eqref{BN}_{3}$ and the fact that one can not expect any time integrability  property on $\mathbb{R}_+$ for the purely transported unknown $y^{\var, \tau},$ which generally leads to a lack of time integrability on $\mathbb{R}_+$ for $(\alpha_{\pm}^{\var,\tau}-\bar{\alpha}_{\pm},\rho_{\pm}^{\var,\tau}-\bar{\rho}_{\pm})$ (see Remarks \ref{rmkH4r}-\ref{rmkH4r2}), and thus for $G^\tau_3, G^\tau_0$ in System \eqref{L-E}. In addition, we can not find a  a rescaling  to reduce the proof to the case $\tau=1$ and then recover the corresponding uniform estimates with respect to $\tau$ thanks to the  homogeneity of the Besov norms as in \cite{c0,c1}. To overcome these new difficulties,  we will  keep track of the dependence of $\var, \tau$ and perform elaborate energy estimates with mixed $L^1$-time and $L^2$-time type dissipation. More precisely, in the low-frequency region, we introduce a purely damped mode (effective flux)
$$
u^{\var,\tau}+\frac{\tau}{\rho^{\var,\tau}}\nabla P^{\var,\tau}
$$
corresponding to Darcy's law \eqref{PM1}$_{2}$ in the low-frequency setting to partially diagonalize the system and capture maximal dissipative structures. In addition, we derive some uniform estimates at a lower regularity level compared to \cite{c0,c1,c2} (see Lemma \ref{lemmalow}). In the high-frequency setting, due to the lack of symmetry, we need to cancel  higher-order terms  so as not to lose derivatives. For that, the construction of a Lyapunov functional in the spirit of Beauchard and Zuazua as in \cite{bea1} with additional nonlinear weights allows us to capture the $L^1$-time dissipation properties in high frequencies (cf. Lemma \ref{lemmahigh}).
Moreover, we also establish the uniform $L^2$-in-time estimates at $\dot{B}^{\frac{d}{2}+1}$-regularity level to recover the necessary bounds of parameters (refer to Lemma \ref{lemmaL2}). Applying these ideas, we obtain uniform estimates in terms of the parameters $\var, \tau$ satisfying $0<\var\leq \tau<1$ for the linearized problem (see Proposition \ref{proplinear}), which is crucial for our later nonlinear analysis.

Let us finally sketch the proof of the justifications of the strong relaxation limits. In fact,  to obtain  convergence rates, we will not estimate the differences of solutions between systems directly. The reason shares similarities with  the proof of global uniform  well-posedness.  {Roughly speaking, since both pressure-relaxation limit and time-relaxation limit are singular limits, there are singular terms which are only uniformly bounded but not necessarily vanishing in the equations satisfied by the difference of solutions.}
To overcome these difficulties, we discover some {\em auxiliary unknowns} associated with the difference systems, which reveal better structures (cancellations), and then perform error estimates on them for each relaxation limit. More details are presented in Sections \ref{sectionrelaxationBNK} and \ref{sectionrelaxationKPM}.



\vspace{2mm}




\section{Functional framework and tools}\label{sectionbesov}

In this section, we recall the notations of the Littlewood-Paley decomposition and Besov spaces. The reader can refer to \cite{bahouri1}[Chapter 2] for a complete overview. Choose a smooth radial non-increasing function $\chi(\xi)$ with compact supported in $B(0,\frac{4}{3})$ and $\chi(\xi)=1$ in $B(0,\frac{3}{4})$ such that
$$
\varphi(\xi):=\chi\left( {\xi}/{2}\right)-\chi(\xi),\quad \sum_{j\in \mathbb{Z}}\varphi(2^{-j}\cdot)=1,\quad \text{{\rm{Supp}}}~ \varphi\subset \Big\{\xi\in \mathbb{R}^{d}~\Big{|}~\frac{3}{4}\leq |\xi|\leq \frac{8}{3}\Big\}.
$$
For any $j\in \mathbb{Z}$, the homogeneous dyadic blocks $\dot{\Delta}_{j}$ and the low-frequency cut-off operator $\dot{S}_{j}$ are defined by
$$
\dot{\Delta}_{j}u:=\mathcal{F}^{-1}(\varphi(2^{-j}\cdot )\mathcal{F}u),\quad\quad \dot{S}_{j}u:= \mathcal{F}^{-1}( \chi (2^{-j}\cdot) \mathcal{F} u).
$$
 From now on, we use the shorthand notation
$$
\dot{\Delta}_{j}u=u_{j}.
$$
Let $\mathcal{S}_{h}'$ be the set of tempered distributions on $\mathbb{R}^{d}$ such that every $u\in \mathcal{S}_{h}'$ satisfies $u\in \mathcal{S}'$ and $\lim_{j\rightarrow-\infty}\|\dot{S}_{j}u\|_{L^{\infty}}=0$. Then it follows that
\begin{equation}\nonumber
\begin{aligned}
&u=\sum_{j\in \mathbb{Z}}u_{j}\quad\text{in}~\mathcal{S}',\quad\quad \dot{S}_{j}u= \sum_{j'\leq j-1}u_{j'},\quad \forall u\in \mathcal{S}_{h}',\\
\end{aligned}
\end{equation}
With the help of these dyadic blocks, the homogeneous Besov space $\dot{B}^{s}$ for $s\in \mathbb{R}$ is defined by
$$
\dot{B}^{s}:=\Big\{u\in \mathcal{S}_{h}'~|~\|u\|_{\dot{B}^{s}}:=\sum_{j\in\mathbb{Z}}2^{js}\|u_{j}\|_{L^{2}}<\infty\Big\}.
$$
We denote the Chemin-Lerner type space $\widetilde{L}^{\varrho}(0,T;\dot{B}^{s})$ for  $s\in\mathbb{R}$ and $T>0$:
$$
\widetilde{L}^{\varrho}(0,T;\dot{B}^{s}):=\Big\{u\in L^{\varrho}(0,T;\mathcal{S}'_{h})~|~ \|u\|_{\widetilde{L}^{\varrho}_{T}(\dot{B}^{s})}:=\sum_{j\in\mathbb{Z}}2^{js}\|u_{j}\|_{L^{\varrho}_{T}(L^{2})}<\infty\Big\}.
$$
By the Minkowski inequality, it holds that
\begin{equation}\nonumber
\begin{aligned}
&\|u\|_{L^{\varrho}_{T}(\dot{B}^{s})}\leq \|u\|_{\widetilde{L}^{\varrho}_{T}(\dot{B}^{s})}\quad \varrho>1,\quad\quad \|u\|_{L^{1}_{T}(\dot{B}^{s})}= \|u\|_{\widetilde{L}^{1}_{T}(\dot{B}^{s})},
\end{aligned}
\end{equation}
where $\|\cdot\|_{L^{\varrho}_{T}(\dot{B}^{s} )}$ is the usual Lebesgue-Besov norm. 

In order to perform our analysis on the low and high frequencies regions, we set the threshold 
\begin{align}
J_{\tau}:=-[\log_{2}{\tau}]+k,\label{J}
\end{align}
for suitable negative integer $k$ (to be determined). 
Denote the following notations for $p\in[1,\infty]$ and $s\in\mathbb{R}$:
\begin{equation}\nonumber
\begin{aligned}
&\|u\|_{\dot{B}^{s}}^{\ell}:=\sum_{j\leq J_{\tau}}2^{js}\|u_{j}\|_{L^{2}},\quad\quad\quad~~ \|u\|_{\dot{B}^{s}}^{h}:=\sum_{j\geq J_{\tau}-1}2^{js}\|u_{j}\|_{L^{\varrho}_{T}(L^{2})},\\
&\|u\|_{\widetilde{L}^{\varrho}_{T}(\dot{B}^{s})}^{\ell}:=\sum_{j\leq J_{\tau}}2^{js}\|u_{j}\|_{L^{2}},\quad\quad \|u\|_{\widetilde{L}^{\varrho}_{T}(\dot{B}^{s})}^{h}:=\sum_{j\geq  J_{\tau}-1}2^{js}\|u_{j}\|_{L^{\varrho}_{T}(L^{2})}.
\end{aligned}
\end{equation}
For any $u\in\mathcal{S}'_{h}$, we also define the low-frequency part $u^{\ell}$ and the high-frequency part $u^{h}$ by
$$
u^{\ell}:=\sum_{j\leq J_{\tau}-1}u_{j},\quad\quad u^{h}:=u-u^{\ell}=\sum_{j\geq J_{\tau}}u_{j}.
$$
It is easy to check for any $s'>0$ that
\begin{equation}\label{lhl}
\left\{
\begin{aligned}
&\|u^{\ell}\|_{\dot{B}^{s}}\leq \|u\|_{\dot{B}^{s}}^{\ell}\leq 2^{J_{\tau}s'}\|u\|_{\dot{B}^{s-s'}}^{\ell}\leq  2^{s'} ( 2^{k }\tau^{-1})^{s'}\|u\|_{\dot{B}^{s-s'}}^{\ell},\\
&\|u^{h}\|_{\dot{B}^{s}}\leq \|u\|_{\dot{B}^{s}}^{h}\leq 2^{-(J_{\tau}-1)s'}\|u\|_{\dot{B}^{s+s'}}^{h}\leq 2^{s'}( 2^{-k }\tau)^{s'}\|u\|_{\dot{B}^{s+s'}}^{h}.
\end{aligned}
\right.
\end{equation}

\medbreak

Next, we state some properties of Besov spaces and related estimates which will be repeatedly used in the rest of paper. The reader can refer to \cite[Chapters 2-3]{bahouri1} for more details. Below, all the properties of Besov norms can be easily extended to the Chemin-Lerner norms.

The first lemma pertains to the so-called Bernstein's inequalities.
\begin{lemma}\label{lemma61}
Let $0<r<R$, $1\leq p\leq q\leq \infty$ and $k\in \mathbb{N}$. For any function $u\in L^p$ and $\lambda>0$, it holds that
\begin{equation}\nonumber
\left\{
\begin{aligned}
&{\rm{Supp}}~ \mathcal{F}(u) \subset \{\xi\in\mathbb{R}^{d}~| ~|\xi|\leq \lambda R\}\Rightarrow \|D^{k}u\|_{L^q}\lesssim\lambda^{k+d(\frac{1}{p}-\frac{1}{q})}\|u\|_{L^p},\\
&{\rm{Supp}}~ \mathcal{F}(u) \subset \{\xi\in\mathbb{R}^{d}~|~ \lambda r\leq |\xi|\leq \lambda R\}\Rightarrow \|D^{k}u\|_{L^{p}}\sim\lambda^{k}\|u\|_{L^{p}}.
\end{aligned}
\right.
\end{equation}
\end{lemma}

Due to the Bernstein inequalities, the Besov spaces have many useful properties:
\begin{lemma}\label{lemma22}
The following properties hold{\rm:}
\begin{itemize}
\item{} For any $s\in\mathbb{R}$ and $q\geq2$, we have the following continuous embeddings{\rm:}
\begin{equation}\nonumber
\begin{aligned}
 \dot{B}^{s}\hookrightarrow  \dot{H}^{s},\quad \quad \dot{B}^{\frac{d}{2}-\frac{d}{q}}\hookrightarrow  L^{q}.
\end{aligned}
\end{equation}
 \item{} $\dot{B}^{\frac{d}{2}}$ is continuously embedded in the set of continuous functions decaying to $0$ at infinity.
\item{}
For any $\sigma\in \mathbb{R}$, the operator $\Lambda^{\sigma}$ is an isomorphism from $\dot{B}^{s}$ to $\dot{B}^{s-\sigma}$.
\item{} Let $s_{1}\in\mathbb{R}$ and $ s_{2}\leq \frac{d}{2}$. Then
    the space $\dot{B}^{s_{1}}\cap \dot{B}^{s_{2}}$ is a Banach space and satisfies weak compact and Fatou properties: If $u_{k}$ is a uniformly bounded sequence of $\dot{B}^{s_{1}}\cap \dot{B}^{s_{2}}$, then an element $u$ of $\dot{B}^{s_{1}}\cap \dot{B}^{s_{2}}$ and a subsequence $u_{n_{k}}$ exist such that
    \begin{equation}\nonumber
    \begin{aligned}
\lim_{k\rightarrow\infty}u_{n_{k}}=u\quad\text{in}\quad\mathcal{S}'\quad\text{and}\quad\|u\|_{\dot{B}^{s_{1}}\cap \dot{B}^{s_{2}}}\lesssim \liminf_{n_{k}\rightarrow \infty} \|u_{n_{k}}\|_{\dot{B}^{s_{1}}\cap \dot{B}^{s_{2}}}.
    \end{aligned}
    \end{equation}
\end{itemize}
\end{lemma}

The following Morse-type product estimates in Besov spaces play a fundamental role in our analysis of nonlinear terms.
\begin{lemma}\label{lemma63}
The following statements hold:
\begin{itemize}
\item{} Let $s>0$. Then $\dot{B}^{s}\cap L^{\infty}$ is a algebra and
    \begin{equation}\label{uv1}
\begin{aligned}
\|uv\|_{\dot{B}^{s}}\lesssim \|u\|_{L^{\infty}}\|v\|_{\dot{B}^{s} }+ \|v\|_{L^{\infty}}\|u\|_{\dot{B}^{s} }.
\end{aligned}
\end{equation}
\item{}  Let $s_{1}, s_{2}$ satisfy $s_{1}, s_{2}\leq \frac{d}{2}$ and $s_{1}+s_{2}>0$. Then there holds
\begin{equation}\label{uv2}
\begin{aligned}
&\|uv\|_{\dot{B}^{s_{1}+s_{2}-\frac{d}{2}}}\lesssim \|u\|_{\dot{B}^{s_{1}}
}\|v\|_{\dot{B}^{s_{2}}}.
\end{aligned}
\end{equation}
\end{itemize}
\end{lemma}

The following commutator estimates will used to control some nonlinearities in high frequencies.

\begin{lemma}\label{lemmacom}
Let $p\in[1,\infty]$ and $-\frac{d}{2}-1< s\leq \frac{d}{2}+1$. Then it holds that
\begin{align}
&\sum_{j\in\mathbb{Z}}2^{js}\|[v,\dot{\Delta}_{j}]\partial_{i}u\|_{L^{2}}\lesssim\|\nabla v\|_{\dot{B}^{\frac{d}{2}}}\|u\|_{\dot{B}^{s}},\quad i=1,...,d,\label{commutator}
\end{align}
for the commutator $[A,B]:=AB-BA$.
\end{lemma}

We prove the following lemma about the continuity for composition of multi-component functions. It should be noted that \eqref{F11} will be used to deal with the two-dimensional case in $\dot{B}^{0}$.
\begin{lemma}\label{lemma64}
Let $m\in \mathbb{N}$, $s>0$, and $G\in C^{\infty}(\mathbb{R}^{m})$ satisfy $G(0,...,0)=0$. Then for any $f_{i}\in\dot{B}^{s}\cap L^{\infty}$ $(i=1,...,m)$, there exists a constant $C_{f}>0$ depending on $\sum_{i=1}^{m}\|f_{i}\|_{L^{\infty}}$, $F$, $s$, $m$ and $d$ such that
\begin{equation}
\begin{aligned}
\|G(f_{1},...,f_{m})\|_{\dot{B}^{s} }\leq C_{f}\sum_{i=1}^{m}\|f_{i}\|_{\dot{B}^{s} }.\label{F1}
\end{aligned}
\end{equation}
In the case $s>-\frac{d}{2}$ and $f_{i}\in\dot{B}^{s}\cap \dot{B}^{\frac{d}{2}}$, it holds that
\begin{equation}
\begin{aligned}
\|G(f_{1},...,f_{m})\|_{\dot{B}^{s} }\leq C_{f}\big(1+\sum_{i=1}^{m}\|f_{i}\|_{\dot{B}^{\frac{d}{2}}} \big)\sum_{i=1}^{m}\|f_{i}\|_{\dot{B}^{s} }.\label{F11}
\end{aligned}
\end{equation}
Furthermore, for any $f^{1}_{i}, f^{2}_{i}\in\dot{B}^{s}\cap \dot{B}^{\frac{d}{2}}$, we have
\begin{equation}\label{F3}
\begin{aligned}
&\|G(f^{1}_{1},...,f^{1}_{m})-G(f^{1}_{1},...,f^{1}_{m})\|_{\dot{B}^{s} }\leq C^{*}_{f}\big(1+\sum_{i=1}^{m}\|f_{i}\|_{\dot{B}^{s}\cap\dot{B}^{\frac{d}{2}}} \big)\sum_{i=1}^{m}\|f_{i}^{1}-f_{i}^{2}\|_{\dot{B}^{s}\cap\dot{B}^{\frac{d}{2}}}.
\end{aligned}
\end{equation}
Here $C^{*}_{f}>0$ depends on $\sum_{i=1}^{m}\|(f^{1}_{i},f^{2}_{i})\|_{L^{\infty}}$, $F$, $s$, $m$ and $d$.

\end{lemma}
\begin{proof} 
The estimate \eqref{F1} can be found in \cite{runst1}[Pages 387-388]. Then for $-\frac{d}{2}<s\leq \frac{d}{2}$, Taylor's formula implies that there exists a sequence $\widetilde{H}_{i}(f_{1},...,f_{m})$ satisfying $\widetilde{H}_{i}(0,...,0)=0$ and 
\begin{align}
&G(f_{1},...,f_{m})=\sum_{i=1}^{m}\big(\partial_{f_{i}}G(0,...,0)+\widetilde{H}_{i}(f_{1},...,f_{m})  \big) f_{i}.\nonumber
\end{align}
This together with the product law \eqref{uv2} and the estimate \eqref{F1} yields \eqref{F11}.

Moreover, we note that
\begin{equation}\nonumber
\begin{aligned}
G(f^{1}_{1},...,f^{1}_{m})-G(&f^{1}_{1},...,f^{1}_{m})=\int_{0}^{1}\frac{d}{ds}G\big(f_1^1+s(f_1^1-f_2^1),...,f_m^1+s(f_m^1-f_m^1)\big)ds\\
&\quad=\sum_{i=1}^{m} (f_i^1-f_i^2)\partial_{f_{i}}G(0,...,0)\\
&\quad\quad+\sum_{i=1}^{m} (f_i^1-f_i^2)\int_{0}^{1}\Big( \partial_{f_{i}}G\big(f_1^1+s(f_1^1-f_2^1),...,f_m^1+s(f_m^1-f_m^1)\big)-\partial_{f_{i}}G(0,...,0)\Big)ds.
\end{aligned}
\end{equation}
Therefore, applying \eqref{uv1}, \eqref{F1} and the embedding $\dot{B}^{\frac{d}{2}}\hookrightarrow L^{\infty}$, we get \eqref{F3}.
\end{proof}


Finally, we give optimal regularity estimates of some linear equations. We mention that such estimates on usual Besov norms can be easily extended to the norms restricted in low or high frequencies. We recall the estimates of  the heat equation as follows (cf. \cite{bahouri1}[Page 157] for example).

\begin{lemma}\label{maximalheat} Let $\mu_{*}>0$, $s\in\mathbb{R}$ and $1\leq p\leq\infty$. For given time $T>0$, assume $u_{0}\in\dot{B}^{s}$ and $f\in \widetilde{L}^{p}(0,T;\dot{B}^{s-2+\frac{2}{p}})$. If $u$ solves the problem  
\begin{equation}\nonumber
\left\{
\begin{aligned}
&\partial_t u- \mu_{*}\Delta u =f,\quad\quad x\in\mathbb{R}^{d},\quad t>0,\\
&u(0, x)=u_0(x),\quad\quad~  x\in\mathbb{R}^{d},
\end{aligned}
\right.
\end{equation}
then the following estimate is fulfilled{\rm:}
\begin{equation}\nonumber
\|u\|_{\widetilde{L}^{\infty}_{t}(\dot{B}^{s})}+\mu_{*}^{\frac{1}{p_{1}}}\|u\|_{\widetilde{L}^{p}_{t}(\dot{B}^{s+\frac{2}{p}})}\leq C\big(\|u_{0}\|_{\dot{B}^{s}}+\mu_{*}^{\frac{1}{p}-1}\|f\|_{\widetilde{L}^{p}_{t}(\dot{B}^{s-2+\frac{2}{p}}
)}\big),\quad t\in(0,T),
\end{equation}
where $C>0$ is a constant independent of $T$ and $\mu_{*}$.
\end{lemma}

We have the regularity estimates of the damped transport equation. Since it can be directly shown by the commutator estimates \eqref{commutator} and Gr\"onwall's inequality as in \cite{burtea1,danchin4},  we omit the proof for brevity.

\begin{lemma}\label{maximaldamped} Let $\lambda_{*}\geq 0$, $p=1$ or $\lambda_{*}>0$, $1\leq p\leq \infty$. For $-\frac{d}{2}<s\leq \frac{d}{2}+1$ and given time $T>0$, assume that $u_{0}\in\dot{B}^{s}$, $v\in L^1(0,T;\dot{B}^{\frac{d}{2}+1})$ and $f\in \widetilde{L}^{p}(0,T;\dot{B}^{s})$. If $u$ solves the problem  
\begin{equation}\nonumber
\left\{
\begin{aligned}
&\partial_t u+v\cdot\nabla u+ \lambda_{*}  u =f,\quad \quad x\in\mathbb{R}^{d},\quad t>0,\\
&u(0, x)=u_0(x),\quad\quad\quad\quad\quad ~ x\in\mathbb{R}^{d},
\end{aligned}
\right.
\end{equation}
then it holds that
\begin{equation}\nonumber
\|u\|_{\widetilde{L}^{\infty}_{t}(\dot{B}^{s})}+ \lambda_{*}^{\frac{1}{p}}\|u\|_{\widetilde{L}^{p}_{t}(\dot{B}^{s})}\leq C{\rm{exp}}\big(C\|v\|_{L^1_{t}(\dot{B}^{\frac{d}{2}+1})}\big)\big(\|u_{0}\|_{\dot{B}^{s}}+\lambda_{*}^{\frac{1}{p}-1}\|f\|_{\widetilde{L}^{p}_{t}(\dot{B}^{s})}\big),\quad t\in(0,T),
\end{equation}
where $C>0$ is a constant independent of $T$ and $\lambda_{*}$.
\end{lemma}

\section{Analysis of the linearized system}\label{section3}

We now consider the linearized problem associated to \eqref{re}, which reads
\begin{equation}
\left\{
\begin{aligned}
&\partial_{t}y+v\cdot\nabla y=0,\\
&\partial_{t}w+v\cdot\nabla w+ (h_{1}+H_1  )\div u+ (h_2+H_2) \dfrac{w}{\var}=S_{1},\\
&\partial_{t}r+v\cdot\nabla r+ (h_3+H_3  )\div u=S_{2},\\
&\partial_{t}u+v\cdot\nabla u+\frac{u}{\tau}+ (h_4+H_4  )\nabla  r+ (h_5+H_5 ) \nabla w  =S_{3},\\
&(y,w,r,u)(0,x)=(y_{0},w_{0},r_{0},u_{0})(x),
\end{aligned}
\right.\label{L}
\end{equation}
where $h_{i}$ ($i=1,...,5$) are given positive constants and $H_{i}=H_{i}(t,x)$ ($i=1,...,5$), $S_{i}=S_{i}(t,x)$ ($i=1,2,3$) are given smooth functions.

We first establish the following a-priori estimate for solutions of the linear problem \eqref{L} uniformly with respect to the parameters $\var, \tau$, which improves the result in \cite{burtea1} without the uniformity with respect to $\tau$. As explained before, the threshold $J_{\tau}$ between low and high frequencies given by \eqref{J} is the key to our analysis.


\begin{proposition}\label{proplinear} Let $d\geq2$, $0<\varepsilon\leq \tau <1$, $T>0$, and the threshold $J_{\tau}$ be given by \eqref{J}. Assume that $(w_{0},r_{0},u_{0})\in\dot{B}^{\frac{d}{2}-1}\cap\dot{B}^{\frac{d}{2}+1}$, $S_{1}, S_{2}, S_{3}\in L^1(0,T; \dot{B}^{\frac{d}{2}-1}\cap \dot{B}^{\frac{d}{2}+1})$, $H_{i}\in C([0,T];\dot{B}^{\frac{d}{2}-1}\cap\dot{B}^{\frac{d}{2}+1})$ and $\partial_{t}H_{i}\in L^1(0,T;\dot{B}^{\frac{d}{2}})$ for $i=1,2,...5$. There exists a constant $c>0$ independent of $T$, $\var$ and $\tau$ such that if
\begin{equation}\label{alinear}
\begin{aligned}
\mathcal{Z}(t):&=\sum_{i=1}^{5} \|H_{i}\|_{\widetilde{L}^{\infty}_{t}(\dot{B}^{\frac{d}{2}-1}\cap\dot{B}^{\frac{d}{2}+1})}\leq c,\quad\quad t\in (0,T),
\end{aligned}
\end{equation}
then for $t\in(0,T)$, the solution $(y,w,r,u)$ of the Cauchy problem \eqref{L} satisfies
\begin{equation}\label{elinear}
\begin{aligned}
\!\!\mathcal{X}(t)&:=\|(y,w,r,u)\|_{\widetilde{L}^{\infty}_{t}(\dot{B}^{\frac{d}{2}-1}\cap \dot{B}^{\frac{d}{2}+1})}+\|(\partial_{t}y,\partial_{t}w,\partial_{t}r,\partial_{t}u)\|_{L^1_{t}(\dot{B}^{\frac{d}{2}})} \\
&\!\!\quad\quad+\frac{1}{\var}\|w\|_{L^1_{t}(\dot{B}^{\frac{d}{2}-1}\cap\dot{B}^{\frac{d}{2}})}+\frac{1}{\sqrt{\var}}\|w\|_{\widetilde{L}^2_{t}(\dot{B}^{\frac{d}{2}-1}\cap\dot{B}^{\frac{d}{2}+1})}\\
&\!\!\quad\quad+\tau\|r\|_{L^1_{t}(\dot{B}^{\frac{d}{2}+1}\cap\dot{B}^{\frac{d}{2}+2})}^{\ell}+\|r\|_{L^1_{t}(\dot{B}^{\frac{d}{2}+1})}^{h}+\tau\|r\|_{L^1_{t}(\dot{B}^{\frac{d}{2}+1})}+\sqrt{\tau}\|r\|_{\widetilde{L}^2_{t}(\dot{B}^{\frac{d}{2}}\cap\dot{B}^{\frac{d}{2}+1})}\\
&\!\!\quad\quad+\|u\|_{L^1_{t}(\dot{B}^{\frac{d}{2}}\cap\dot{B}^{\frac{d}{2}+1})}+\frac{1}{\sqrt{\tau}}\|u\|_{\widetilde{L}^2_{t}(\dot{B}^{\frac{d}{2}-1}\cap\dot{B}^{\frac{d}{2}+1})}\\
&\!\!\quad\quad+ \frac{1}{\tau}\| u+\tau(h_4+H_4  )\nabla r \|_{L^1_{t}(\dot{B}^{\frac{d}{2}-1}\cap \dot{B}^{\frac{d}{2}})}\\
&\!\!\leq C_{0}{\rm{exp}}\Big(C_{0}\int_{0}^{t}\mathcal{V}(s)ds\Big) \Big(\|(y_{0},w_{0},r_{0},u_{0})\|_{\dot{B}^{\frac{d}{2}-1}\cap \dot{B}^{\frac{d}{2}+1}}+\|(S_{1}, S_{2}, S_{3})\|_{L^1_{t}( \dot{B}^{\frac{d}{2}-1} \cap \dot{B}^{\frac{d}{2}+1})} \Big),
\end{aligned}
\end{equation}
where $C_{0}>1$ is a universal constant, and $\mathcal{V}(t)$ is denoted by
\begin{equation}\label{V}
\begin{aligned}
&\mathcal{V}(t):=\|v(t)\|_{\dot{B}^{\frac{d}{2}}\cap\dot{B}^{\frac{d}{2}+1}}+\sum_{i=1}^{5}\|\partial_{t}H_{i}(t)\|_{\dot{B}^{\frac{d}{2}}}.
\end{aligned}
\end{equation}

\end{proposition}

\begin{proof}
First, we deal with the purely transport unknown $y$. By the regularity estimate in Lemma \ref{maximaldamped} for the transport equation $\eqref{L}_{1}$, it follows that
\begin{equation}\label{ey}
\begin{aligned}
\|y\|_{\widetilde{L}^{\infty}_{t}(\dot{B}^{\frac{d}{2}-1}\cap\dot{B}^{\frac{d}{2}+1})}&\lesssim {\rm{exp}}\Big(\int_{0}^{t}\|v(s)\|_{\dot{B}^{\frac{d}{2}+1}}ds\Big)\|y_{0}\|_{\dot{B}^{\frac{d}{2}-1}\cap \dot{B}^{\frac{d}{2}+1}}.
\end{aligned}
\end{equation}
And direct produce law \eqref{uv2} for the equation $\eqref{L}_{1}$ gives that
\begin{equation}\label{eyt}
\begin{aligned}
\|\partial_{t}y\|_{L^1_{t}(\dot{B}^{\frac{d}{2}})}\lesssim  \int_{0}^{t}\|v(s)\|_{\dot{B}^{\frac{d}{2}}}\|y(s)\|_{\dot{B}^{\frac{d}{2}+1}}ds.
\end{aligned}
\end{equation}
Similarly, we also get from \eqref{uv2} and $\eqref{L}_{3}$ that
\begin{equation}\label{rt}
\begin{aligned}
\|\partial_{t}r\|_{L^1_{t}(\dot{B}^{\frac{d}{2}})}&\lesssim \int_{0}^{t}\|v(s)\|_{\dot{B}^{\frac{d}{2}}} \|r(s)\|_{\dot{B}^{\frac{d}{2}+1}}ds+(1+\|H_{3}\|_{\widetilde{L}^{\infty}_{t}(\dot{B}^{\frac{d}{2}})})\|u\|_{L^1_{t}(\dot{B}^{\frac{d}{2}+1})}+\|S_{2}\|_{L^1_{t}(\dot{B}^{\frac{d}{2}})},
\end{aligned}
\end{equation}
and 
\begin{equation}\label{et}
\begin{aligned}
\|(\partial_{t}w,\partial_{t}u)\|_{L^1_{t}(\dot{B}^{\frac{d}{2}})}&\lesssim \int_{0}^{t}\|v(s)\|_{\dot{B}^{\frac{d}{2}}} \|(w,u)(s)\|_{\dot{B}^{\frac{d}{2}+1}}ds+\frac{1}{\tau}\|u+\tau (h_5+H_5  )\nabla r\|_{L^1_{t}(\dot{B}^{\frac{d}{2}})}\\
&\quad\quad+\Big(1+\sum_{i=1}^{5}\|H_{i}\|_{\widetilde{L}^{\infty}_{t}(\dot{B}^{\frac{d}{2}})}\Big)\left(\frac{1}{\var}\|w\|_{L^1_{t}(\dot{B}^{\frac{d}{2}})}+\|u\|_{L^1_{t}(\dot{B}^{\frac{d}{2}+1})}\right)+\|(S_{1},S_{3})\|_{L^1_{t}(\dot{B}^{\frac{d}{2}})}.
\end{aligned}
\end{equation}
The conclusion of the proof will follow from Lemmas \ref{lemmalow}-\ref{lemmaL2} given and proven in the next three subsections.
Indeed, combining \eqref{ey}-\eqref{et} and the uniform estimates of $(w,r,u)$ from Lemmas \ref{lemmalow}-\ref{lemmaL2} together and taking the constant $\eta>0$ suitable small in Lemma \ref{lemmaL2}, we obtain
\begin{equation}
\begin{aligned}
\mathcal{X}(t)&\lesssim \|(S_{1}, S_{2}, S_{3})\|_{L^1_{t}( \dot{B}^{\frac{d}{2}-1} \cap \dot{B}^{\frac{d}{2}+1})}+\big(\sqrt{\mathcal{Z}(t)}+\mathcal{Z}(t)\big)\mathcal{X}(t)+\int_{0}^{t}\mathcal{V}(s)\mathcal{X}(s)ds\nonumber\\
&\quad+ \|(y_{0},w_{0},r_{0},u_{0})\|_{\dot{B}^{\frac{d}{2}-1}\cap \dot{B}^{\frac{d}{2}+1}}.
\end{aligned}
\end{equation}
Then making use of the Gr\"onwall inequality and the smallness assumption \eqref{alinear} of $\mathcal{Z}(t),$ we obtain the uniform a-priori estimate \eqref{elinear}.
\end{proof}


\subsection{Low-frequency analysis}\label{subsectionlow}



Motivated by Darcy's law $\eqref{PM1}_{2}$, we introduce the following effective flux
\begin{align}
z:=u+\tau (h_4+H_4  )\nabla  r,\label{z}
\end{align}
which undergoes a purely damped effect in the low-frequency region $|\xi|\leq \frac{2^{k}}{\tau}$ and allows us to diagonalize the subsystem \eqref{L}$_{2}$-\eqref{L}$_{4}$ up to some higher-order terms that can be absorbed. Indeed, substituting \eqref{z} into \eqref{L}, we obtain
\begin{equation}
\left\{
\begin{aligned}
&\partial_{t}w+ \frac{h_2}{\var}w=L_1+R_1+S_{1},\\
&\partial_{t}r-h_{3}h_{4}\tau\Delta r=L_2+R_{2}+S_{2},\\
&\partial_{t}z+\frac{z}{\tau} =L_3+R_{3}+S_{3},\\
&(w,r,z)(0, x)=(w_{0},r_{0},z_{0})(x),
\end{aligned}
\right.\label{L-low1}
\end{equation}
where the higher-order linear terms $L_{i}$ ($i=1,2,3$) are denoted as
\begin{equation}\nonumber
\left\{
\begin{aligned}
&L_1:=h_{1} \big( h_{4}\tau\Delta r -\div z\big),\\
& L_2:=-h_{3}\div z,\\
&L_3:=h_{3}h_{4}\tau \nabla \big( h_{4}\tau  \Delta r-\div z\big)-h_{5}\nabla w,
\end{aligned}
\right.
\end{equation}
and the nonlinear terms $R_{i}$ ($i=1,2,3$) are defined by
\begin{equation}\label{R123}
\left\{
\begin{aligned}
&R_{1}:=-v\cdot\nabla w-H_{1}\div u+h_{1}\tau \div (H_4 \nabla r)- \dfrac{1}{\var}H_2 w,\\
&R_{2}:=-v\cdot\nabla r-H_{3}\div u+h_{3}\tau  \div (H_4 \nabla r),\\
&R_{3}:=-v\cdot\nabla u-H_{5}\nabla w-\tau\partial_t (H_4 \nabla r)+h_{4}\tau\nabla R_{2}.
\end{aligned}
\right.
\end{equation}

Now, to establish the $\dot{B}^{\frac{d}{2}-1}\cap\dot{B}^{\frac{d}{2}}$-estimates in low frequencies to the solutions of System \eqref{L} uniformly with respect to both $\var$ and $\tau$, we understand the equations in \eqref{L-low1} are decoupled. More precisely, we will treat the equations of $w$ and $z$ as damped equations and $r$ as a heat equation, respectively. This viewpoint plays a key role in the  proof of the following lemma.



\begin{lemma}\label{lemmalow} Let $T>0$, and the threshold $J_{\tau}$ be given by \eqref{J}. Then for $t\in(0,T)$, the solution  $(w,r,u)$ to the    linear problem \eqref{L}$_2$-\eqref{L}$_4$  satisfies
\begin{equation}\label{low}
\begin{aligned}
&\|(w,r,u)\|_{\widetilde{L}^{\infty}_{t}(\dot{B}^{\frac{d}{2}-1}\cap\dot{B}^{\frac{d}{2}})}^{\ell}+\frac{1}{\var}\|w\|_{L^1_{t}(\dot{B}^{\frac{d}{2}-1}\cap\dot{B}^{\frac{d}{2}})}^{\ell}+\frac{1}{\sqrt{\var}}\|w\|_{\widetilde{L}^2_{t}(\dot{B}^{\frac{d}{2}-1}\cap\dot{B}^{\frac{d}{2}})}^{\ell}\\
&\quad\quad+\tau\|r\|_{L^1_{t}(\dot{B}^{\frac{d}{2}+1}\cap\dot{B}^{\frac{d}{2}+2})}^{\ell}+\sqrt{\tau}\|r\|_{\widetilde{L}^2_{t}(\dot{B}^{\frac{d}{2}}\cap\dot{B}^{\frac{d}{2}+1})}^{\ell}\\
&\quad\quad+\|u\|_{L^1_{t}(\dot{B}^{\frac{d}{2}}\cap \dot{B}^{\frac{d}{2}+1})}^{\ell}+\frac{1}{\sqrt{\tau}}\|u\|_{\widetilde{L}^2_{t}(\dot{B}^{\frac{d}{2}-1}\cap \dot{B}^{\frac{d}{2}})}^{\ell}+\frac{1}{\tau}\|z\|_{L^1_{t}(\dot{B}^{\frac{d}{2}-1}\cap \dot{B}^{\frac{d}{2}})}^{\ell}\\
&\quad\lesssim \|(w_{0},r_{0},u_{0})\|_{\dot{B}^{\frac{d}{2}-1}\cap \dot{B}^{\frac{d}{2}}}+\|(S_{1},S_{2},S_{3})\|_{L^1_{t}(\dot{B}^{\frac{d}{2}-1}\cap\dot{B}^{\frac{d}{2}})}^{\ell}+\mathcal{Z}(t) \mathcal{X}(t)+\int_{0}^{t} \mathcal{V}(s)\mathcal{X}(s) ds ,
\end{aligned}
\end{equation}
where $\mathcal{Z}(t)$, $\mathcal{X}(t)$,  $\mathcal{V}(t)$ and $z$ are defined by \eqref{alinear}, \eqref{elinear}, \eqref{V} and \eqref{z}, respectively.
\end{lemma}

\begin{remark}
In {\rm\cite{burtea1}}, the authors obtained the low-frequency estimates by constructing a related Lyapunov functional.  However, that method does not lead to the desired estimates which uniform with respect to $\tau$. Moreover, it should be noted that the effective unknown $z$ given by \eqref{z} enables us to capture the heat-like behavior of the unknown $r$ in low frequencies directly, which is consistent with the  parabolic nature of the limiting porous media equations.
\end{remark}

\subsubsection{The $\dot{B}^{\frac{d}{2}}$-estimates}

We first perform $\dot{B}^{\frac{d}{2}}$-estimates in low frequencies for the heat equation \eqref{L-low1}$_{2}$.
It follows from the regularity estimate  in Lemma \ref{maximalheat} that
\begin{equation}\label{rd21}
\begin{aligned}
&\|r\|_{\widetilde{L}^{\infty}_{t}(\dot{B}^{\frac{d}{2}})}^{\ell}+\tau\|r\|_{L^1_{t}(\dot{B}^{\frac{d}{2}+2})}^{\ell}\lesssim \|r_{0}\|_{\dot{B}^{\frac{d}{2}}}^{\ell}+\|L_{2}\|_{L^1_{t}(\dot{B}^{\frac{d}{2}})}^{\ell}+\|R_{2}\|_{L^1_{t}(\dot{B}^{\frac{d}{2}})}^{\ell}+\|S_{2}\|_{L^1_{t}(\dot{B}^{\frac{d}{2}})}^{\ell}\\
&\quad\lesssim \|r_{0}\|_{\dot{B}^{\frac{d}{2}}}^{\ell}+2^{J_{\tau}}\|z\|_{L^1_{t}(\dot{B}^{\frac{d}{2}})}^{\ell}+\|R_{2}\|_{L^1_{t}(\dot{B}^{\frac{d}{2}})}^{\ell}+\|S_{2}\|_{L^1_{t}(\dot{B}^{\frac{d}{2}})}^{\ell}.
\end{aligned}
\end{equation}
Applying Lemma \ref{maximaldamped} to the damped equation \eqref{L-low1}$_{1}$, we get
\begin{equation}\label{wd21}
\begin{aligned}
&\|w\|_{\widetilde{L}^{\infty}_{t}(\dot{B}^{\frac{d}{2}})}^{\ell}+\frac{1}{\var}\|w\|_{L^1_{t}(\dot{B}^{\frac{d}{2}})}^{\ell}\\
&\quad\lesssim \|w_{0}\|_{\dot{B}^{\frac{d}{2}}}^{\ell}+\|L_{1}\|_{L^1_{t}(\dot{B}^{\frac{d}{2}})}^{\ell}+\|R_{1}\|_{L^1_{t}(\dot{B}^{\frac{d}{2}})}^{\ell}+\|S_{1}\|_{L^1_{t}(\dot{B}^{\frac{d}{2}})}^{\ell}\\
&\quad\lesssim \|w_{0}\|_{\dot{B}^{\frac{d}{2}}}^{\ell}+ \tau\|r\|_{L^1_{t}(\dot{B}^{\frac{d}{2}+2})}^{\ell}+2^{J_{\tau}}\|z\|_{L^1_{t}(\dot{B}^{\frac{d}{2}})}^{\ell}+\|R_{1}\|_{L^1_{t}(\dot{B}^{\frac{d}{2}})}^{\ell}+\|S_{1}\|_{L^1_{t}(\dot{B}^{\frac{d}{2}})}^{\ell}\\
&\quad\lesssim \|(w_{0},r_{0})\|_{\dot{B}^{\frac{d}{2}}}^{\ell}+2^{J_{\tau}}\|z\|_{L^1_{t}(\dot{B}^{\frac{d}{2}})}^{\ell}+\|(R_{1},R_{2})\|_{L^1_{t}(\dot{B}^{\frac{d}{2}})}^{\ell}+\|(S_{1},S_{2})\|_{L^1_{t}(\dot{B}^{\frac{d}{2}})}^{\ell},
\end{aligned}
\end{equation}
where we  used inequality \eqref{rd21} to control  terms involving $r$ in equation \eqref{L-low1}$_1$.

Similarly, by virtue of inequality \eqref{rd21} and Lemmas \ref{maximalheat}-\ref{maximaldamped}, we have for equation \eqref{L-low1}$_{3}$ that
\begin{equation}\label{zd2}
\begin{aligned}
&\|z\|_{\widetilde{L}^{\infty}_{t}(\dot{B}^{\frac{d}{2}})}^{\ell}+\frac{1}{\tau}\|z\|_{L^1_{t}(\dot{B}^{\frac{d}{2}})}^{\ell}\\
&\quad\lesssim \|z_{0}\|_{\dot{B}^{\frac{d}{2}}}^{\ell}+\|L_{3}\|_{L^1_{t}(\dot{B}^{\frac{d}{2}})}^{\ell}+\|(R_{3},S_{3})\|_{L^1_{t}(\dot{B}^{\frac{d}{2}})}^{\ell}\\
&\quad\lesssim \|z_{0}\|_{\dot{B}^{\frac{d}{2}}}^{\ell}+2^{J_{\tau}}\|w\|_{L^1_{t}(\dot{B}^{\frac{d}{2}})}^{\ell}+\tau^2 2^{J_{\tau}}\|r\|_{L^1_{t}(\dot{B}^{\frac{d}{2}+2})}^{\ell}+\tau2^{2J_{\tau}}\|z\|_{L^1_{t}(\dot{B}^{\frac{d}{2}+2})}^{\ell}+\|(R_{3},S_{3})\|_{L^1_{t}(\dot{B}^{\frac{d}{2}})}^{\ell}.
\end{aligned}
\end{equation}
 Since the threshold $J_{\tau}$ satisfies condition \eqref{J}, thus $\tau 2^{J_{\tau}}\sim 2^{k}<<1$ for   suitable negative integer $k$.  Due to the condiiton $\var\leq \tau$ so that $2^{J_{\tau}}\|w\|_{L^1_{t}(\dot{B}^{\frac{d}{2}})}^{\ell}\leq \frac{ 2^{k+1}}{\var}\|w\|_{L^1_{t}(\dot{B}^{\frac{d}{2}})}^{\ell}$, we have by the inequalities \eqref{rd21}-\eqref{zd2} that
\begin{equation}\label{wrzd21}
\begin{aligned}
&\|(w,r,z)\|_{\widetilde{L}^{\infty}_{t}(\dot{B}^{\frac{d}{2}})}^{\ell}+\tau\|r\|_{L^1_{t}(\dot{B}^{\frac{d}{2}+2})}^{\ell}+\frac{1}{\var}\|w\|_{L^1_{t}(\dot{B}^{\frac{d}{2}})}^{\ell}+\frac{1}{\tau}\|z\|_{L^1_{t}(\dot{B}^{\frac{d}{2}})}^{\ell}\\
&\quad\lesssim \|(w_{0},r_{0},z_{0})\|_{\dot{B}^{\frac{d}{2}}}^{\ell}+\|(R_{1},R_{2},R_{3})\|_{L^1_{t}(\dot{B}^{\frac{d}{2}})}^{\ell}+\|(S_{1},S_{2},S_{3})\|_{L^1_{t}(\dot{B}^{\frac{d}{2}})}^{\ell}.
\end{aligned}
\end{equation}
The terms on the right-hand side of \eqref{wrzd21} can be estimated as follows. First, one derives from inequality \eqref{lhl} and   product law \eqref{uv2} and the composition estimate \eqref{F1} that
\begin{equation}\label{wrzd211}
\begin{aligned}
\|z_{0}\|_{\dot{B}^{\frac{d}{2}}}^{\ell}&\lesssim \|(r_{0},u_{0})\|_{\dot{B}^{\frac{d}{2}}}^{\ell}+  \|H_{4}(0)\|_{\dot{B}^{\frac{d}{2}}} \|r_{0}\|_{\dot{B}^{\frac{d}{2}}}\lesssim \|(r_{0},u_{0})\|_{\dot{B}^{\frac{d}{2}}}.
\end{aligned}
\end{equation}
By the product law \eqref{uv2} again, we also get
\begin{equation}\label{wrzd212}
\left\{
\begin{aligned}
&\|v\cdot\nabla w\|_{L^1_{t}(\dot{B}^{\frac{d}{2}})}^{\ell}\lesssim  \int_{0}^{t} \|v(s)\|_{\dot{B}^{\frac{d}{2}}}  \|w(s)\|_{\dot{B}^{\frac{d}{2}+1}}ds,\\
&\|H_{1}\div u\|_{L^1_{t}(\dot{B}^{\frac{d}{2}})}^{\ell}\lesssim \|H_{1}\|_{\widetilde{L}^{\infty}_t(\dot{B}^{\frac{d}{2}})}\|u\|_{L^1_{t}(\dot{B}^{\frac{d}{2}+1})},\\
&\frac{1}{\var}\|H_{2}w\|_{L^1_{t}(\dot{B}^{\frac{d}{2}})}^{\ell}\lesssim \|H_{2}\|_{\widetilde{L}^{\infty}_{t}(\dot{B}^{\frac{d}{2}})}\frac{1}{\var}\|w\|_{L^1_{t}(\dot{B}^{\frac{d}{2}})}.
\end{aligned}
\right.
\end{equation}
According to \eqref{lhl} and \eqref{uv1}-\eqref{uv2}, the tricky nonlinear term $H_4 \nabla r$ in \eqref{R123} can be estimated as
\begin{equation}\label{wrzd213}
\begin{aligned}
&\tau\| \div(H_4 \nabla r)\|_{L^1_{t}(\dot{B}^{\frac{d}{2}})}^{\ell}\\
&\quad\lesssim \tau\|H_4 \nabla r^{\ell}\|_{L^1_{t}(\dot{B}^{\frac{d}{2}+1})}^{\ell}+\tau\|H_4 \nabla r^{h}\|_{L^1_{t}(\dot{B}^{\frac{d}{2}+1})}^{\ell}\\
&\quad\lesssim \tau\|H_4 \nabla r^{\ell}\|_{L^1_{t}(\dot{B}^{\frac{d}{2}+1})}^{\ell}+\|H_4 \nabla r^{h}\|_{L^1_{t}(\dot{B}^{\frac{d}{2}})}^{\ell}\\
&\quad\lesssim \|H_{4}\|_{\widetilde{L}^{\infty}_{t}(\dot{B}^{\frac{d}{2}})}\tau\|r\|_{L^1_{t}(\dot{B}^{\frac{d}{2}+2})}^{\ell}+\|H_{4}\|_{\widetilde{L}^{\infty}_{t}(\dot{B}^{\frac{d}{2}+1})} \tau\|r\|_{L^1_{t}(\dot{B}^{\frac{d}{2}+1})}^{\ell}+\|H_{4}\|_{\widetilde{L}^{\infty}_{t}(\dot{B}^{\frac{d}{2}})}\|r\|_{L^1_{t}(\dot{B}^{\frac{d}{2}+1})}^{h}.
\end{aligned}
\end{equation}

\begin{remark}\label{rmkH4r}
The above estimate \eqref{wrzd213} for $H_4 \nabla r$ arising from two pressures implies that one needs uniform $\dot{B}^{\frac d2-1}$-estimates for low frequencies. Indeed, as $H_4$ does not have the either $L^1$-in-time or $L^2$-in-time integrability property, the product law  \eqref{uv1} in  $\dot{B}^{\frac d2+1}$ indicates us to discover the control of $\tau\|r\|_{L^1_{t}(\dot{B}^{\frac{d}{2}+1})}^{\ell},$ which can not be obtained from the $\dot{B}^{\frac d2}$-estimates in this section.
\end{remark}
\begin{remark}\label{rmkH4r2}
It is also one of the reasons why we need to perform the $\dot{B}^{\frac d2+1}$-estimates in the both low and high frequencies in the later Section \ref{sectiond21}. Indeed, in the low-frequency setting, the uniform $\widetilde{L}^\infty_t(\dot{B}^{\frac d2})$-norm is not enough to produce the uniform $\widetilde{L}^\infty_t(\dot{B}^{\frac d2+1})$-estimates required in \eqref{wrzd213} due to the inclusion \eqref{lhl}.
\end{remark}

Now, one derives from inequalities \eqref{wrzd211}-\eqref{wrzd213} that
\begin{equation}\label{R1d2}
\begin{aligned}
\|R_{1}\|_{L^1_{t}(\dot{B}^{\frac{d}{2}})}^{\ell}&\lesssim \|v\cdot\nabla w\|_{L^1_{t}(\dot{B}^{\frac{d}{2}})}^{\ell}+\|H_{1}\div u\|_{L^1_{t}(\dot{B}^{\frac{d}{2}})}^{\ell}\\
&\quad+\tau\| H_4\nabla r\|_{L^1_{t}(\dot{B}^{\frac{d}{2}+1})}^{\ell}+\frac{1}{\var}\|H_{2}w\|_{L^1_{t}(\dot{B}^{\frac{d}{2}})}^{\ell}\\
&\lesssim \mathcal{Z}(t)\mathcal{X}(t)+\int_{0}^{t}\mathcal{V}(s)\mathcal{X}(s)ds.
\end{aligned}
\end{equation}
Similarly, we have
\begin{equation}\label{R2d2}
\begin{aligned}
\|R_{2}\|_{L^1_{t}(\dot{B}^{\frac{d}{2}})}^{\ell}&\lesssim \|v\cdot\nabla r\|_{L^1_{t}(\dot{B}^{\frac{d}{2}})}^{\ell}+\|H_{3}\div u\|_{L^1_{t}(\dot{B}^{\frac{d}{2}})}^{\ell}+\tau\| H_4 \nabla r\|_{L^1_{t}(\dot{B}^{\frac{d}{2}+1})}^{\ell}\\
&\lesssim \mathcal{Z}(t)\mathcal{X}(t)+\int_{0}^{t}\mathcal{V}(s)\mathcal{X}(s)ds.
\end{aligned}
\end{equation}
To estimate $R_{3}$, we notice that \eqref{lhl} together with \eqref{uv2} implies
\begin{equation}\label{r3d2d1}
\begin{aligned}
\tau\|\partial_t( H_4\nabla r)\|_{L^1_{t}(\dot{B}^{\frac{d}{2}})}^{\ell}&\lesssim \|\partial_t( H_4\nabla r)\|_{L^1_{t}(\dot{B}^{\frac{d}{2}-1})}^{\ell}\\
&\lesssim \int_{0}^{t}\|\partial_{t}H_{4}(s)\|_{\dot{B}^{\frac{d}{2}}} \|r(s)\|_{\dot{B}^{\frac{d}{2}}}ds+\|H_{4}\|_{\widetilde{L}^{\infty}_{t}(\dot{B}^{\frac{d}{2}})} \|\partial_{t}r\|_{L^1_{t}(\dot{B}^{\frac{d}{2}})},
\end{aligned}
\end{equation}
and
\begin{equation}\nonumber
\begin{aligned}
\|H_{5}\nabla w\|_{L^1_{t}(\dot{B}^{\frac{d}{2}})}^{\ell}\lesssim \frac{1}{\tau} \|H_{5}\nabla w\|_{L^1_{t}(\dot{B}^{\frac{d}{2}-1})}^{\ell}\lesssim \|H_{5}\|_{\widetilde{L}^{\infty}_{t}(\dot{B}^{\frac{d}{2}} )}\frac{1}{\var}\|w\|_{L^1_{t}(\dot{B}^{\frac{d}{2}})},
\end{aligned}
\end{equation}
where we   used the assumption $\var\leq \tau$. Thus, it holds that
\begin{equation}\label{R3d2}
\begin{aligned}
\|R_{3}\|_{L^1_{t}(\dot{B}^{\frac{d}{2}})}^{\ell}&\lesssim\int_{0}^{t}\|v(s)\|_{\dot{B}^{\frac{d}{2}}}\|u(s)\|_{\dot{B}^{\frac{d}{2}+1}}ds+\tau\| H_4 \nabla r\|_{L^1_{t}(\dot{B}^{\frac{d}{2}+1})}^{\ell}+\|H_{5}\nabla w\|_{L^1_{t}(\dot{B}^{\frac{d}{2}})}^{\ell}\\
&\quad+\tau\|\partial_t( H_4\nabla r) \|_{L^1_{t}(\dot{B}^{\frac{d}{2}-1})}^{\ell}+\|R_{2}\|_{L^1_{t}(\dot{B}^{\frac{d}{2}-1})}^{\ell}\\
&\lesssim  \mathcal{Z}(t) \mathcal{X}(t)+\int_{0}^{t}\mathcal{V}(s)\mathcal{X}(s)ds.
\end{aligned}
\end{equation}
We substitute inequalities \eqref{wrzd211}, \eqref{R1d2}-\eqref{R2d2} and \eqref{R3d2} into inequality \eqref{wrzd21} and use standard interpolation to get
\begin{equation}\label{low11}
\begin{aligned}
&\|(w,r,z)\|_{\widetilde{L}^{\infty}_{t}(\dot{B}^{\frac{d}{2}})}^{\ell}+\tau\|r\|_{L^1_{t}(\dot{B}^{\frac{d}{2}+2})}^{\ell}+\sqrt{\tau}\|r\|_{\widetilde{L}^2_{t}(\dot{B}^{\frac{d}{2}+1})}^{\ell}\\
&\quad\quad+\frac{1}{\var}\|w\|_{L^1_{t}(\dot{B}^{\frac{d}{2}})}^{\ell}+\frac{1}{\sqrt{\var}}\|w\|_{\widetilde{L}^2_{t}(\dot{B}^{\frac{d}{2}})}^{\ell}+\frac{1}{\tau}\|z\|_{L^1_{t}(\dot{B}^{\frac{d}{2}})}^{\ell}\\
&\quad\lesssim \|(w_{0},r_{0},u_{0})\|_{\dot{B}^{\frac{d}{2}}}+\|(S_{1},S_{2},S_{3})\|_{L^1_{t}(\dot{B}^{\frac{d}{2}})}^{\ell}+\mathcal{Z}(t) \mathcal{X}(t)+\int_{0}^{t}\mathcal{V}(s)\mathcal{X}(s)ds.
\end{aligned}
\end{equation}
Thence, we rewrite the form \eqref{z} and use inequalities \eqref{lhl} and \eqref{wrzd213} to obtain the $L^{1}_{t}(\dot{B}^{\frac{d}{2}})$-estimate of $u$ as follows:
\begin{equation}\nonumber
\begin{aligned}
\|u\|_{L^1_{t}(\dot{B}^{\frac{d}{2}+1})}^{\ell}&\lesssim \|z\|_{L^1_{t}(\dot{B}^{\frac{d}{2}+1})}^{\ell}+\tau\|\nabla r\|_{L^1_{t}(\dot{B}^{\frac{d}{2}+1})}^{\ell}+\tau\|H_{4} \nabla r\|_{L^1_{t}(\dot{B}^{\frac{d}{2}+1})}^{\ell} \\
&\lesssim\frac{1}{\tau}\|z\|_{L^1_{t}(\dot{B}^{\frac{d}{2}})}^{\ell}+\tau\|r\|_{L^1_{t}(\dot{B}^{\frac{d}{2}+2})}^{\ell}+\|H_{4}\|_{\widetilde{L}^{\infty}_{t}(\dot{B}^{\frac{d}{2}})}\tau\|r\|_{L^1_{t}(\dot{B}^{\frac{d}{2}+2})}^{\ell} \\
&\quad+\|H_{4}\|_{\widetilde{L}^{\infty}_{t}(\dot{B}^{\frac{d}{2}+1})} \tau\|r\|_{L^1_{t}(\dot{B}^{\frac{d}{2}+1})}^{\ell}+\|H_{4}\|_{\widetilde{L}^{\infty}_{t}(\dot{B}^{\frac{d}{2}})}\|r\|_{L^1_{t}(\dot{B}^{\frac{d}{2}+1})}^{h}.
\end{aligned}
\end{equation}
Similarly, we have
\begin{equation}\nonumber
\begin{aligned}
\|u\|_{\widetilde{L}^{\infty}_{t}(\dot{B}^{\frac{d}{2}})}^{\ell}&\lesssim \|(z,r)\|_{\widetilde{L}^{\infty}_{t}(\dot{B}^{\frac{d}{2}})}^{\ell}+\|H_{4}\|_{\widetilde{L}^{\infty}_{t}(\dot{B}^{\frac{d}{2}})}\|r\|_{\widetilde{L}^{\infty}_{t}(\dot{B}^{\frac{d}{2}})},
\end{aligned}
\end{equation}
and
\begin{equation}\nonumber
\begin{aligned}
\frac{1}{\sqrt{\tau}}\|u\|_{\widetilde{L}^2_{t}(\dot{B}^{\frac{d}{2}})}^{\ell}&\lesssim \frac{1}{\sqrt{\tau}}\|z\|_{\widetilde{L}^2_{t}(\dot{B}^{\frac{d}{2}})}^{\ell}+\sqrt{\tau}\|r\|_{\widetilde{L}^2_{t}(\dot{B}^{\frac{d}{2}+1})}^{\ell}+\|H_{4}\|_{\widetilde{L}^{\infty}_{t}(\dot{B}^{\frac{d}{2}})}\sqrt{\tau}\|r\|_{\widetilde{L}^2_{t}(\dot{B}^{\frac{d}{2}+1})}.
\end{aligned}
\end{equation}
We thus obtain from inequality \eqref{low11} that
\begin{equation}\label{low14}
    \begin{aligned}
   &\|u\|_{\widetilde{L}^{\infty}_{t}(\dot{B}^{\frac{d}{2}})}^{\ell}+ \frac{1}{\sqrt{\tau}}\|u\|_{\widetilde{L}^2_{t}(\dot{B}^{\frac{d}{2}})}^{\ell}+\|u\|_{L^1_{t}(\dot{B}^{\frac{d}{2}+1})}^{\ell}\\
   &\quad \lesssim \|(w_{0},r_{0},u_{0})\|_{\dot{B}^{\frac{d}{2}}}+\|(S_{1},S_{2},S_{3})\|_{L^1_{t}(\dot{B}^{\frac{d}{2}})}^{\ell}+\mathcal{Z}(t) \mathcal{X}(t)+\int_{0}^{t}\mathcal{V}(s)\mathcal{X}(s)ds.
    \end{aligned}
\end{equation}

\subsubsection{The $\dot{B}^{\frac{d}{2}-1}$-estimates}


We perform the $\dot{B}^{\frac{d}{2}-1}$-estimates so as to control $\tau\|r\|_{L^1_{t}(\dot{B}^{\frac{d}{2}+1})}^{\ell}$, as explained in Remark \ref{rmkH4r}. Arguing similarly as for inequalities  \eqref{rd21}-\eqref{wrzd21}, we have
\begin{equation}\label{wrzd2-11}
\begin{aligned}
&\|(w,r,z)\|_{\widetilde{L}^{\infty}_{t}(\dot{B}^{\frac{d}{2}-1})}^{\ell}+\tau\|r\|_{L^1_{t}(\dot{B}^{\frac{d}{2}+1})}^{\ell}+\frac{1}{\var}\|w\|_{L^1_{t}(\dot{B}^{\frac{d}{2}-1})}^{\ell}+\frac{1}{\tau}\|z\|_{L^1_{t}(\dot{B}^{\frac{d}{2}-1})}^{\ell}\\
&\quad\lesssim\|(w_{0},r_{0},u_{0})\|_{\dot{B}^{\frac{d}{2}-1}}+\|(R_{1},R_{2},R_{3})\|_{L^1_{t}(\dot{B}^{\frac{d}{2}-1})}^{\ell}+\|(S_{1},S_{2},S_{3})\|_{L^1_{t}(\dot{B}^{\frac{d}{2}-1})}^{\ell}.
\end{aligned}
\end{equation}
Direct calculations give
\begin{equation}\label{R1R2d2-1}
\begin{aligned}
&\|(R_{1},R_{2})\|_{L^1_{t}(\dot{B}^{\frac{d}{2}-1})}^{\ell}\\
&\quad\lesssim \int_{0}^{t}\|v(s)\|_{\dot{B}^{\frac{d}{2}}}  \|(w,r)(s)\|_{\dot{B}^{\frac{d}{2}}}ds+\|(H_{1},H_{3})\|_{\widetilde{L}^{\infty}_t(\dot{B}^{\frac{d}{2}})}\|u\|_{L^1_{t}(\dot{B}^{\frac{d}{2}})}\\
&\quad\quad+ \|H_4\|_{\widetilde{L}^{\infty}_t(\dot{B}^{\frac{d}{2}})} \tau\|r\|_{L^1_{t}(\dot{B}^{\frac{d}{2}+1})}+\|H_{2}\|_{\widetilde{L}^{\infty}_{t}(\dot{B}^{\frac{d}{2}})}\frac{1}{\var}\|w\|_{L^1_{t}(\dot{B}^{\frac{d}{2}-1})}\\
&\quad\lesssim  \mathcal{Z}(t) \mathcal{X}(t)+\int_{0}^{t}\mathcal{V}(s)\mathcal{X}(s)ds.
\end{aligned}
\end{equation}
By inequalities \eqref{r3d2d1}, \eqref{R1R2d2-1} and product law \eqref{uv2} for $d\geq2$, the term $R_{3}$ can be bounded by
\begin{equation}\label{R3d2-1}
\begin{aligned}
\|R_{3}\|_{L^1_{t}(\dot{B}^{\frac{d}{2}-1})}^{\ell}&\lesssim \int_{0}^{t}\|v(s)\|_{\dot{B}^{\frac{d}{2}}}  \|u(s)\|_{\dot{B}^{\frac{d}{2}}}ds+\|H_{5}\|_{\widetilde{L}^{\infty}_{t}(\dot{B}^{\frac{d}{2}} )}\frac{1}{\var}\| w\|_{L^1_{t}(\dot{B}^{\frac{d}{2}})}\\
&\quad\quad+\tau\|\partial_t( H_4\nabla r)\|_{L^1_{t}(\dot{B}^{\frac{d}{2}-1})}^{\ell}+\|R_{2}\|_{L^1_{t}(\dot{B}^{\frac{d}{2}-1})}^{\ell}\\
&\lesssim \mathcal{Z}(t) \mathcal{X}(t)+\int_{0}^{t}\mathcal{V}(s)\mathcal{X}(s)ds.
\end{aligned}
\end{equation}
Inserting \eqref{R1R2d2-1} and \eqref{R3d2-1} into \eqref{wrzd2-11} and taking advantage of interpolation, we obtain
\begin{equation}\label{low21}
\begin{aligned}
&\|(w,r,z)\|_{\widetilde{L}^{\infty}_{t}(\dot{B}^{\frac{d}{2}-1})}^{\ell}+\tau\|r\|_{L^1_{t}(\dot{B}^{\frac{d}{2}+1})}^{\ell}+\sqrt{\tau}\|r\|_{\widetilde{L}^2_{t}(\dot{B}^{\frac{d}{2}})}^{\ell}\\
&\quad\quad+\frac{1}{\var}\|w\|_{L^1_{t}(\dot{B}^{\frac{d}{2}-1})}^{\ell}+\frac{1}{\sqrt{\var}}\|w\|_{\widetilde{L}^2_{t}(\dot{B}^{\frac{d}{2}-1})}^{\ell}+\frac{1}{\tau}\|z\|_{L^1_{t}(\dot{B}^{\frac{d}{2}-1})}^{\ell}\\
&\quad\lesssim \|(w_{0},r_{0},u_{0})\|_{\dot{B}^{\frac{d}{2}-1}}+\|(S_{1}, S_{2}, S_{3})\|_{L^1_{t}( \dot{B}^{\frac{d}{2}-1})}^{\ell}\\
&\quad\quad+\mathcal{Z}(t) \mathcal{X}(t)+\int_{0}^{t}\mathcal{V}(s)\mathcal{X}(s)ds.
\end{aligned}
\end{equation}
This together with inequality \eqref{lhl} and the fact that $u=z-\tau(h_4+H_4)\nabla r$ leads to
\begin{equation}\nonumber
\begin{aligned}
\|u\|_{\widetilde{L}^{\infty}_{t}(\dot{B}^{\frac{d}{2}-1})}^{\ell}&\lesssim \|(z,r)\|_{\widetilde{L}^{\infty}_{t}(\dot{B}^{\frac{d}{2}-1})}^{\ell}+\|H_{4}\|_{\widetilde{L}^{\infty}_{t}(\dot{B}^{\frac{d}{2}})}\|r\|_{\widetilde{L}^{\infty}_{t}(\dot{B}^{\frac{d}{2}-1})}.
\end{aligned}
\end{equation}
Similarly, one gets
\begin{equation}\nonumber
\begin{aligned}
\frac{1}{\sqrt{\tau}}\|u\|_{\widetilde{L}^2_{t}(\dot{B}^{\frac{d}{2}-1})}^{\ell}&\lesssim \frac{1}{\sqrt{\tau}}\|z\|_{\widetilde{L}^2_{t}(\dot{B}^{\frac{d}{2}-1})}^{\ell}+\sqrt{\tau}\|r\|_{\widetilde{L}^2_{t}(\dot{B}^{\frac{d}{2}})}^{\ell}+\|H_{4}\|_{\widetilde{L}^{\infty}_{t}(\dot{B}^{\frac{d}{2}})}\sqrt{\tau}\|r\|_{\widetilde{L}^2_{t}(\dot{B}^{\frac{d}{2}})},
\end{aligned}
\end{equation}
and
\begin{equation}\nonumber
\begin{aligned}
\|u\|_{L^1_{t}(\dot{B}^{\frac{d}{2}})}^{\ell}&\lesssim \frac{1}{\tau}\|z\|_{L^1_{t}(\dot{B}^{\frac{d}{2}-1})}^{\ell}+\tau\|r\|_{L^1_{t}(\dot{B}^{\frac{d}{2}+1})}^{\ell}+\|H_{4}\|_{\widetilde{L}^{\infty}_{t}(\dot{B}^{\frac{d}{2}})}\tau\|r\|_{L^1_{t}(\dot{B}^{\frac{d}{2}+1})}.
\end{aligned}
\end{equation}
Combining the above three estimates, we are led to
\begin{equation}\label{low24}
\begin{aligned}
\|u\|_{\widetilde{L}^{\infty}_{t}(\dot{B}^{\frac{d}{2}-1})}^{\ell}+\frac{1}{\sqrt{\tau}}\|u\|_{\widetilde{L}^2_{t}(\dot{B}^{\frac{d}{2}-1})}^{\ell}+\|u\|_{L^1_{t}(\dot{B}^{\frac{d}{2}})}^{\ell}&\lesssim \|(w_{0},r_{0},u_{0})\|_{\dot{B}^{\frac{d}{2}-1}}+\|(S_{1},S_{2},S_{3})\|_{L^1_{t}(\dot{B}^{\frac{d}{2}-1})}^{\ell}\\
&\quad\quad+\mathcal{Z}(t) \mathcal{X}(t)+\int_{0}^{t}\mathcal{V}(s)\mathcal{X}(s)ds.
\end{aligned}
\end{equation}
Putting the estimates \eqref{low},  \eqref{low11}, \eqref{low14} and \eqref{low21}, \eqref{low24} together, we complete the proof of Lemma \ref{lemmalow}.

\subsection{High-frequency analysis}\label{subsectionhigh}

In this section, we establish some uniform high-frequency estimates of solutions to the linear problem \eqref{L} in terms of the Lyapunov functional. More precisely, we establish the $\widetilde{L}^{\infty}_{t}(\dot{B}^{\frac{d}{2}-1}\cap \dot{B}^{\frac{d}{2}})$-estimates, and furthermore obtain the control of higher-order $L^1_{t}(\dot{B}^{\frac{d}{2}}\cap \dot{B}^{\frac{d}{2}+1})$-norms.



\begin{lemma}\label{lemmahigh} 
Let $T>0$, and the threshold $J_{\tau}$ be given by \eqref{J}. Then for any $t\in(0,T)$, the solution $(w,r,u)$ to the linear problem \eqref{L}$_2$-\eqref{L}$_4$ satisfies
\begin{equation}\label{high}
\begin{aligned}
&\|(w,r,u)\|_{\widetilde{L}^{\infty}_{t}(\dot{B}^{\frac{d}{2}-1}\cap\dot{B}^{\frac{d}{2}})}^{h}+\|(w,r,u)\|_{L^1_{t}(\dot{B}^{\frac{d}{2}+1})}^{h}+\frac{1}{\var}\|w\|_{L^1_{t}(\dot{B}^{\frac{d}{2}-1}\cap\dot{B}^{\frac{d}{2}})}^{h}+\frac{1}{\sqrt{\var}}\|w\|_{\widetilde{L}^2_{t}(\dot{B}^{\frac{d}{2}-1}\cap\dot{B}^{\frac{d}{2}})}^{h}\\
&\quad\quad+\sqrt{\tau}\|r\|_{\widetilde{L}^{2}_{t}(\dot{B}^{\frac{d}{2}+1})}^{h}+\frac{1}{\tau}\|z\|_{L^1_{t}(\dot{B}^{\frac{d}{2}-1}\cap\dot{B}^{\frac{d}{2}})}^{h}\\
&\quad\lesssim \|(w_{0},r_{0},u_{0})\|_{\dot{B}^{\frac{d}{2}+1}}^{h}+\|(S_{1},S_{2},S_{3})\|_{L^1_{t}(\dot{B}^{\frac{d}{2}+1})}^{h}+\mathcal{Z}(t)\mathcal{X}(t)+\int_{0}^{t}\mathcal{V}(s)\mathcal{X}(s)ds.
\end{aligned}
\end{equation}
\end{lemma}

\begin{proof}
To prove of Lemma \ref{lemmahigh}, we localize in frequencies for the equations $\eqref{L}_{2}$-$\eqref{L}_{4}$ as
\begin{equation}
\left\{
\begin{aligned}
&\partial_{t}w_{j}+v\cdot\nabla w_{j}+ (h_{1}+H_1  )\div u_{j}+ (h_2+H_2) \dfrac{w_{j}}{\var}=\dot{\Delta}_{j}S_{1}+T_{j}^1,\\
&\partial_{t}r_{j}+v\cdot\nabla r_{j}+ (h_3+H_3  )\div u_{j}=\dot{\Delta}_{j}S_{2}+T_{j}^2,\\
&\partial_{t}u_{j}+v\cdot\nabla u_{j}+\frac{u_{j}}{\tau}+ (h_4+H_4  )\nabla  r_{j}+ (h_5+H_5 ) \nabla w_{j}  =\dot{\Delta}_{j}S_{3}+T_{j}^3,
\end{aligned}
\right.\label{Lj}
\end{equation}
with the commutator terms
\begin{equation}\label{T1T2T3}
\left\{
\begin{aligned}
&T_{j}^1:=[v,\dot{\Delta}_{j}]\nabla w+[H_{1},\dot{\Delta}_{j}]\div u+\frac{1}{\var} [H_{2}, \dot{\Delta}_{j}]w,\\
&T_{j}^2:=[v,\dot{\Delta}_{j}]\nabla r+[H_{3},\dot{\Delta}_{j}]\div u,\\
&T_{j}^3:=[v,\dot{\Delta}_{j}]\nabla u+[H_{4},\dot{\Delta}_{j}]\nabla r+[H_{5},\dot{\Delta}_{j}]\nabla w.
\end{aligned}
\right.
\end{equation}
Multiplying $\eqref{Lj}_{3}$ by $u_{j}$ and integrating the resulting equation by parts, we get 
\begin{equation}\label{ujh}
\begin{aligned}
&\frac{d}{dt}\int_{\mathbb{R}^{d}} \frac{1}{2}|u_{j}|^2dx+\int_{\mathbb{R}^{d}} \frac{1}{\tau}|u_{j}|^2 dx\\
&\quad -\int_{\mathbb{R}^{d}}\big((h_{4}+H_{4}) r_{j}\div u_{j}+(h_{5}+H_{5}) w_{j}\div u_{j}             \big)dx\\
&\lesssim\|\div v\|_{L^{\infty}}\|u_{j}\|_{L^2}^2+\big(\|\dot{\Delta}_{j}S_{3}\|_{L^2}+\|T_{j}^3\|_{L^2} \big)\|u_{j}\|_{L^2}+\|\nabla H_{4}\|_{L^{\infty}}\|(w_{j},r_{j})\|_{L^2}\|u_{j}\|_{L^2}.
\end{aligned}
\end{equation}
Thence, we multiply $\eqref{Lj}_{1}$ by $\frac{h_{5}+H_{5}}{h_{1}+H_{1}} w_{j}$ and integrate the resulting equation by parts to show
\begin{equation}\label{wjh}
\begin{aligned}
&\frac{d}{dt}\int_{\mathbb{R}^{d}}\frac{1}{2}\frac{h_{5}+H_{5}}{h_{1}+H_{1}}|w_{j}|^2 dx\\
&\quad+\int_{\mathbb{R}^{d}}\left( (h_{5}+H_{5})w_{j}\div u_{j}+\frac{(h_{2}+H_{2})(h_{5}+H_{5})}{\var(h_{1}+H_{1})}|w_{j}|^2\right)dx\\
&\lesssim\bigg( \left\| \partial_{t}\Big( \frac{h_{5}+H_{5}}{h_{1}+H_{1}}\Big)\right\|_{L^{\infty}} +\left\|\frac{h_{5}+H_{5}}{h_{1}+H_{1}}\right\|_{L^{\infty}}\|\div v\|_{L^{\infty}}+\left\|\nabla \Big( \frac{h_{5}+H_{5}}{h_{1}+H_{1}}\Big)\right\|_{L^{\infty}}\|v\|_{L^{\infty}}\bigg)\|w_{j}\|_{L^2}^2\\
&\quad+\left\|\frac{h_{5}+H_{5}}{h_{1}+H_{1}}\right\|_{L^{\infty}}\big( \|\dot{\Delta}_{j}S_{1}\|_{L^2}+\|T_{j}^1\|_{L^2}\big) \|w_{j}\|_{L^2}.
\end{aligned}
\end{equation}
Similarly, direct computations on $\eqref{Lj}_{2}$ yield
\begin{equation}\label{rjh}
\begin{aligned}
&\frac{d}{dt}\int_{\mathbb{R}^{d}}\frac{1}{2}\frac{h_{4}+H_{4}}{h_{3}+H_{3}}|r_{j}|^2 dx+\int_{\mathbb{R}^{d}}(h_{4}+H_{4}) r_{j}\div u_{j}dx\\
&\leq\bigg(\left\| \partial_{t}\Big( \frac{h_{4}+H_{4}}{h_{3}+H_{3}}\Big)\right\|_{L^{\infty}}+\left\|\frac{h_{4}+H_{4}}{h_{3}+H_{3}}\right\|_{L^{\infty}}\|\div v\|_{L^{\infty}}+\left\|\nabla \Big( \frac{h_{4}+H_{4}}{h_{3}+H_{3}}\Big)\right\|_{L^{\infty}}\|v\|_{L^{\infty}}\bigg)\|r_{j}\|^2\\
&\quad+\left\|\frac{h_{4}+H_{4}}{h_{3}+H_{3}}\right\|_{L^{\infty}}\big( \|\dot{\Delta}_{j}S_{2}\|_{L^2}+\|T_{j}^2\|_{L^2}\big)\|r_{j}\|_{L^2}.
\end{aligned}
\end{equation}
To derive the cross estimate and capture the dissipative property of $r_{j}$, we gain by taking the $L^2$-inner product of $\eqref{Lj}_{3}$ with $\nabla r_{j}$ that 
\begin{equation}\label{urjh1}
\begin{aligned}
&\int_{\mathbb{R}^{d}}\partial_{t}u_{j}\cdot\nabla r_{j}dx+\int_{\mathbb{R}^{d}}  (h_{4}+H_{4})|\nabla r_{j}|^2 dx\\
&\quad\quad+\int_{\mathbb{R}^{d}} \left( (h_{5}+H_{5})\nabla w_{j}\cdot\nabla r_{j}+\frac{1}{\tau}u_{j}\cdot \nabla r_{j} \right)dx\\
&\quad\lesssim\big( \|v\|_{L^{\infty}}\|\nabla u_{j}\|_{L^2}+\|\dot{\Delta}_{j}S_{3}\|_{L^2}+\|T_{j}^3\|_{L^2}\big)\|\nabla r_{j}\|_{L^2},
\end{aligned}
\end{equation}
and taking the $L^2$-inner product of $\eqref{Lj}_{2}$ with $\div u_{j}$ that 
\begin{equation}\label{urjh2}
\begin{aligned}
&\int_{\mathbb{R}^{d}}u_{j}\cdot\nabla \partial_{t}r_{j} dx-\int_{\mathbb{R}^{d}}  (h_{3}+H_{3})|\div u_{j}|^2 dx\\
&\quad\lesssim  \big( \|v\|_{L^{\infty}}\|\nabla r_{j}\|_{L^2}+ \|\dot{\Delta}_{j}S_{2}\|_{L^2}+\|T_{j}^2\|_{L^2}\big)\|\div u_{j}\|_{L^2}.
\end{aligned}
\end{equation}
In the spirit of the work \cite{bea1} by Beauchard and Zuazua, for a small constant $\eta_{*}>0$ to be determined, we define the following Lyapunov functional with  nonlinear weights as
\begin{equation}\nonumber
\begin{aligned}
\mathcal{L}_{j}(t):=\int_{\mathbb{R}^{d}} \frac{1}{2}\left( \frac{h_{5}+H_{5}}{h_{1}+H_{1}}|w_{j}|^2+\frac{h_{4}+H_{4}}{h_{3}+H_{3}}|r_{j}|^2+|u_{j}|^2\right)dx+\frac{\eta_{*}}{\tau} 2^{-2j}\int_{\mathbb{R}^{d}}   u_{j}\cdot \nabla r_{j} dx,
\end{aligned}
\end{equation}
and its dissipation rate
\begin{equation}\nonumber
\begin{aligned}
\mathcal{H}_{j}(t):&=\int_{\mathbb{R}^{d}}\left( \frac{1}{\tau}|u_{j}|^2 +\frac{(h_{2}+H_{2})(h_{5}+H_{5})}{\var(h_{1}+H_{1})}|w_{j}|^2\right)dx\\
&\quad+\frac{\eta_{*}}{\tau} 2^{-2j} \int_{\mathbb{R}^{d}}\left(  (h_{4}+H_{4})|\nabla r_{j}|^2 + (h_{5}+H_{5})\nabla w_{j}\cdot\nabla r_{j}+\frac{1}{\tau}u_{j}\cdot \nabla r_{j} \right)dx.
\end{aligned}
\end{equation} 
One derives from assumption \eqref{alinear} and the embedding $\dot{B}^{\frac{d}{2}}\hookrightarrow L^{\infty}$ that
\begin{equation}\label{Hismall}
\begin{aligned}
\|H_{i}\|_{L^{\infty}_{t}(L^{\infty})}+\|\nabla H_{i}\|_{L^{\infty}_{t}(L^{\infty})}\lesssim \|H_{i}\|_{\widetilde{L}^{\infty}_{t}(\dot{B}^{\frac{d}{2}}\cap\dot{B}^{\frac{d}{2}+1})}\lesssim c<<1,
\end{aligned}
\end{equation}
which together with estimates \eqref{ujh}-\eqref{urjh2} and the fact that $2^{-j}\lesssim \tau\leq 1$ for any $j\geq J_{\tau}-1$ yields the following Lyapunov inequality:
\begin{equation}\label{lyapunov}
\begin{aligned}
\frac{d}{dt}\mathcal{L}_{j}(t)+\mathcal{H}_{j}(t)&\lesssim (\|\div v\|_{L^{\infty}}+\|v\|_{L^{\infty}}+\sum_{i=1}^{5}\|\partial_{t}H_{i}\|_{L^{\infty}} ) \|(r_{j},w_{j},u_{j})\|_{L^2}^2\\
&\quad+\left(\sum_{i=1}^{5}\|\partial_{t}H_{i}\|_{L^{\infty}}+\|\dot{\Delta}_{j}(S_{1},S_{2},S_{3})\|_{L^2}+\|(T_{j}^1,T_{j}^2,T_{j}^3)\|_{L^2} \right) \|(r_{j},w_{j},u_{j})\|_{L^2}.
\end{aligned}
\end{equation}
It follows from the smallness condition \eqref{Hismall}, Bernstein's inequality in Lemma \ref{lemma61} and the fact $2^{-j}\lesssim \tau$ that 
\begin{equation}\nonumber
\begin{aligned}
&(1-\eta_{*})\|(w_{j},r_{j},u_{j})\|_{L^2}^2\lesssim \mathcal{L}_{j}(t)\lesssim (1+\eta_{*})\|(w_{j},r_{j},u_{j})\|_{L^2}^2,
\end{aligned}
\end{equation}
and
\begin{equation}\nonumber
\begin{aligned}
\mathcal{H}_{j}(t)&\gtrsim \frac{1}{\tau}\|u_{j}\|_{L^2}^2+\frac{1}{\var}\|w_{j}\|_{L^2}^2+\frac{\eta_{*}}{\tau}2^{-2j}\left(\|\nabla r_{j}\|_{L^2}-\|\nabla w_{j}\|_{L^2}^2-\frac{1}{\tau^2}\|u_{j}\|_{L^2}\right)\\
&\gtrsim \frac{1}{\tau}(1-\eta_{*})\|u_{j}\|_{L^2}^2+\frac{1}{\var}(1-\eta_{*} )\|w_{j}\|_{L^2}^2+\frac{\eta_{*}}{\tau}\|r_{j}\|_{L^2},
\end{aligned}
\end{equation}
where one has used the condition $\var\leq \tau$. Thus, we can choose a sufficiently small constant $\eta_{*}>0$ independent of $\var$ and $\tau$ so that 
\begin{equation}\label{sim}
\begin{aligned}
&\mathcal{L}_{j}(t)\sim \|(w_{j},r_{j},u_{j})\|_{L^2}^2,\quad\quad \mathcal{H}_{j}(t)\gtrsim \frac{1}{\tau}\|(w_{j},r_{j},u_{j})\|_{L^2}^2\gtrsim \frac{1}{\tau}\mathcal{L}_{j}(t).
\end{aligned}
\end{equation}
Dividing the two sides of \eqref{lyapunov} by $\sqrt{\mathcal{L}_{j}(t)+\eta}$ for any $\eta>0$, we have
\begin{equation}\nonumber
\begin{aligned}
&\frac{d}{dt}\sqrt{\mathcal{L}_{j}(t)+\eta}+\frac{1}{\tau}\sqrt{\mathcal{L}_{j}(t)+\eta}-\frac{\eta}{\tau\sqrt{\mathcal{L}_{j}(t)+\eta}}\\
&~~\lesssim \left(\|\div v\|_{L^{\infty}}+\|v\|_{L^{\infty}}+\sum_{i=1}^{5}\|\partial_{t}H_{i}\|_{L^{\infty}}\right)\|(r_{j},w_{j},u_{j})\|_{L^2}\\
&~~\quad+\|\dot{\Delta}_{j}(S_{1},S_{2},S_{3})\|_{L^2}+\|(T_{j}^1,T_{j}^2,T_{j}^3)\|_{L^2}.
\end{aligned}
\end{equation}
From \eqref{sim} and the embedding $\dot{B}^{\frac{d}{2}}\hookrightarrow L^{\infty}$ we get after integrating  the above inequality over $[0,t]$ and taking the limit as $\eta\rightarrow0$ that
\begin{equation}\label{high101}
\begin{aligned}
\tau\|(w,r,u)\|_{\widetilde{L}^{\infty}_{t}(\dot{B}^{\frac{d}{2}+1})}^{h}&+\|(w,r,u)\|_{L^1_{t}(\dot{B}^{\frac{d}{2}+1})}^{h}\lesssim \tau\|(w_{0},r_{0},u_{0})\|_{\dot{B}^{\frac{d}{2}+1}}^{h}\\
&\quad\quad+\int_{0}^{t}\left(\|v(s)\|_{\dot{B}^{\frac{d}{2}}\cap\dot{B}^{\frac{d}{2}+1}}+\sum_{i=1}^{5}\|\partial_{t}H_{i}(s)\|_{\dot{B}^{\frac{d}{2}}}\right) \tau\|(w,r,u)(s)\|_{\dot{B}^{\frac{d}{2}+1}}^{h} ds\\
&\quad\quad+\tau\sum_{j\geq J_{\tau}-1}2^{j(\frac{d}{2}+1)}\|(T_{j}^1,T_{j}^2,T_{j}^3)\|_{L^1_{t}(L^2)}+\tau\|(S_{1},S_{2},S_{3})\|_{L^1_{t}(\dot{B}^{\frac{d}{2}+1})}^{h}.
\end{aligned}
\end{equation}
According to the commutator estimate \eqref{commutator}, it follows that
\begin{equation}\nonumber
\begin{aligned}
\tau\sum_{j\geq J_{\tau}-1}2^{j(\frac{d}{2}+1)}\|(T_{j}^1,T_{j}^2,T_{j}^3)\|_{L^1_{t}(L^2)}&\lesssim \int_{0}^{t}\|v(s)\|_{\dot{B}^{\frac{d}{2}+1}}\|(w,r,u)(s)\|_{\dot{B}^{\frac{d}{2}+1}}ds\\
&\quad\quad+\sum_{i=1}^{4}\|H_{i}\|_{\widetilde{L}^{\infty}_{t}(\dot{B}^{\frac{d}{2}+1})}\left(\frac{1}{\var} \|w\|_{L^1_{t}(\dot{B}^{\frac{d}{2}})}+ \tau\|r\|_{L^1_{t}(\dot{B}^{\frac{d}{2}+1})}+\|u\|_{L^1_{t}(\dot{B}^{\frac{d}{2}+1})}\right)\\
&\quad\lesssim \mathcal{Z}(t)\mathcal{X}(t).
\end{aligned}
\end{equation}
This with inequality \eqref{high101} leads to
\begin{equation}\label{high11}
\begin{aligned}
&\tau\|(w,r,u)\|_{\widetilde{L}^{\infty}_{t}(\dot{B}^{\frac{d}{2}+1})}^{h}+\|(w,r,u)\|_{L^1_{t}(\dot{B}^{\frac{d}{2}+1})}^{h}\\
&\quad\lesssim \tau\|(w_{0},r_{0},u_{0})\|_{\dot{B}^{\frac{d}{2}+1}}^{h}+\tau\|(S_{1},S_{2},S_{3})\|_{L^1_{t}(\dot{B}^{\frac{d}{2}+1})}^{h}\\
&\quad\quad+\mathcal{Z}(t)\mathcal{X}(t)+\int_{0}^{t}\mathcal{V}(s)\mathcal{X}(s)ds.
\end{aligned}
\end{equation}
On the other hand, for any $\eta>0$, we deduce from inequality \eqref{wjh} that
\begin{equation}\nonumber
\begin{aligned}
&\frac{d}{dt}\sqrt{\|w_{j}\|_{L^2}^2+\eta}+\frac{1}{\var}\sqrt{\|w_{j}\|_{L^2}^2+\eta}\\
&\lesssim 2^{j} \|u_{j}\|_{L^2}+\|(\partial_{t}H_{1},\partial_{t}H_{5})\|_{L^{\infty}}\|w_{j}\|_{L^2}\\
&\quad+\|\div v\|_{L^{\infty}} \|w_{j}\|_{L^2}+\|v\|_{L^{\infty}}\|w_{j}\|_{L^2}+\|\dot{\Delta}_{j}S_{1}\|_{L^2}+\|T_{j}^1\|_{L^2},
\end{aligned}
\end{equation}
which together with \eqref{high11} implies
\begin{equation}\nonumber
\begin{aligned}
\frac{1}{\var}\|w\|_{L^1_{t}(\dot{B}^{\frac{d}{2}})}^{h}&\lesssim \|w_{0}\|_{\dot{B}^{\frac{d}{2}}}^{h}+\|u\|_{L^1_{t}(\dot{B}^{\frac{d}{2}+1})}^{h}\\
&\quad+\|S_{1}\|_{L^1_{t}(\dot{B}^{\frac{d}{2}})}^{h}+\tau\sum_{j\geq J_{\tau}-1}2^{j(\frac{d}{2}+1)}\|T_{j}^1\|_{L^1_{t}(L^2)}\\
&\quad+\int_{0}^{t}\big( \|(\partial_{t} H_{1},\partial_{t} H_{5})(s)\|_{\dot{B}^{\frac{d}{2}}}+\|v(s)\|_{\dot{B}^{\frac{d}{2}}\cap\dot{B}^{\frac{d}{2}+1}}\big) \|w(s)\|_{\dot{B}^{\frac{d}{2}}}ds\\
&\lesssim \tau\|(w_{0},r_{0},u_{0})\|_{\dot{B}^{\frac{d}{2}+1}}^{h}+\tau\|(S_{1},S_{2},S_{3})\|_{L^1_{t}(\dot{B}^{\frac{d}{2}+1})}^{h}\\
&\quad+\mathcal{Z}(t)\mathcal{X}(t)+\int_{0}^{t}\mathcal{V}(s)\mathcal{X}(s)ds.
\end{aligned}
\end{equation}
Thanks to inequality \eqref{lhl}, one has 
\begin{align}
\|(w,r,u)\|_{\widetilde{L}^{\infty}_{t}(\dot{B}^{\frac{d}{2}-1})}^{h}\lesssim \tau\|(w,r,u)\|_{\widetilde{L}^{\infty}_{t}(\dot{B}^{\frac{d}{2}})}^{h}\lesssim \tau\|(w,r,u)\|_{\widetilde{L}^{\infty}_{t}(\dot{B}^{\frac{d}{2}+1})}^{h}.\label{345}
\end{align}
Finally,  the remain estimates in \eqref{high} can be achieved similarly to \eqref{345}. We omit the details here and complete the proof of Lemma \ref{lemmahigh}.
\end{proof}

\subsection{Recovering the $\dot{B}^{\frac{d}{2}+1}$-estimates}\label{sectiond21}

 As explained in  Remark \ref{rmkH4r2}, we need to establish the uniform $L^{\infty}_{t}(\dot{B}^{\frac{d}{2}+1})$-norm estimate of $(w,r,u)$ which  in fact leads to the uniform control of $\widetilde{L}^2_{t}(\dot{B}^{\frac{d}{2}+1})$-norms for $(\frac{1}{\sqrt{\var}}w,\frac{1}{\sqrt{\tau}}u)$ at both low and high frequencies as a byproduct.
 
 \begin{lemma}\label{lemmaL2} 
Let $T>0$, and the threshold $J_{\tau}$ be given by \eqref{J}. Then for any $t\in(0,T)$, the solution $(w,r,u)$ to the linear problem \eqref{L}$_2$-\eqref{L}$_4$ satisfies
\begin{equation}\label{all}
\begin{aligned}
&\|(w,r,u)\|_{\widetilde{L}^{\infty}_{t}(\dot{B}^{\frac{d}{2}+1})}+\frac{1}{\sqrt{\var}}\|w\|_{\widetilde{L}^2_{t}(\dot{B}^{\frac{d}{2}+1})}+\frac{1}{\sqrt{\tau}}\|u\|_{\widetilde{L}^{2}_{t}(\dot{B}^{\frac{d}{2}+1})}\\
&\quad\lesssim \|(w_{0},r_{0},u_{0})\|_{\dot{B}^{\frac{d}{2}+1}}+\|(S_{1},S_{2},S_{3})\|_{L^1_{t}(\dot{B}^{\frac{d}{2}+1})}\\
&\quad\quad+(\eta+\sqrt{\mathcal{Z}(t)})\mathcal{X}(t)+\frac{1}{\eta}\int_{0}^{t}\mathcal{V}(s)\mathcal{X}(s)ds,
\end{aligned}
\end{equation}
where $\eta\in(0,1)$ is a constant to be chosen.

\end{lemma}

\begin{proof}
We perform a $L^2$-in-time type estimates and make use of the  decay-in-$\tau$ of $u$ for $L^2$-time type norms. In fact, for any $j\in\mathbb{Z}$, by combining inequalities \eqref{ujh}-\eqref{wjh} together, we get
\begin{equation}\label{wujh}
\begin{aligned}
&\frac{d}{dt}\int_{\mathbb{R}^{d}} \frac{1}{2}\Big( \frac{h_{5}+H_{5}}{h_{1}+H_{1}}|w_{j}|^2+\frac{h_{4}+H_{4}}{h_{3}+H_{3}}|r_{j}|^2+|u_{j}|^2\Big)dx\\
&\quad+\int_{\mathbb{R}^{d}}\Big( \frac{1}{\tau}|u_{j}|^2 +\frac{(h_{2}+H_{2})(h_{5}+H_{5})}{\var(h_{1}+H_{1})}|w_{j}|^2\Big)dx\\
&\lesssim \|\div v\|_{L^{\infty}}\|(w_{j},r_{j},u_{j})\|_{L^2}^2+\Big(\sum_{i=1}^{5}\|\partial_{t}H_{i}\|_{L^{\infty}}+\|v\|_{L^{\infty}} \Big)\|(w_{j},r_{j})\|_{L^2}^2\\
&\quad+\|\nabla H_{4}\|_{L^{\infty}}\|(w_{j},r_{j})\|_{L^2}\|u_{j}\|_{L^2}+\|T_{j}^1\|_{L^2} \|w_{j}\|_{L^2}\\
&\quad+\|T_{j}^2\|_{L^2}\|r_{j}\|_{L^2}+\|T_{j}^3\|_{L^2}\|u_{j}\|_{L^2}+\|\dot{\Delta}_{j}(S_{1},S_{2},S_{3})\|_{L^2}\|(w_{j},r_{j},u_{j})\|_{L^2}.
\end{aligned}
\end{equation}
Furthermore, from \eqref{wujh} we have
\begin{equation}\label{highl2}
\begin{aligned}
&\|(w,r,u)\|_{\widetilde{L}^{\infty}_{t}(\dot{B}^{\frac{d}{2}+1})}+\frac{1}{\sqrt{\var}}\|w\|_{\widetilde{L}^{2}_{t}(\dot{B}^{\frac{d}{2}+1})}+\frac{1}{\sqrt{\tau}}\|u\|_{\widetilde{L}^{2}_{t}(\dot{B}^{\frac{d}{2}+1})}\\
&\quad\lesssim\|(w_{0},r_{0},u_{0})\|_{\dot{B}^{\frac{d}{2}+1}}+\|v\|_{L^1_{t}(\dot{B}^{\frac{d}{2}+1})}^{\frac{1}{2}}\|(w,r,u)\|_{\widetilde{L}^{\infty}_{t}(\dot{B}^{\frac{d}{2}+1})}\\
&\quad\quad+\Big(\int_{0}^{t}\Big( \sum_{i=1}^{5}\|\partial_{t}H_{i}(s)\|_{\dot{B}^{\frac{d}{2}}}+\|v(s)\|_{\dot{B}^{\frac{d}{2}}}\Big)\|(w,r)(s)\|_{\dot{B}^{\frac{d}{2}+1}}ds   \Big)^{\frac{1}{2}}\|(w,r)\|_{\widetilde{L}^{\infty}_{t}(\dot{B}^{\frac{d}{2}+1})}^{\frac{1}{2}}\\
&\quad\quad+\Big(\|H_{4}\|_{\widetilde{L}^{\infty}_{t}(\dot{B}^{\frac{d}{2}+1})}\frac{1}{\sqrt{\tau}}\|u\|_{\widetilde{L}^{2}_{t}(\dot{B}^{\frac{d}{2}+1})}\sqrt{\tau}\|(w,r)\|_{\widetilde{L}^{2}_{t}(\dot{B}^{\frac{d}{2}+1})} \Big)^{\frac{1}{2}}\\
&\quad\quad+\sum_{j\in\mathbb{Z}}2^{j(\frac{d}{2}+1)}\Big( \int_{0}^{t}\big(\|T_{j}^1\|_{L^2}\|w_{j}\|_{L^2}+\|T_{j}^2\|_{L^2}\|r_{j}\|_{L^2}+\|T_{j}^3\|_{L^2}\|u_{j}\|_{L^2} \big)ds\Big)^{\frac{1}{2}}\\
&\quad\quad+\Big(\|(S_{1},S_{2},S_{3})\|_{L^1_{t}(\dot{B}^{\frac{d}{2}+1})} \|(w,r,u)\|_{\widetilde{L}^{\infty}_{t}(\dot{B}^{\frac{d}{2}+1})}\Big)^{\frac{1}{2}}.
\end{aligned}
\end{equation}
The right-hand side of inequality \eqref{highl2} can be estimated as follows. By the commutator estimate \eqref{commutator}, we have
\begin{equation}\nonumber
\begin{aligned}
&\sum_{j\in\mathbb{Z}}2^{j(\frac{d}{2}+1)} \Big(\int_{0}^{t}\|T_{j}^1\|_{L^2}\|w_{j}\|_{L^2} ds\Big)^{\frac{1}{2}}\\
&\quad\lesssim \bigg(\|w\|_{\widetilde{L}^{\infty}_{t}(\dot{B}^{\frac{d}{2}+1})}\int_{0}^{t}\|v(s)\|_{\dot{B}^{\frac{d}{2}+1}}\|w(s)\|_{\dot{B}^{\frac{d}{2}+1}}ds \\
&\quad\quad+\|H_{1}\|_{\widetilde{L}^{\infty}_{t}(\dot{B}^{\frac{d}{2}+1})}\|w\|_{\widetilde{L}^{\infty}_{t}(\dot{B}^{\frac{d}{2}+1})}\|u\|_{L^1_{t}(\dot{B}^{\frac{d}{2}+1})}+\|H_{2}\|_{\widetilde{L}^{\infty}_{t}(\dot{B}^{\frac{d}{2}+1})}\frac{1}{\var}\|w\|_{\widetilde{L}^{2}_{t}(\dot{B}^{\frac{d}{2}+1})}^2\bigg)^{\frac{1}{2}}\\
&\quad\lesssim\sqrt{\mathcal{Z}(t)}\mathcal{X}(t)+\Big(\int_{0}^{t}\mathcal{V}(s)\mathcal{X}(s)ds\Big)^{\frac{1}{2}}\sqrt{\mathcal{X}(t)} .
\end{aligned}
\end{equation}
Similarly, it holds
\begin{equation}\nonumber
\begin{aligned}
&\sum_{j\in\mathbb{Z}}2^{j(\frac{d}{2}+1)} \Big(\int_{0}^{t}\|T_{j}^2\|_{L^2}\|r_{j}\|_{L^2} ds\Big)^{\frac{1}{2}}\\
&\quad \lesssim \bigg(\|r\|_{\widetilde{L}^{\infty}_{t}(\dot{B}^{\frac{d}{2}+1})} \int_{0}^{t}\|v(s)\|_{\dot{B}^{\frac{d}{2}+1}}\|r(s)\|_{\dot{B}^{\frac{d}{2}+1}}ds +\|H_{3}\|_{\widetilde{L}^{\infty}_{t}(\dot{B}^{\frac{d}{2}+1})}\|r\|_{\widetilde{L}^{\infty}_{t}(\dot{B}^{\frac{d}{2}+1})}\|u\|_{L^1_{t}(\dot{B}^{\frac{d}{2}+1})} \bigg)^{\frac{1}{2}}\\
&\quad\lesssim\sqrt{\mathcal{Z}(t)}\mathcal{X}(t)+\Big(\int_{0}^{t}\mathcal{V}(s)\mathcal{X}(s)ds\Big)^{\frac{1}{2}}\sqrt{\mathcal{X}(t)} ,
\end{aligned}
\end{equation}
and
\begin{equation}\nonumber
\begin{aligned}
&\sum_{j\in\mathbb{Z}}2^{j(\frac{d}{2}+1)} \Big(\int_{0}^{t}\|T_{j}^3\|_{L^2}\|u_{j}\|_{L^2} ds\Big)^{\frac{1}{2}}\\
&\quad\lesssim \bigg(\|u\|_{\widetilde{L}^{\infty}_{t}(\dot{B}^{\frac{d}{2}+1})}\int_{0}^{t}\|v(s)\|_{\dot{B}^{\frac{d}{2}+1}} \|u(s)\|_{\dot{B}^{\frac{d}{2}+1}}ds +\|(H_{4},H_{5})\|_{\widetilde{L}^{\infty}_{t}(\dot{B}^{\frac{d}{2}+1})} \|r\|_{\widetilde{L}^{\infty}_{t}(\dot{B}^{\frac{d}{2}+1})}\|u\|_{L^1_{t}(\dot{B}^{\frac{d}{2}+1})}\bigg)^{\frac{1}{2}}\\
&\quad\lesssim \sqrt{\mathcal{Z}(t)}\mathcal{X}(t)+\Big(\int_{0}^{t}\mathcal{V}(s)\mathcal{X}(s)ds\Big)^{\frac{1}{2}}\sqrt{\mathcal{X}(t)} .
\end{aligned}
\end{equation}
We conclude from the above estimates that
\begin{equation}\nonumber
\begin{aligned}
&\|(w,r,u)\|_{\widetilde{L}^{\infty}_{t}(\dot{B}^{\frac{d}{2}+1})}+\frac{1}{\sqrt{\var}}\|w\|_{\widetilde{L}^{2}_{t}(\dot{B}^{\frac{d}{2}+1})}+\frac{1}{\sqrt{\tau}}\|u\|_{\widetilde{L}^{2}_{t}(\dot{B}^{\frac{d}{2}+1})}\\
&\quad\lesssim \|(w_{0},r_{0},u_{0})\|_{\dot{B}^{\frac{d}{2}+1}}+\|(S_{1},S_{2},S_{3})\|_{L^1_{t}(\dot{B}^{\frac{d}{2}+1})}\\
&\quad\quad+\sqrt{\mathcal{Z}(t)}\mathcal{X}(t)+\Big(\int_{0}^{t}\mathcal{V}(s)\mathcal{X}(s)ds\Big)^{\frac{1}{2}}\sqrt{\mathcal{X}(t)}.
\end{aligned}
\end{equation}
Applying H\"older's inequality to the above estimate leads to inequality \eqref{all}.
\end{proof}

\section{Global  well-posedness   for the nonlinear problems}\label{section4}

\subsection{The Cauchy problem of System \eqref{BN}}\label{subsectionexistenceBN}

In this section, we prove the uniform in $\var$ and $\tau$ global existence and uniqueness of solutions to the Cauchy problem for \eqref{BN} subject to the initial data $(\alpha_{\pm,0}, \rho_{\pm,0}, u_{0})$. i.e. Theorem \ref{theorem11}. For simplicity, we omit the superscript concerning the parameters $\var$ and $\tau$ in this section.

\vspace{2mm}

\underline{\it\textbf{Proof of Theorem \ref{theorem11}:}}~Let $(\alpha_{\pm,0}, \rho_{\pm,0},u_{0})$ satisfy the smallness condition \eqref{a1} and denote
$$
\mathcal{X}_{0}:=\|(\alpha_{\pm,0}-\bar{\alpha}_{\pm}, \rho_{\pm,0}-\bar{\rho}_{\pm},u_{0})\|_{\dot{B}^{\frac{d}{2}-1}\cap\dot{B}^{\frac{d}{2}+1}}.
$$
Let $(y_{0},w_{0},r_{0})$ be the perturbation of $(\alpha_{\pm,0}, \rho_{\pm,0})$ given by \eqref{y0w0r0}. 

\begin{itemize}
\item {\emph{Step 1: Construction of approximation sequence}}
\end{itemize}

For any $n\geq1$, we define the regularized perturbation
\begin{align}
(y^{n}_{0},w^{n}_{0}, r^{n}_{0},u^{n}_{0})(x)=(\dot{S}_{n}y_{0},\dot{S}_{n}w_{0}, \dot{S}_{n}r_{0},\dot{S}_{n}u_{0})(x),\nonumber
\end{align}
where $\dot{S}_{n}$ is the low-frequency cut-off operator (see Section \ref{sectionbesov}). One can verify that $(y^{n}_{0},w^{n}_{0}, r^{n}_{0},u^{n}_{0})$ is smooth and converges to $(y_{0},w_{0},r_{0},u_{0})$ strongly in $\dot{B}^{\frac{d}{2}-1}\cap\dot{B}^{\frac{d}{2}+1}$ as $n\rightarrow\infty$. In addition, due to Lemmas \ref{lemma63} and \ref{lemma64}, there exists a constant $C_{0}^{*}$ independent of $n$, $\var$ and $\tau$ such that
\begin{equation}\label{initialbound}
\begin{aligned}
\|(y^{n}_{0},w^{n}_{0}, r^{n}_{0},u^{n}_{0})\|_{\dot{B}^{\frac{d}{2}-1}\cap\dot{B}^{\frac{d}{2}+1}}\leq C_{0}^{*}\mathcal{X}_{0}.
\end{aligned}
\end{equation}
\normalcolor

Set $(y^{0},w^{0},r^{0},u^{0}):=(0,0,0,0)$. For any $n\geq0$, we consider the approximate scheme for \eqref{re} as follows:
\begin{equation}\label{ren}
\left\{
\begin{aligned}
&\partial_{t}y^{n+1}+u^{n}\cdot\nabla y^{n+1}=0,\\
&\partial_{t}w^{n+1}+u^{n}\cdot\nabla w^{n+1}+ (\bar{F}_{1}+G_1^{n}  )\div u^{n+1}+ (\bar{F}_2+G_2^{n}) \dfrac{w^{n+1}}{\var}=0,\\
&\partial_{t}r^{n+1}+u^{n}\cdot\nabla r^{n+1}+ (\bar{F}_3+G_3^{n}  )\div u^{n+1}=F_{4}^{n}\frac{(w^{n})^2}{\var},\\
&\partial_{t}u^{n+1}+u^{n}\cdot\nabla u^{n+1}+\frac{u^{n+1}}{\tau}+ ( \bar{F}_{0}+G_0^{n}  )(\nabla  r^{n+1}+(\gamma_{+}-\gamma_{-}) \nabla w^{n+1})  =0,\\
&(y^{n+1},w^{n+1},r^{n+1},u^{n+1})(0,x)=(y^{n}_{0},w^{n}_{0}, r^{n}_{0},u^{n}_{0})(x),
\end{aligned}
\right.
\end{equation}
with $F_{i}^{n}=F_{i}^{\var,\tau}(y^{n},w^{n},r^{n})$, $\bar{Fi}$ and $G_{i}^{n}=G^{\var,\tau}_{i}(y^{n},w^{n},r^{n})$ defined in \eqref{Fi}, \eqref{barFi} and \eqref{Gi}, respectively. 
We define the functional space $\mathbb{E}_t$ associated to the following norm:
\begin{equation}\nonumber
\begin{aligned}
\|(y,w,r,u)\|_{\mathbb{E}_{t}}:&= \|(y ,w ,r ,u )\|_{\widetilde{L}^{\infty}_{t}(\dot{B}^{\frac{d}{2}-1}\cap \dot{B}^{\frac{d}{2}+1})}+\|(\partial_{t}y,\partial_{t}w,\partial_{t}r,\partial_{t}u)\|_{L^1_{t}(\dot{B}^{\frac{d}{2}})}\\
&\quad+\frac{1}{\var}\|w \|_{L^1_{t}(\dot{B}^{\frac{d}{2}-1}\cap\dot{B}^{\frac{d}{2}})}+\frac{1}{\sqrt{\var}}\|w \|_{\widetilde{L}^2_{t}(\dot{B}^{\frac{d}{2}-1}\cap\dot{B}^{\frac{d}{2}+1})}\\
&\quad+\tau \|r \|_{L^1_{t}(\dot{B}^{\frac{d}{2}+1}\cap\dot{B}^{\frac{d}{2}+2})}^{\ell}+\|r\|_{L^1_{t}(\dot{B}^{\frac{d}{2}+1})}^{h}+\tau\|r\|_{L^1_{t}(\dot{B}^{\frac{d}{2}+1})}+\sqrt{\tau}\|r\|_{\widetilde{L}^2_{t}(\dot{B}^{\frac{d}{2}}\cap\dot{B}^{\frac{d}{2}+1})}\\
&\quad+\|u \|_{L^1_{t}(\dot{B}^{\frac{d}{2}}\cap\dot{B}^{\frac{d}{2}+1})}+\frac{1}{\sqrt{\tau}}\|u \|_{\widetilde{L}^2_{t}(\dot{B}^{\frac{d}{2}-1}\cap\dot{B}^{\frac{d}{2}+1})}.
\end{aligned}
\end{equation}
For any fixed $n\geq1$, we assume that  $(y^{n} ,w^{n} ,r^{n} ,u^{n} )$ satisfies
\begin{align}
\|(y^{n},w^{n},r^{n},u^{n})\|_{\mathbb{E}_{t}}+\frac{1}{\tau}\|u^{n} +\tau (\bar{F}_{0}+G^{n-1}_{0}  )\nabla r^{n} \|_{L^1_{t}(\dot{B}^{\frac{d}{2}-1}\cap \dot{B}^{\frac{d}{2}})}\leq 2C_{0}C_{0}^{*}\mathcal{X}_{0},\quad t>0,\label{apriori}
\end{align}
where the constants $C_{0}$ and $C_{0}^{*}$ are given by \eqref{elinear} and \eqref{initialbound}, respectively. Since the initial data is smooth, by virtue of the classical theorems for the transport equation $\eqref{ren}_{1}$ and the symmetric hyperbolic system $\eqref{ren}_{2}$-$\eqref{ren}_{4}$ (cf. \cite{bahouri1,serre1}), there exists a unique global solution $(y^{n+1} ,w^{n+1} ,r^{n+1} ,u^{n+1} )\in \mathcal{C}(\mathbb{R}_{+}; H^{s})$ with all $s>\frac{d}{2}+1$.

\begin{itemize}
\item {\emph{ Step 2: Uniform estimate}}
\end{itemize}

Our goal is to show that $(y^{n+1} ,w^{n+1} ,r^{n+1} ,u^{n+1} )$ also satisfies the estimate \eqref{apriori} uniformly in $n$, $\var$, $\tau$ and time. To this end, we first let $\mathcal{X}_{0}\leq 1$. It follows from \eqref{apriori} and the composition estimates in Lemma \ref{lemma64} that
\begin{equation}\label{Gin}
\begin{aligned}
&\sum_{i=0}^{4}\|G_{i}^{n}\|_{\widetilde{L}^{\infty}_{t}(\dot{B}^{\frac{d}{2}-1}\cap\dot{B}^{\frac{d}{2}+1})}\leq C_{2}^{*}\|(y^{n},r^{n},w^{n})\|_{\widetilde{L}^{\infty}_{t}(\dot{B}^{\frac{d}{2}-1}\cap\dot{B}^{\frac{d}{2}+1})},
\end{aligned}
\end{equation}
and similarly, 
\begin{equation}\label{dtGin}
\begin{aligned}
\sum_{i=0}^{4}\|\partial_{t}G_{i}^{n}\|_{L^1_{t}(\dot{B}^{\frac{d}{2}})}&\leq C_{3}^{*} \|(\partial_{t}y^{n},\partial_{t}w^{n},\partial_{t}r^{n})\|_{L^1_{t}(\dot{B}^{\frac{d}{2}})},
\end{aligned}
\end{equation}
for some universal constants $C_{2}^{*}$ and $C_{3}^{*}$. According to \eqref{initialbound} and \eqref{Gin}, we can justify the condition \eqref{alinear} provided that
\begin{align}
    \mathcal{X}_{0}\leq c_{1}^{*}:=\frac{c}{2C_{0}C_{0}^{*}C_{2}^{*}}.\nonumber
\end{align} 
Hence, we are able to employ the uniform a-priori estimate established in Proposition \ref{proplinear} to obtain
\begin{equation}\label{oooooooooo}
\begin{aligned}
&\|(y^{n+1},w^{n+1},r^{n+1},u^{n+1})\|_{\mathbb{E}_t}+\frac{1}{\tau}\|u^{n+1} +\tau (\bar{F}_{0}+G^{n}_{0}  )\nabla r^{n+1} \|_{L^1_{t}(\dot{B}^{\frac{d}{2}-1}\cap \dot{B}^{\frac{d}{2}})}\\
&\quad\leq C_{0} {\rm{exp}}\Big(C_{0}\int_{0}^{t}\big(\|u^{n}(s)\|_{\dot{B}^{\frac{d}{2}}\cap\dot{B}^{\frac{d}{2}+1}}+\sum_{i=0}^{4}\|\partial_{t}G_{i}^{n}(s)\|_{\dot{B}^{\frac{d}{2}}}\big)ds \Big)\\
&\quad\quad\times \Big( \|(y^{n}_{0},w^{n}_{0}, r^{n}_{0},u^{n}_{0})\|_{\dot{B}^{\frac{d}{2}-1}\cap \dot{B}^{\frac{d}{2}+1}}+\|F_{4}^{n}\frac{(w^{n})^2}{\var}\|_{L^1_{t}(\dot{B}^{\frac{d}{2}-1}\cap \dot{B}^{\frac{d}{2}+1})}\Big).
\end{aligned}
\end{equation}
Applying  \eqref{apriori}, Lemma \ref{lemma63} and \ref{lemma64} gives directly
\begin{equation}\label{oooooooooo1}
\begin{aligned}
\|F_{4}^{n}\frac{(w^{n})^2}{\var}\|_{L^1_{t}(\dot{B}^{\frac{d}{2}-1}\cap \dot{B}^{\frac{d}{2}+1})}\leq\frac{ C_{4}^{*}}{\var}\|w^{n}\|_{\widetilde{L}^2_{t}(\dot{B}^{\frac{d}{2}-1}\cap\dot{B}^{\frac{d}{2}+1})}^2 ,
\end{aligned}
\end{equation}
where $C_{4}^{*}>0$ is a universal constant. Combining  \eqref{initialbound}, \eqref{apriori}, \eqref{dtGin}, \eqref{oooooooooo} and \eqref{oooooooooo1} together, we have
\begin{equation}\nonumber
\begin{aligned}
&\|(y^{n+1},w^{n+1},r^{n+1},u^{n+1})\|_{\mathbb{E}_t}+\frac{1}{\tau}\|u^{n+1} +\tau (\bar{F}_{0}+G^{n}_{0}  )\nabla r^{n+1} \|_{L^1_{t}(\dot{B}^{\frac{d}{2}-1}\cap \dot{B}^{\frac{d}{2}})}\\
&\quad\leq C_{0}e^{2(1+C_{3}^{*})C_{0}^2C_{0}^{*}\mathcal{X}_{0}}\Big( C_{0}^{*}\mathcal{X}_{0}+C_{4}^{*}(2C_{0} C_{0}^{*} \mathcal{X}_{0})^2\Big)\leq  2C_{0}C_{0}^{*}\mathcal{X}_{0},
\end{aligned}
\end{equation}
as long as 
\begin{align}
    \mathcal{X}_{0}\leq c_{2}^{*}:=\min\Big\{ \frac{1}{2(1+C_{3}^{*})C_{0}^2C_{0}^{*}\log \frac{3}{2}}, \frac{2}{9(2C_{0} C_{0}^{*})^2C^{*}_{4}}\Big\}\nonumber
\end{align} 
such that $e^{2(1+C_{3}^{*})C_{0}^2C_{0}^{*}\mathcal{X}_{0}}\leq \frac{3}{2}$ and $C_{4}^{*}(2C_{0} C_{0}^{*} \mathcal{X}_{0})^2\leq \frac{1}{3}\mathcal{X}_{0}$. Thus, the uniform estimate \eqref{apriori} holds true for any $n\geq0$.

\begin{itemize}
\item {\emph{Step 3: Strong convergence}}
\end{itemize}

The uniform estimate  \eqref{apriori} enables us to obtain the weak compactness of the approximate sequence. In order to pass the limit in every nonlinear term of \eqref{ren} as $n\rightarrow\infty$, one needs to have robust strong compactness in a suitable sense. Classical compact embedding theorem merely gives the strong convergence locally in space-time and up to a subsequence, which is not enough for System \eqref{ren}.  In what follows,  we show that the strong convergence holds in $\mathbb{R}_{+}\times\mathbb{R}^{d}$ for the whole sequence.

\begin{lemma}\label{lemmastrong}
There exists a small constant $c_{3}^{*}\in(0, \min\{1,c_{1}^{*},c_{2}^{*}\}]$ and a limit $(y,w,r,u)$ such that if $\mathcal{X}_{0}\leq c_{3}^{*}$, then as $n\rightarrow\infty$,
\begin{equation}\label{strong}
\begin{aligned}
&(y^{n},w^{n},r^{n},u^{n})\rightarrow (y,w,r,u)\quad\text{strongly in}\quad  L^{\infty}(\mathbb{R}_{+};\dot{B}^{\frac{d}{2}-1}\cap\dot{B}^{\frac{d}{2}}).
\end{aligned}
\end{equation}
In particular, we have
\begin{equation}\label{strong1}
\begin{aligned}
&(y^{n},w^{n},r^{n},u^{n})\rightarrow (y,w,r,u)\quad\text{strongly in}\quad  L^{\infty}(\mathbb{R}_{+};L^{d}\cap L^{\infty}).
\end{aligned}
\end{equation}
\end{lemma}

\begin{proof}
In order to show \eqref{strong}, one needs to perform uniform energy estimates on the error unknown
$$
(\widetilde{y}^{n},\widetilde{w}^{n},\widetilde{r}^{n},\widetilde{u}^{n}):=(y^{n+1}-y^{n},w^{n+1}-w^{n},r^{n+1}-r^{n},u^{n+1}-u^{n}).
$$
Following the framework in Section \ref{section3}, we aim to estimate the functional
\begin{align}
\widetilde{\mathcal{X}}^{n}(t)&=\|(\widetilde{y}^{n},\widetilde{w}^{n},\widetilde{r}^{n},\widetilde{u}^{n})\|_{\widetilde{L}^{\infty}_{t}(\dot{B}^{\frac{d}{2}-1}\cap \dot{B}^{\frac{d}{2}})}\nonumber+\frac{1}{\var}\|\widetilde{w}^{n}\|_{L^1_{t}(\dot{B}^{\frac{d}{2}-1})}+\frac{1}{\sqrt{\var}}\|\widetilde{w}^{n}\|_{\widetilde{L}^2_{t}(\dot{B}^{\frac{d}{2}-1}\cap\dot{B}^{\frac{d}{2}})}\nonumber\\
&\quad+\tau\|\widetilde{r}^{n}\|_{L^1_{t}(\dot{B}^{\frac{d}{2}+1})}^{\ell}+\|\widetilde{r}^{n}\|_{L^1_{t}(\dot{B}^{\frac{d}{2}})}^{h}+\sqrt{\tau}\|\widetilde{r}^{n}\|_{\widetilde{L}^2_{t}(\dot{B}^{\frac{d}{2}})}\nonumber\\
&\quad+\|\widetilde{u}^{n}\|_{L^1_{t}(\dot{B}^{\frac{d}{2}})}+\frac{1}{\sqrt{\tau}}\|\widetilde{u}^{n}\|_{\widetilde{L}^2_{t}(\dot{B}^{\frac{d}{2}-1}\cap\dot{B}^{\frac{d}{2}})}\nonumber+ \frac{1}{\tau}\| \widetilde{u}^{n}+\tau( \bar{F}_{0}+G_0^{n-1}  )\nabla \widetilde{r}^{n} \|_{L^1_{t}(\dot{B}^{\frac{d}{2}-1})}.\nonumber
\end{align}
To this matter, one can verify that for any $n\geq1$, $(\widetilde{y}^{n},\widetilde{w}^{n+1},\widetilde{r}^{n+1},\widetilde{u}^{n+1})$ solves
\begin{equation}\label{rendelta}
\left\{
\begin{aligned}
&\partial_{t}\widetilde{y}^{n+1}+u^{n}\cdot\nabla \widetilde{y}^{n+1}=\widetilde{S}^{n}_{1},\\
&\partial_{t}\widetilde{w}^{n+1}+u^{n}\cdot\nabla \widetilde{w}^{n+1}+ (\bar{F}_{1}+G_1^{n}  )\div \widetilde{u}^{n+1}+ (\bar{F}_2+G_2^{n}) \dfrac{\widetilde{w}^{n+1}}{\var}=\widetilde{S}^{n}_{2},\\
&\partial_{t}\widetilde{r}^{n+1}+u^{n}\cdot\nabla \widetilde{r}^{n+1}+ (\bar{F}_3+G_3^{n}  )\div \widetilde{u}^{n+1}=\widetilde{S}^{n}_{3},\\
&\partial_{t}\widetilde{u}^{n+1}+u^{n}\cdot\nabla \widetilde{u}^{n+1}+\frac{\widetilde{u}^{n+1}}{\tau}+ ( \bar{F}_{0}+G_0^{n}  )(\nabla  \widetilde{r}^{n+1}+(\gamma_{+}-\gamma_{-}) \nabla \widetilde{w}^{n+1}) =\widetilde{S}^{n}_{4},\\
&(\widetilde{y}^{n+1},\widetilde{w}^{n+1},\widetilde{r}^{n+1},\widetilde{u}^{n+1})(0,x)=\dot{\Delta}_{n}(y_{0},w_{0}, r_{0},u_{0})(x),
\end{aligned}
\right.
\end{equation}
with 
\begin{equation}\nonumber
\left\{
\begin{aligned}
&\widetilde{S}^{n}_{1}:=-\widetilde{u}^{n}\cdot\nabla y^{n},\\
&\widetilde{S}^{n}_{2}:=-\widetilde{u}^{n}\cdot\nabla w^{n}-(G_{1}^{n}-G_{1}^{n-1})\div u^{n}-(G_{2}^{n}-G_{2}^{n-1})\frac{w^{n}}{\var},\\
&\widetilde{S}^{n}_{3}:=-\widetilde{u}^{n}\cdot\nabla r^{n}-(G_{3}^{n}-G_{3}^{n-1})\div u^{n}+(F_{4}^{n}-F_{4}^{n-1})\frac{(w^{n})^2}{\var}+F_{4}^{n-1}\frac{(w^{n}+w^{n-1})\widetilde{w}^{n}}{\var},\\
&\widetilde{S}^{n}_{4}:=-\widetilde{u}^{n}\cdot\nabla u^{n}-(G_{0}^{n}-G_{0}^{n-1})(\nabla r^{n}+(\gamma_{+}-\gamma_{-})\nabla w^{n} ).
\end{aligned}
\right.
\end{equation}
First, employing Lemma \ref{maximaldamped} to $\eqref{rendelta}_{1}$ yields
\begin{equation}\label{deltay}
\begin{aligned}
&\|\widetilde{y}^{n+1}\|_{\widetilde{L}^{\infty}_{t}(\dot{B}^{\frac{d}{2}-1}\cap\dot{B}^{\frac{d}{2}})}\lesssim {\rm{exp}}\Big(\|u^{n}\|_{L^1_{t}(\dot{B}^{\frac{d}{2}+1})}\Big)(\|\dot{\Delta}_{n} y_{0}\|_{\dot{B}^{\frac{d}{2}-1}\cap\dot{B}^{\frac{d}{2}}}+\|\widetilde{S}^{n}_{1}\|_{L^1_{t}(\dot{B}^{\frac{d}{2}-1}\cap\dot{B}^{\frac{d}{2}})}).
\end{aligned}
\end{equation}
By similar computations on $\eqref{rendelta}_{2}$-$\eqref{rendelta}_{4}$ as in Lemma \ref{lemmalow}, we have the low-frequency estimate at the $\dot{B}^{\frac{d}{2}-1}$ regularity level:
\begin{equation}\label{lowdelta}
\begin{aligned}
&\|(\widetilde{w}^{n+1},\widetilde{r}^{n+1},\widetilde{u}^{n+1})\|_{\widetilde{L}^{\infty}_{t}(\dot{B}^{\frac{d}{2}-1})}^{\ell}+\tau\|\widetilde{r}^{n+1}\|_{L^1_{t}(\dot{B}^{\frac{d}{2}+1})}^{\ell}+\sqrt{\tau}\|\widetilde{r}^{n+1}\|_{\widetilde{L}^2_{t}(\dot{B}^{\frac{d}{2}})}^{\ell}\\
&\quad\quad+\frac{1}{\var}\|\widetilde{w}^{n+1}\|_{L^1_{t}(\dot{B}^{\frac{d}{2}-1})}^{\ell}+\frac{1}{\sqrt{\var}}\|\widetilde{w}^{n+1}\|_{\widetilde{L}^2_{t}(\dot{B}^{\frac{d}{2}-1})}^{\ell}+\frac{1}{\tau}\| \widetilde{u}^{n+1}+\tau( \bar{F}_{0}+G_0^{n}  )\nabla \widetilde{r}^{n+1} \|_{L^1_{t}(\dot{B}^{\frac{d}{2}-1})}^{\ell}\\
&\quad\lesssim \|\dot{\Delta}_{n}(w_{0},r_{0},u_{0})\|_{\dot{B}^{\frac{d}{2}-1}}+\|(\widetilde{S}^{n}_{2},\widetilde{S}^{n}_{3},\widetilde{S}^{n}_{4})\|_{L^1_{t}(\dot{B}^{\frac{d}{2}-1})}^{\ell}+\mathcal{Z}^{n}(t)\widetilde{\mathcal{X}}^{n+1}(t)+\int_{0}^{t}\mathcal{V}^{n}(s)\widetilde{\mathcal{X}}^{n+1}(s)ds,
\end{aligned}
\end{equation}
with 
\begin{equation}\nonumber
\left\{
\begin{aligned}
\mathcal{Z}^{n}(t)&=\sum_{i=0}^{3}\|G_{i}^{n}\|_{\dot{B}^{\frac{d}{2}-1}\cap\dot{B}^{\frac{d}{2}+1}},\\
\mathcal{V}^{n}(t)&=\|u^{n}\|_{\dot{B}^{\frac{d}{2}}\cap\dot{B}^{\frac{d}{2}+1}}+\sum_{i=0}^{3}\|\partial_{t}G_{i}^{n}\|_{\dot{B}^{\frac{d}{2}-1}\cap\dot{B}^{\frac{d}{2}+1}}.
\end{aligned}
\right.
\end{equation}
Moreover, as in Lemma \ref{lemmahigh}, one can obtain the following estimate in the high-frequency region:
\begin{equation}\label{highdelta}
\begin{aligned}
&\|(\widetilde{w}^{n+1},\widetilde{r}^{n+1},\widetilde{u}^{n+1})\|_{\widetilde{L}^{\infty}_{t}(\dot{B}^{\frac{d}{2}-1})\cap L^1_{t}(\dot{B}^{\frac{d}{2}})}^{h}+\frac{1}{\tau}\|\widetilde{u}^{n+1}+\tau( \bar{F}_{0}+G_0^{n}  )\nabla \widetilde{r}^{n+1}\|_{L^1_{t}(\dot{B}^{\frac{d}{2}-1})}^{h}\\
&\quad\quad+\frac{1}{\var}\|\widetilde{w}^{n+1}\|_{L^1_{t}(\dot{B}^{\frac{d}{2}-1})}^{h}+\frac{1}{\sqrt{\var}}\|\widetilde{w}^{n+1}\|_{\widetilde{L}^2_{t}(\dot{B}^{\frac{d}{2}-1})}^{h}+\sqrt{\tau}\|\widetilde{r}^{n+1}\|_{\widetilde{L}^{2}_{t}(\dot{B}^{\frac{d}{2}})}^{h}\\
&\quad\lesssim \|\dot{\Delta}_{n}(w_{0},r_{0},u_{0})\|_{\dot{B}^{\frac{d}{2}}}^{h}+\|(\widetilde{S}^{n}_{2},\widetilde{S}^{n}_{3},\widetilde{S}^{n}_{4})\|_{L^1_{t}(\dot{B}^{\frac{d}{2}})}^{h}+\mathcal{Z}^{n}(t)\widetilde{\mathcal{X}}^{n+1}(t)+\int_{0}^{t}\mathcal{V}^{n}(s)\widetilde{\mathcal{X}}^{n+1}(s)ds.
\end{aligned}
\end{equation}
Finally, following the proof of Lemma \ref{lemmaL2}, we also have
\begin{equation}\label{alldelta}
\begin{aligned}
&\|(\widetilde{w}^{n+1},\widetilde{r}^{n+1},\widetilde{u}^{n+1})\|_{\widetilde{L}^{\infty}_{t}(\dot{B}^{\frac{d}{2}})}+\frac{1}{\sqrt{\var}}\|\widetilde{w}^{n+1}\|_{\widetilde{L}^2_{t}(\dot{B}^{\frac{d}{2}})}+\frac{1}{\sqrt{\tau}}\|\widetilde{u}^{n+1}\|_{\widetilde{L}^{2}_{t}(\dot{B}^{\frac{d}{2}})}\\
&\quad\lesssim \|\dot{\Delta}_{n}(w_{0},r_{0},u_{0})\|_{\dot{B}^{\frac{d}{2}}}+\|(\widetilde{S}^{n}_{2},\widetilde{S}^{n}_{3},\widetilde{S}^{n}_{4})\|_{L^1_{t}(\dot{B}^{\frac{d}{2}})}\\
&\quad\quad+(\widetilde{\eta}+\sqrt{\mathcal{Z}^{n}(t)})\widetilde{\mathcal{X}}^{n+1}(t)+\frac{1}{\widetilde{\eta}}\int_{0}^{t}\mathcal{V}^{n}(s)\widetilde{\mathcal{X}}^{n+1}(s)ds,
\end{aligned}
\end{equation}
where $\widetilde{\eta}>0$ is a constant to be chosen. Combining \eqref{deltay}-\eqref{alldelta} together, we arrive at
\begin{equation}\label{Xtwide}
\begin{aligned}
\widetilde{\mathcal{X}}^{n+1}(t)&\lesssim \|\dot{\Delta}_{n}(y_{0},w_{0},r_{0},u_{0})\|_{\dot{B}^{\frac{d}{2}-1}\cap\dot{B}^{\frac{d}{2}}}+\|(\widetilde{S}^{n}_{2},\widetilde{S}^{n}_{3},\widetilde{S}^{n}_{4})\|_{L^1_{t}(\dot{B}^{\frac{d}{2}-1}\cap\dot{B}^{\frac{d}{2}})}\\
&\quad+(\widetilde{\eta}+\sqrt{\mathcal{Z}^{n}(t)}+\mathcal{Z}^{n}(t))\widetilde{\mathcal{X}}^{n+1}(t)+(1+\frac{1}{\widetilde{\eta}})\int_{0}^{t}\mathcal{V}^{n}(s)\widetilde{\mathcal{X}}^{n+1}(s)ds.
\end{aligned}
\end{equation}
Now we estimate the right-hand side of \eqref{Xtwide} as follow. First, due to $\Delta_{n'}\dot{\Delta}_{n}=0$ with $|n-n'|\geq 2$, one has
\begin{equation}\nonumber
\begin{aligned}
\|\dot{\Delta}_{n}(y_{0},w_{0},r_{0},u_{0})\|_{\dot{B}^{\frac{d}{2}-1}\cap\dot{B}^{\frac{d}{2}}}\lesssim \sum_{|n'-n|\leq 1}(2^{(\frac{d}{2}-1)n'}+2^{\frac{d}{2}n'})\|\dot{\Delta}_{n'}(y_{0},w_{0},r_{0},u_{0})\|_{L^2}.
\end{aligned}
\end{equation}
Thence, applying uniform estimate \eqref{apriori} leads to
\begin{equation}\nonumber
\left\{
\begin{aligned}
&\mathcal{Z}^{n}(t)\lesssim \|(y^{n},w^{n},r^{n})\|_{\widetilde{L}^{\infty}_{t}(\dot{B}^{\frac{d}{2}-1}\cap \dot{B}^{\frac{d}{2}+1})}\lesssim \mathcal{X}_{0},\nonumber\\
&\int_{0}^{t}\mathcal{V}^{n}(s)ds\lesssim \|(\partial_{t}y^{n},\partial_{t}w^{n},\partial_{t}r^{n})\|_{L^1_{t}(\dot{B}^{\frac{d}{2}})}\lesssim \mathcal{X}_{0}.\nonumber
\end{aligned}
\right.
\end{equation}
Regarding the nonlinear terms on the right-hand side of \eqref{Xtwide}, one deduces from \eqref{uv2}, \eqref{F1}-\eqref{F3} and \eqref{apriori} that
\begin{align}
&\|(\widetilde{S}^{n}_{1},\widetilde{S}^{n}_{2},\widetilde{S}^{n}_{3},\widetilde{S}^{n}_{4})\|_{L^1_{t}(\dot{B}^{\frac{d}{2}-1}\cap\dot{B}^{\frac{d}{2}})}\nonumber\\
&\quad\lesssim \|(y^{n},w^{n},r^{n})\|_{\widetilde{L}^{\infty}_{t}(\dot{B}^{\frac{d}{2}}\cap \dot{B}^{\frac{d}{2}+1})}\|\widetilde{u}^{n}\|_{L^1_{t}(\dot{B}^{\frac{d}{2}})}+\|(\widetilde{y}^{n},\widetilde{w}^{n},\widetilde{r}^{n},\widetilde{u}^{n})\|_{\widetilde{L}^{\infty}_{t}(\dot{B}^{\frac{d}{2}})}\|u^{n}\|_{L^1_{t}(\dot{B}^{\frac{d}{2}}\cap\dot{B}^{\frac{d}{2}+1})}\nonumber\\
&\quad\quad+\frac{1}{\sqrt{\var}}\|(w^{n},w^{n-1})\|_{\widetilde{L}^{2}_{t}(\dot{B}^{\frac{d}{2}})}\frac{1}{\sqrt{\var}}\|\widetilde{w}^{n}\|_{\widetilde{L}^{2}_{t}(\dot{B}^{\frac{d}{2}-1}\cap\dot{B}^{\frac{d}{2}})}\lesssim \mathcal{X}_{0}\widetilde{\mathcal{X}}^{n}(t).\nonumber
\end{align}
Gathering the above estimates into \eqref{Xtwide} and \eqref{deltay} and letting both $\widetilde{\eta}$ and $\mathcal{X}_{0}$ be sufficiently small, we obtain
\begin{equation}\nonumber
\begin{aligned}
\widetilde{\mathcal{X}}^{n+1}(t)&\lesssim \sum_{|n'-n|\leq 1}(2^{(\frac{d}{2}-1)n'}+2^{\frac{d}{2}n'})\|\dot{\Delta}_{n'}(y_{0},w_{0},r_{0},u_{0})\|_{L^2}+\mathcal{X}_{0}\widetilde{\mathcal{X}}^{n}(t).
\end{aligned}
\end{equation}
Summing this over $n\geq 1$ leads to 
\begin{equation}\nonumber
\begin{aligned}
\sum_{n=1}^{\infty} \widetilde{\mathcal{X}}^{n+1}(t)\lesssim \|(y_{0},w_{0},r_{0},u_{0})\|_{\dot{B}^{\frac{d}{2}-1}\cap\dot{B}^{\frac{d}{2}}}+\mathcal{X}_{0}\sum_{n=1}^{\infty} \widetilde{\mathcal{X}}^{n}(t).
\end{aligned}
\end{equation}
Given
$(y^{0},w^{0},r^{0},u^{0})=(0,0,0,0)$ and $(y^{1},w^{1},r^{1},u^{1})=(y_{0},w_{0},r_{0},u_{0})$, we take sufficiently small $\mathcal{X}_{0}$ to have
\begin{equation}\label{Xninfty}
\begin{aligned}
&\sum_{n=0}^{\infty}  \widetilde{\mathcal{X}}^{n}(t)\lesssim \|(y_{0},w_{0},r_{0},u_{0})\|_{\dot{B}^{\frac{d}{2}-1}\cap\dot{B}^{\frac{d}{2}}}.
\end{aligned}
\end{equation}
Now we define
$(y,w,r,u):=\sum_{n'=0}^{\infty}(\widetilde{y}^{n'},\widetilde{w}^{n'},\widetilde{r}^{n'},\widetilde{u}^{n'}).$
Thanks to \eqref{Xninfty} and 
$$
(y^{n},w^{n},r^{n},u^{n})=\sum_{n'=0}^{n}(\widetilde{y}^{n'},\widetilde{w}^{n'},\widetilde{r}^{n'},\widetilde{u}^{n'}),
$$
it follows that
\begin{equation}\nonumber
\begin{aligned}
\|(y^{n},w^{n},r^{n},u^{n})-(y,w,r,u)\|_{L^{\infty}_{t}(\dot{B}^{\frac{d}{2}-1}\cap\dot{B}^{\frac{d}{2}})}\leq \sum_{n'\geq n+1}^{\infty}\widetilde{\mathcal{X}}^{n'}(t)\underset{n\rightarrow\infty}{\rightarrow} 0.
\end{aligned}
\end{equation}Gathering the embeddings in Lemma \ref{lemma22}, we get \eqref{strong1}, which finishes the proof of Lemma \ref{lemmastrong}.

\end{proof}

\begin{itemize}
\item {\emph{Step 4: Global existence}}
\end{itemize}

Let $\mathcal{X}_{0}\leq c_3^*$. By virtue of the strong convergence properties \eqref{strong}-\eqref{strong1}, one can pass to the limit  as $n\rightarrow\infty$ in \eqref{ren} and justify that the limit $(y,w,r,u)$, obtained in Lemma \ref{lemmastrong}, is indeed a global strong solution to the Cauchy problem \eqref{re}. In addition, taking advantage of Fatou's property in Lemma \ref{lemma22}, for all $t>0$, we have 
\begin{equation}
\begin{aligned}
    &\|(y,w,r,u)\|_{\mathbb{E}_t}+\frac{1}{\tau}\|u +\tau (\bar{F}_{0}+G_{0}  )\nabla r \|_{L^1_{t}(\dot{B}^{\frac{d}{2}-1}\cap \dot{B}^{\frac{d}{2}})}\\
   &\quad \lesssim  \liminf_{n\rightarrow \infty} \Big(\|(y^{n},w^{n},r^{n},u^{n})\|_{\mathbb{E}_t}+\frac{1}{\tau}\|u^{n} +\tau (\bar{F}_{0}+G^{n-1}_{0}  )\nabla r^{n} \|_{L^1_{t}(\dot{B}^{\frac{d}{2}-1}\cap \dot{B}^{\frac{d}{2}})}\Big)\lesssim \mathcal{X}_{0}.\label{uniformlimit}
\end{aligned}
\end{equation}

To prove the time continuity property in \eqref{r1}-\eqref{uniform1}, our proof relies on the uniform bound \eqref{uniformlimit} and employs a reasoning analogous to that found in \cite{danchin1}. Since $\|(\partial_ty,\partial_tw,\partial_tr,\partial_tu)\|$ lies in $L^1(\mathbb{R}_{+};{\dot{B}^{\frac d2}})$, one has
$(y,w,r,u)\in\mathcal{C}_b(\mathbb{R}_{+};\dot{B}^{\frac d2}).$ To recover $(y,w,r,u)\in\mathcal{C}_b(\mathbb{R}_{+};\dot{B}^{\frac d2-1}\cap \dot{B}^{\frac{d}{2}+1})$, we shall investigate each equations separately. Recall that the solution $(y,w,r,u)$ satisfies
\[
\left\{
\begin{array}
[c]{l}%
\partial_ty=-u\cdot \nabla y,
\\\partial_tw=-u\cdot \nabla w-\bigl(\bar F_{1}+G_1\bigr)\div u-\bigl(\bar F_2+G_2\bigr)\dfrac{w}{\var},\\\partial_tr=-u\cdot \nabla r-\bigl(\bar F_{3}+G_3\bigr)\div u-F_4\dfrac{w^2}{\var},\\
\partial_tu=-u\cdot \nabla u-\dfrac{u}{\tau}- \bigl(\bar F_0+G_0\bigr)%
\nabla  r-\left(  \gamma_{+}-\gamma_{-}\right)  \bigl(\bar F_0+G_0\bigr)\nabla w.
\end{array}
\right.
\]
As the right-hand side terms of the components $y,r$ and $w$ belong to $L^1(\mathbb{R}_{+};\dot{B}^{\frac{d}{2}-1})$, we directly get $(y,r,w)\in\mathcal{C}_b(\mathbb{R}_{+};\dot{B}^{\frac d2-1}).$ Concerning the equation of $u$, its right-hand side lies in $L^2(\mathbb{R}_{+};\dot{B}^{\frac d2-1})$ thus we can also recover that $u$ belongs to $\mathcal{C}_b(\mathbb{R}_{+};\dot{B}^{\frac d2-1})$. We are left with recovering the time continuity of $(y,w,r,u)$ in $\dot{B}^{\frac d2+1}$. First, for a fixed $j\in \mathbb{Z}$, each $(y_j,w_j,r_j,u_j)$ is continuous in time with values in $L^2$ due to Bernstein's inequality.
Now, thanks to $(y,w,r,u)\in L^\infty(\mathbb{R}_{+};\dot{B}^{\frac d2+1}),$ for any $\eta>0$, there exists an large integer $J_*$ such that, for all $t>0$
\begin{eqnarray}\nonumber \sum_{|j|\geq J_*}2^{j\left(\frac d2+1\right)}\|(y_j,w_j,r_j,u_j)\|_{L^\infty_t(L^2)}<\frac{\eta}{2}.
\end{eqnarray}
Thus, for any time $t, t'\in \mathbb{R}_{+}$, we have
\begin{align*}
\|(y,w,r,u)(t)-(y,w,r,u)(t')\|_{\dot{B}^{\frac d2+1}}&\leq \sum_{|j|<J_*}2^{j\left(\frac d2+1\right)}\|\ddj((y,w,r,u)(t)-(y,w,r,u)(t'))\|_{L^2}\\&+\sum_{|j|\geq J_*}2^{j\left(\frac d2+1\right)}\|\ddj((y,w,r,u)(t)-(y,w,r,u)(t'))\|_{L^2}&\\
& \leq 2^{J_*} \|(y,w,r,u)(t)-(y,w,r,u)(t')\|_{\dot{B}^{\frac d2}}+\eta\underset{t\to t'}{\rightarrow} \eta.
\end{align*} 
Since $\eta$ is an arbitrary constant, we get $(y,w,r,u)\in \mathcal{C}_b(\mathbb{R}_{+};\dot{B}^{\frac d2+1})$. Finally, applying the inverse function theorem, we can see that once $\alpha_{\pm}$ and $\rho_{\pm}$ are determined by \eqref{newunknows}, then $(\alpha_{\pm},\rho_{\pm}, u)\in \mathcal{C}_{b}(\mathbb{R}_{+};\dot{B}^{\frac{d}{2}-1}\cap\dot{B}^{\frac{d}{2}+1})$ is the unique global strong solution to the original Cauchy problem of System \eqref{BN}. Using \eqref{uniformlimit}, product laws and composite estimates, we are able to verify that $(\alpha_{\pm},\rho_{\pm}, u)$ satisfies the properties \eqref{r1}-\eqref{uniform1}. To complete the proof, we prove the uniqueness in our regularity framework below.

\begin{itemize}
\item {\emph{ Step 5: Uniqueness}}
\end{itemize}

The final step is to show the uniqueness of solutions to \eqref{BN} belonging to the regularity class  \eqref{r1}. We emphasize that the proof does not require the smallness of regularity for initial data. It suffices to consider the reformulated system \eqref{re}. Without loss of generality, as the parameters do not affect our argument for proving uniqueness, we set $\var=\tau=1$. Let $(y_{1},w_{1},r_{1},u_{1})$ and $(y_{2},w_{2},r_{2},u_{2})$ be two solutions to \eqref{re} subject to the same initial data  $(y_{0},w_{0},r_{0},u_{0})$, satisfying \eqref{r1} and $\bar{F}_{i}+G_{i}(y_{j},r_{j},w_{j})>0$ for $i=0,1,2,3$ and $j=1,2$. The difference
$$
(\widetilde{y},\widetilde{w},\widetilde{r},\widetilde{u}):=(y_{1}-y_{2},w_{1}-w_{2},r_{1}-r_{2},u_{1}-u_{2})
$$
solves
\begin{equation}\label{renerror}
\left\{
\begin{aligned}
&\partial_{t}\widetilde{y}+u_{1}\cdot\nabla \widetilde{y}=\widetilde{S}_{1},\\
&\partial_{t}\widetilde{w}+u_{1}\cdot\nabla \widetilde{w}+ (\bar{F}_{1}+G_1^{1}  )\div \widetilde{u}+ \bar{F}_2 \widetilde{w}=\widetilde{S}_{2},\\
&\partial_{t}\widetilde{r}+u_{1}\cdot\nabla \widetilde{r}+ (\bar{F}_3+G_3^{1}  )\div \widetilde{u}=\widetilde{S}_{3},\\
&\partial_{t}\widetilde{u}+u_{1}\cdot\nabla \widetilde{u}+\widetilde{u}+ ( \bar{F}_{0}+G_0^{1}  )(\nabla  \widetilde{r}+(\gamma_{+}-\gamma_{-}) \nabla \widetilde{w}) =\widetilde{S}_{4},\\
&(\widetilde{y},\widetilde{w},\widetilde{r},\widetilde{u})(0,x)=(0,0,0,0),
\end{aligned}
\right.
\end{equation}
where we denoted 
\begin{equation}\nonumber
\left\{
\begin{aligned}
\widetilde{S}_{1}:&=-\widetilde{u} \cdot\nabla y_{2},\\
\widetilde{S}_{2}:&=-\widetilde{u}\cdot\nabla w_{2}-(G_{1}^1-G_{1}^2)\div u_{2}-(G_2^{1}-G_2^{2})w_{2}-G_2^{1} \widetilde{w},\\
\widetilde{S}_{3}:&=-\widetilde{u} \cdot\nabla r_{2}-(G_3^{1}-G_{3}^2)\div u_{2}+(F_{4}^1-F_{4}^2)(w_{1})^2-F_{4}^2(w_{1}+w_{2})\widetilde{w},\\
\widetilde{S}^{4}:&=-\widetilde{u} \cdot\nabla u_{2}-(G_0^{1}-G_0^{2})(\nabla r_{2}+(\gamma_{+}-\gamma_{-})\nabla w_{2}).
\end{aligned}
\right.
\end{equation}
Here $G_{i}^{l}:=G_{i}(y_l,w_l,r_l)$ and $F_{4}^{l}:=F_{4}(y_l,w_l,r_l)$ with $i=0,1,2,3,$ and $l=1,2$. Applying Lemma \ref{maximaldamped} to $\eqref{renerror}$ implies that, for all $t>0$,
\begin{equation}\label{deltayun}
\begin{aligned}
\|\widetilde{y}(t)\|_{\dot{B}^{\frac{d}{2}}}&\lesssim {\rm{exp}}\Big(\|u_{1}\|_{L^1_{t}(\dot{B}^{\frac{d}{2}+1})}\Big)\int_{0}^{t}\|\widetilde{S}^{n}_{1}\|_{\dot{B}^{\frac{d}{2}}}ds.
\end{aligned}
\end{equation}
Through the application of the weighted Lyapunov functional method, as outlined in \eqref{Lj}-\eqref{rjh}, we obtain
\begin{equation}\label{deltauj}
\begin{aligned}
&\frac{d}{dt}\int_{\mathbb{R}^{d}} \frac{1}{2}\left( W_{1}|\widetilde{w}_{j}|^2+W_{2}|\widetilde{r}_{j}|^2+|\widetilde{u}_{j}|^2\right)dx+\int_{\mathbb{R}^{d}}\left( W_{3}|\widetilde{w}_{j}|^2+|\widetilde{u}_{j}|^2\right)dx\\
&\lesssim \Big(\|(\partial_{t}y_{1},\partial_{t}w_{1},\partial_{t}r_{1},\nabla y_{1}, \nabla w_{1},\nabla r_{1})\|_{L^{\infty}}+\|u_{1}\|_{W^{1,\infty}}\Big)\|(\widetilde{w}_{j},\widetilde{r}_{j},\widetilde{u}_{j})\|_{L^2}^2\\
&\quad+\Big(\|(\widetilde{T}_{1,j},\widetilde{T}_{2,j},\widetilde{T}_{3,j})\|_{L^2} +\|\dot{\Delta}_{j}(\widetilde{S}_{2},\widetilde{S}_{3},\widetilde{S}_{4})\|_{L^2}\Big)\|(\widetilde{w}_{j},\widetilde{r}_{j},\widetilde{u}_{j})\|_{L^2},
\end{aligned}
\end{equation}
where $W_{i}=W_{i}(y_{1},w_{1},r_{1})>0$, $i=1,2,3$ are some smooth weight functions depending on $\bar{F}_{i}+G_{i}^{1}>0$, and the commutator terms are given by
\begin{equation}\nonumber
\left\{
\begin{aligned}
&\widetilde{T}_{1,j}:=[u_{1},\dot{\Delta}_{j}]\nabla \widetilde{w}+[G_{1}^1,\dot{\Delta}_{j}]\div \widetilde{u},\\
&\widetilde{T}_{2,j}:=[u_{1},\dot{\Delta}_{j}]\nabla \widetilde{r}+[G_{3}^1,\dot{\Delta}_{j}]\div \widetilde{u},\\
&\widetilde{T}_{3,j}:=[u_{1},\dot{\Delta}_{j}]\nabla \widetilde{u} +[G_{0}^1,\dot{\Delta}_{j}]\nabla \widetilde{r}+[G_{0}^1,\dot{\Delta}_{j}]\nabla \widetilde{w}.
\end{aligned}
\right.
\end{equation}
Integrating \eqref{deltauj} over $[0,t]$, taking the square root of both sides and summing the resulting estimate over $j\in\mathbb{Z}$ with the weight $2^{\frac{d}{2}j}$, we get
\begin{equation}\nonumber
\begin{aligned}
&\|(\widetilde{w} ,\widetilde{r} ,\widetilde{u} )\|_{\widetilde{L}^{\infty}_{t}(\dot{B}^{\frac{d}{2}})}+\|\widetilde{u} \|_{\widetilde{L}^{2}_{t}(\dot{B}^{\frac{d}{2}})}+\|\widetilde{w} \|_{\widetilde{L}^{2}_{t}(\dot{B}^{\frac{d}{2}})}\\
&\quad\lesssim \bigg( \int_{0}^{t}\Big((\|(\partial_{t}y_{1},\partial_{t}w_{1},\partial_{t}r_{1},\nabla y_{1}, \nabla w_{1},\nabla r_{1})\|_{L^{\infty}}+\|u _{1}\|_{W^{1,\infty}})\|(\widetilde{w} ,\widetilde{r} ,\widetilde{u} )\|_{\dot{B}^{\frac{d}{2}}} \\
&\quad\quad+\sum_{j\in\mathbb{Z}}2^{j\frac{d}{2}} \|(\widetilde{T}_{1,j},\widetilde{T}_{2,j},\widetilde{T}_{3,j})\|_{L^2}+\|(\widetilde{S}_{2},\widetilde{S}_{3},\widetilde{S}_{4})\|_{\dot{B}^{\frac{d}{2}}}\Big) ds \bigg)^{\frac{1}{2}}\|(\widetilde{w} ,\widetilde{r} ,\widetilde{u} )\|_{\widetilde{L}^{\infty}_{t}(\dot{B}^{\frac{d}{2}})}^{\frac{1}{2}},\\
\end{aligned}
\end{equation}
which together with Young's inequality implies
\begin{equation}\label{widetildeeee}
\begin{aligned}
\|(\widetilde{w} ,\widetilde{r} ,\widetilde{u} )(t)\|_{\dot{B}^{\frac{d}{2}}}&\lesssim \int_{0}^{t}\Big((\|(\partial_{t}y_{1},\partial_{t}w_{1},\partial_{t}r_{1},\nabla y_{1}, \nabla w_{1},\nabla r_{1})\|_{L^{\infty}}+\|u _{1}\|_{W^{1,\infty}})\|(\widetilde{w} ,\widetilde{r} ,\widetilde{u} )\|_{\dot{B}^{\frac{d}{2}}} \\
&\quad\quad+\sum_{j\in\mathbb{Z}}2^{j\frac{d}{2}} \|(\widetilde{T}_{1,j},\widetilde{T}_{2,j},\widetilde{T}_{3,j})\|_{L^2}+\|(\widetilde{S}_{2},\widetilde{S}_{3},\widetilde{S}_{4})\|_{\dot{B}^{\frac{d}{2}}} \Big)ds,\quad \quad t>0.
\end{aligned}
\end{equation}
Using the classical commutator estimate in Lemma \ref{lemmacom} implies
\begin{equation}\nonumber
\begin{aligned}
&\sum_{j\in\mathbb{Z}}2^{j\frac{d}{2}} \|(\widetilde{T}_{1,j},\widetilde{T}_{2,j},\widetilde{T}_{3,j})\|_{L^2}\lesssim \|(y_{1}, r_{1}, w_{1}, u_{1})\|_{\dot{B}^{\frac{d}{2}+1}}\|(\widetilde{w},\widetilde{r},\widetilde{u})\|_{\dot{B}^{\frac{d}{2}}}.
\end{aligned}
\end{equation}
In addition, according to standard product laws and composite estimates in Lemmas \ref{lemma63} and \ref{lemma64}, the nonlinear terms $(\widetilde{S}_{1},\widetilde{S}_{2},\widetilde{S}_{3},\widetilde{S}_{4})$ can be handled as
\begin{equation}\nonumber
\begin{aligned}
\|(\widetilde{S}_{1},\widetilde{S}_{2},\widetilde{S}_{3},\widetilde{S}_{4})\|_{\dot{B}^{\frac{d}{2}}}\lesssim \Big(\|(y_{2},w_{2},r_{2},u_{2})\|_{\dot{B}^{\frac{d}{2}+1}}+\|(w_{1},w_{2})\|_{\dot{B}^{\frac{d}{2}}}\Big)\|(\widetilde{y},\widetilde{w} ,\widetilde{r} ,\widetilde{u} )\|_{\dot{B}^{\frac{d}{2}}}. 
\end{aligned}
\end{equation}
Substituting the above two estimates into \eqref{deltayun} and \eqref{widetildeeee} and then employing Gr\"onwall's inequality, we end up with $\|(\widetilde{y},\widetilde{w} ,\widetilde{r} ,\widetilde{u} )(t)\|_{\dot{B}^{\frac{d}{2}}}=0$ for all $t>0$, which concludes the proof of Theorem \ref{theorem11}.
\normalcolor

\subsection{The Cauchy problem of Systems \eqref{K} and \eqref{PM}}\label{sec:exKapila}

We provide a brief explanation of the proof of the global existence and uniqueness for System \eqref{PM}. The proof of the result for System \eqref{K} (Theorem \ref{theoremK}) follows a very similar procedure, so we omit the details here for brevity. The uniformity of the estimate \eqref{uniform1} for System \eqref{K} allows us to construct solutions for System \eqref{PM} by taking the limit as the relaxation parameter $\tau\to 0$.

The following lemma states the uniform estimate verified by the solutions of System \eqref{Keta}, which are rescaled from estimate \eqref{uniform1K} obtained in Theorem \ref{theoremK} for System \eqref{K}.
\begin{lemma}\label{lemmaKuniform}
Let  $(\alpha_{\pm}^{\tau},\rho_{\pm}^{\tau}, u^{\tau})$ be the global solution to the Cauchy problem of System \eqref{K} subject to the initial data $(\alpha_{\pm,0}^\tau, \rho_{\pm,0}^\tau, u_{0}^\tau)$ given by Theorem \ref{theoremK} and $(\beta_{\pm}^{\tau},\varrho_{\pm}^{\tau}, v^{\tau})$ be defined by the diffusive scaling \eqref{scalingK}, then  it holds that 
\begin{equation}\label{uniformK}
\begin{aligned}
&\|(\beta_{\pm}^{\tau}-\bar{\alpha}_{\pm}, \varrho_{\pm}^{\tau}-\bar{\rho}_{\pm})\|_{\widetilde{L}^{\infty} (\dot{B}^{\frac{d}{2}-1}\cap\dot{B}^{\frac{d}{2}+1})}\\
&\quad\quad+\|(\Pi^{\tau}-\bar{P},\varrho_{\pm}^{\tau}-\bar{\rho}_{\pm})\|_{L^1 (\dot{B}^{\frac{d}{2}+1})}+\|(\Pi^{\tau}-\bar{P},\varrho_{\pm}^{\tau}-\bar{\rho})\|_{\widetilde{L}^2 (\dot{B}^{\frac{d}{2}}\cap\dot{B}^{\frac{d}{2}+1})}\\
&\quad\quad+\|v^{\tau}\|_{L^1 (\dot{B}^{\frac{d}{2}}\cap\dot{B}^{\frac{d}{2}+1})}+\|v^{\tau}\|_{\widetilde{L}^2 (\dot{B}^{\frac{d}{2}-1}\cap\dot{B}^{\frac{d}{2}+1})}+\frac{1}{\tau}\| z^{\tau}\|_{L^1 (\dot{B}^{\frac{d}{2}-1}\cap\dot{B}^{\frac{d}{2}})}\\
&\quad\leq C\|(\alpha^\tau_{\pm,0}-\bar{\alpha}_{\pm}, \rho^\tau_{\pm,0}-\bar{\rho}_{\pm},u^\tau_{0})\|_{\dot{B}^{\frac{d}{2}-1}\cap\dot{B}^{\frac{d}{2}+1}},
\end{aligned}
\end{equation}
with $z^{\tau}:=v^{\tau}+\frac{1}{\varrho^{\tau}}\nabla \Pi^{\tau}$, and $C>0$ a universal constant.
\end{lemma}

\underline{\it\textbf{Proof of Theorem \ref{theoremPM}:}}~ Assume that the initial data $(\beta_{\pm,0},\varrho_{\pm,0})$ satisfies \eqref{a1PM}. For any $\tau\in(0,1)$, we define the regularized data as
\begin{align}
(\alpha^{\tau}_{\pm,0},\rho^{\tau}_{\pm,0})(x):=\dot{S}_{[\frac{1}{\tau}]}(\beta_{\pm, 0},\varrho_{\pm, 0})(x)\quad {\rm and}\quad u^\tau_0(x):=0.\nonumber
\end{align} 
Hence, by employing Theorem \ref{theoremK} we can obtain a sequence $(\alpha_{\pm}^{\tau}, \rho_{\pm}^{\tau},u^{\tau})$, which is the global solution to System \eqref{K} subject to the initial data $(\alpha^{\tau}_{\pm,0},\rho^{\tau}_{\pm,0}, u^\tau_0)$.
Taking the diffusive scaling \eqref{scalingK}, one has that 
$(\beta_{\pm}^{\tau}, \varrho_{\pm}^{\tau}, v^\tau)$ is the global solution to System \eqref{Keta} subject to the initial data $(\alpha^{\var}_{\pm,0},\rho^{\var}_{\pm,0}, u^\tau_0/\tau)$. In view of the uniform estimate \eqref{uniformK} established in Lemma \ref{lemmaKuniform}, the Aubin-Lions lemma and the cantor diagonal process, there exists a limit $(\beta_{\pm}, \varrho_{\pm})$ such that as $\tau\rightarrow0$, up to a subsequence, $\chi(\beta_{\pm}^{\tau}, \varrho_{\pm}^{\tau})$ converges to $\chi(\beta_{\pm}, \varrho_{\pm})$ in $C([0,T];\dot{B}^{s})$ $(s<\frac{d}{2}+1)$ strongly for any given time $T>0$ and $\chi\in C^{\infty}_{c}(\mathbb{R}^{d}\times [0,T])$. Thus, we can check that $(\beta_{\pm}, \varrho_{\pm})$ solves System \eqref{PM} in the sense of distributions. Furthermore, taking advantage of the Fatou property and the optimal regularity estimate in Lemma \ref{maximalheat} for the equation of $\Pi$, we can conclude \eqref{r1PM}. Finally, the uniqueness can be obtained in a simple fashion. The interested reader may also refer to \cite{c1,c2} for more details.


\section{Relaxation limits with convergence rates}\label{sectionrelaxation}

\subsection{Pressure-relaxation limit: System \eqref{BN} to System \eqref{K}}\label{sectionrelaxationBNK}

In this section, we prove Theorem \ref{theorem12} related to the convergence rate of the relaxation process between System \eqref{BN} and System \eqref{K}. Let $(\alpha^{\var,\tau}_{\pm},\rho^{\var,\tau}_{\pm}, u^{\var,\tau})$ and  $(\alpha^{\tau}_{\pm},\rho^{\tau}_{\pm}, u^{\tau})$ be the global solutions to System \eqref{BN} with the initial data $(\alpha^{\var,\tau}_{\pm,0},\rho^{\var,\tau}_{\pm,0},u^{\var,\tau}_{0})$ and System \eqref{K} with the initial data $(\alpha^{\var,\tau}_{\pm,0},\rho^{\var,\tau}_{\pm,0},u^{\var,\tau}_{0})$ given by Theorems \ref{theorem11} and \ref{theoremK}, respectively. Denote the error variables
\begin{align}
&( \delta \alpha_{\pm}, \delta \rho_{\pm},\delta u):=(\alpha^{\var,\tau}_{\pm}-\alpha^{\tau}_{\pm}, u^{\var,\tau}-u^{\tau}),\nonumber\\
&(\delta\rho,\delta P_{\pm}, \delta P):=( \rho^{\var,\tau}_{\pm}-\rho^{\tau}_{\pm},\rho^{\var,\tau}-\rho^{\tau}, P_{\pm}(\rho^{\var,\tau}_{\pm})-P_{\pm}(\rho^{\tau}_{\pm}),P^{\var,\tau}-P^{\tau}),\nonumber
\end{align}
and the initial data of $\delta P$
\begin{align}
   \delta P|_{t=0}= P_{0}^{\var,\tau}-P^{\tau}_{0},\quad P_{0}^{\var}:=\alpha_{+,0}^{\var, \tau}P_{+}(\rho_{+,0}^{\var, \tau})+\alpha_{-,0}^{\var, \tau}P_{-}(\rho_{-,0}^{\var, \tau}),\quad  P_{0}^{\tau}:=P_{+}(\rho_{+,0}^{\tau}).\label{deltaP0}
\end{align}

First, to avoid dealing with difficult nonlinearities in the equation of $ \delta\alpha_{\pm}$, we work with the following purely transported variable instead of $\delta\alpha_{\pm}$:
\begin{align}
    \delta Y:=\frac{\alpha^{\var,\tau}_{+}\rho^{\var,\tau}_{+}}{\rho^{\var,\tau}}-\frac{\alpha^{\tau}_{+}\rho^{\tau}_{+}}{\rho^{\tau}}. \label{deltaY}
\end{align}
with the initial data
\begin{align}
\delta Y|_{t=0}=Y^{\var,\tau}_{0}-Y^{\tau}_{0},\quad Y^{\var,\tau}_{0}:=\frac{\alpha_{+,0}^{\var,\tau}\rho_{+,0}^{\var,\tau}}{\alpha_{+,0}^{\var,\tau}\rho_{+,0}^{\var,\tau}+\alpha_{-,0}^{\var,\tau}\rho_{-,0}^{\var,\tau}},\quad Y^{\tau}_{0}:= \frac{\alpha^{\tau}_{+,0}\rho^{\tau}_{+,0}}{\alpha^{\tau}_{+,0}\rho^{\tau}_{+,0}+\alpha^{\tau}_{-,0}\rho_{-,0}^{\tau}}.\label{deltaY0}
\end{align}
\begin{lemma} For $d\geq3$, under the assumption \eqref{errora1}, $\delta Y$ satisfies the following estimate{\rm:}
\begin{equation}\label{deltaalpha}
\begin{aligned}
\|\delta Y\|_{\widetilde{L}^{\infty}_{t}(\dot{B}^{\frac{d}{2}-2}\cap\dot{B}^{\frac{d}{2}-1})}&\lesssim  \sqrt{\var \tau}+{o}(1) \|\delta u\|_{L^1_{t}(\dot{B}^{\frac{d}{2}-1} )}.
\end{aligned}
\end{equation}
\end{lemma}
\begin{proof}
Since the equation of $\delta Y$ reads  
\begin{align}
&\partial_{t}\delta Y+u^{\var,\tau}\cdot\nabla \delta Y=-\delta u\cdot\nabla \frac{\alpha^{\tau}_{+}\rho^{\tau}_{+}}{\rho^{\tau}},\nonumber
\end{align}
 Lemma \ref{maximaldamped} and the product law \eqref{uv2} for $d\geq3$ gives
\begin{equation}\nonumber
\begin{aligned}
\|\delta Y\|_{\widetilde{L}^{\infty}_{t}(\dot{B}^{\frac{d}{2}-2}\cap\dot{B}^{\frac{d}{2}-1})}&\lesssim {\rm{exp}}\Big(\|u^{\var,\tau}\|_{L^1_{t}(\dot{B}^{\frac{d}{2}+1})}\Big)\Big(\sqrt{\var\tau}+\|\delta u\|_{L^1_{t}(\dot{B}^{\frac{d}{2}-1})}  \| \nabla\frac{\alpha^{\tau}_{+}\rho^{\tau}_{+}}{\rho^{\tau}}\|_{\widetilde{L}^{\infty}_{t}(\dot{B}^{\frac{d}{2}-1}\cap\dot{B}^{\frac{d}{2}})} \Big).
\end{aligned}
\end{equation}
This together with the uniform estimate \eqref{uniform1} leads to \eqref{deltaalpha}.
\end{proof}

We are now ready to estimate $(\delta\alpha_{\pm},\delta \rho_{\pm}, \delta P_{\pm}, \delta P)$. It is easy to verify that $P^{\var,\tau}$ satisfies
\begin{equation}\label{Pvar}
\begin{aligned}
\partial_{t}P^{\var,\tau}+u^{\var,\tau}\cdot\nabla P^{\var,\tau}&=-\big(\gamma_{+}\alpha^{\var,\tau}_{+}P_{+}(\rho_{+}^{\var,\tau})+\gamma_{-}\alpha^{\var,\tau}_{-}P_{-}(\rho_{-}^{\var,\tau}) \big) \div u^{\var,\tau}\\
    &\quad\quad-\alpha^{\var,\tau}_{+}\alpha^{\var,\tau}_{-}\big( (\gamma_{+}-1)P_{+}(\rho_{+}^{\var,\tau})-(\gamma_{-}-1)P_{-}(\rho_{-}^{\var,\tau}) \big) \frac{P_{+}(\rho_{+}^{\var,\tau})-P_{-}(\rho_{-}^{\var,\tau})}{\var}.
\end{aligned}
\end{equation}
And the equation of $P^{\tau}$ reads
\begin{equation}\label{Ptau}
\begin{aligned}
\partial_{t}P^{\tau}+u^{\tau}\cdot\nabla P^{\tau}+\frac{\gamma_{+}\gamma_{-}P^{\tau}}{\gamma_{+}\alpha_{-}^{\tau}+\gamma_{-}\alpha_{+}^{\tau}}\div u^{\tau}=0.
\end{aligned}
\end{equation}
However it is not suitable to estimate $\delta P$ directly from \eqref{Pvar}-\eqref{Ptau} as the decay rate of $P_{+}(\rho_{+}^{\var,\tau})-P_{-}(\rho_{-}^{\var,\tau})$ can not be faster than $\varepsilon$ in view of \eqref{uniform1}. To overcome this difficulty, we introduce an auxiliary  unknown  $Q^{\var,\tau}:=P^{\var,\tau}-\Gamma_{1}^{\var,\tau}\big(P_{+}(\rho_{+}^{\var,\tau})-P_{-}(\rho_{-}^{\var,\tau})\big)$ which verifies
\begin{equation}\label{eqQ}
    \begin{aligned}
    \partial_{t} Q^{\var,\tau}+u^{\var,\tau}\cdot\nabla  Q^{\var,\tau}+\Gamma_{2}^{\var,\tau} \div u^{\var,\tau}&=-\Gamma_{3}^{\var,\tau}(P_{+}(\rho_{+}^{\var,\tau})-P_{-}(\rho_{-}^{\var,\tau}) \big) \div u^{\var,\tau}\\
    &\quad+(\partial_{t}\Gamma_{1}^{\var,\tau}+u^{\var,\tau}\cdot\nabla \Gamma_{1}^{\var,\tau})\big(P_{+}(\rho_{+}^{\var,\tau})-P_{-}(\rho_{-}^{\var,\tau})\big),
    \end{aligned}
\end{equation}
with
\begin{equation}\nonumber
\left\{
    \begin{aligned}
    &\Gamma_{1}^{\var,\tau}:=\frac{\alpha^{\var,\tau}_{+}\alpha^{\var,\tau}_{-}\big( (\gamma_{+}-1)P_{+}(\rho_{+}^{\var,\tau})-(\gamma_{-}-1)P_{-}(\rho_{-}^{\var,\tau}) \big) }{ \gamma_{+}\alpha^{\var,\tau}_{-}P_{+}(\rho_{+}^{\var,\tau})+\gamma_{-}\alpha^{\var,\tau}_{+}P_{-}(\rho_{-}^{\var,\tau})},\\
&\Gamma_{2}^{\var,\tau}:=\frac{\gamma_{+}\gamma_{-}P_{+}(\rho_{+}^{\var,\tau})P_{-}(\rho_{-}^{\var,\tau})}{\gamma_{+}\alpha_{-}P_{+}(\rho_{+}^{\var,\tau})+\gamma_{-}\alpha^{\var,\tau}_{+}P_{-}(\rho_{-}^{\var,\tau})},\\
&\Gamma_{3}^{\var,\tau}:=\frac{\alpha_{+}^{\var,\tau}\alpha_{-}^{\var,\tau}\big(\gamma_{+}P_{+}(\rho_{+}^{\var,\tau})-\gamma_{-}P_{-}(\rho_{-}^{\var,\tau})\big)}{\gamma_{+}\alpha^{\var,\tau}_{-}P_{+}(\rho_{+}^{\var,\tau})+\gamma_{-}\alpha^{\var,\tau}_{+}P_{-}(\rho_{-}^{\var,\tau})}.
\end{aligned}
\right.
\end{equation}
With this formulation, it will be possible to derive the $\mathcal{O}(\varepsilon)$ bounds for the last term on the right-hand side of \eqref{eqQ}. Define
\begin{equation}\label{Qvar}
    \begin{aligned}
&\delta Q:=P^{\var,\tau}-P^{\tau}-\Gamma_{1}^{\var,\tau}\big(P_{+}(\rho_{+}^{\var,\tau})-P_{-}(\rho_{-}^{\var,\tau})\big).
    \end{aligned}
\end{equation}
The next lemma implies that to estimate $(\delta \alpha_{\pm}, \delta \rho_{\pm}, \delta P_{\pm}, \delta P)$, it is sufficient to control $(\delta Y,\delta Q,P_{+}(\rho_{+}^{\var,\tau})-P_{-}(\rho_{-}^{\var,\tau}))$.

\begin{lemma}\label{lemmarecover} For $d\geq3$, under the assumption \eqref{errora1}, the following estimates follow{\rm:}
\begin{equation}\label{estimatedeltaalpha}
\left\{
\begin{aligned}
&\|(\delta\alpha_{\pm},\delta \rho_{\pm}, \delta\rho)\|_{\widetilde{L}^{\infty}_{t}(\dot{B}^{\frac{d}{2}-1}\cap\dot{B}^{\frac{d}{2}})} \lesssim  \|\big(\delta Y,\delta Q,P_{+}(\rho_{+}^{\var,\tau})-P_{-}(\rho_{-}^{\var,\tau}) \big)\|_{\widetilde{L}^{\infty}_{t}(\dot{B}^{\frac{d}{2}-2}\cap\dot{B}^{\frac{d}{2}-1})},\\
&\|\delta \rho_{\pm}\|_{\widetilde{L}^2_{t}(\dot{B}^{\frac{d}{2}-1})}\lesssim \|\big(\delta Q,P_{+}(\rho_{+}^{\var,\tau})-P_{-}(\rho_{-}^{\var,\tau}) \big)\|_{\widetilde{L}^{2}_{t}(\dot{B}^{\frac{d}{2}-1})}.
\end{aligned}
\right.
\end{equation}
\end{lemma}

\begin{proof}
Due to \eqref{deltaY} and
\begin{equation}\label{deltarho0}
    \begin{aligned}
\delta\rho=(\rho_{+}^{\var,\tau}-\rho_{-}^{\var,\tau})\delta\alpha_{+}+\alpha_{+}^{\tau}\delta \rho_{+}+\alpha_{-}^{\tau}\delta\rho_{-},
    \end{aligned}
\end{equation}
it holds that
\begin{equation}\nonumber
    \begin{aligned}
    \delta Y&=\frac{1}{\rho^{\var,\tau}\rho^{\tau}} \big( \rho_{+}^{\var,\tau}\rho^{\tau} \delta\alpha^{\var,\tau}_{+}+\alpha^{\tau}_{+} \rho^{\tau}\delta\rho_{+}-\alpha_{+}^{\tau}\rho_{+}^{\tau}\delta\rho \big)\\
    &=\frac{1}{\rho^{\var,\tau}\rho^{\tau}} \big( (\alpha^{\tau}_{-}\rho_{+}^{\var,\tau}\rho_{-}^{\tau}+\alpha^{\tau}_{+}\rho_{+}^{\tau}\rho_{-}^{\var,\tau})\delta \alpha_{+} +\alpha_{+}^{\tau}\alpha_{-}^{\tau}\rho_{-}^{\tau}\delta\rho_{+}-\alpha_{+}^{\tau}\alpha_{-}^{\tau}\rho_{+}^{\tau}\delta \rho_{-} \big).
    \end{aligned}
\end{equation}
This implies
\begin{equation}\label{deltaalpha0}
    \begin{aligned}
    &\delta \alpha_{+}=\frac{1}{\alpha^{\tau}_{-}\rho_{+}^{\var,\tau}\rho_{-}^{\tau}+\alpha^{\tau}_{+}\rho_{+}^{\tau}\rho_{-}^{\var,\tau}} \big( \rho^{\var,\tau}\rho^{\tau} \delta Y-\alpha_{+}^{\tau}\alpha_{-}^{\tau}\rho_{-}^{\tau}\delta\rho_{+}+\alpha_{+}^{\tau}\alpha_{-}^{\tau}\rho_{+}^{\tau}\delta \rho_{-}\big).
    \end{aligned}
\end{equation}
Inserting \eqref{deltaalpha0} into \eqref{deltarho0}, we have
\begin{equation}\label{deltarho00}
    \begin{aligned}
    \delta\rho&=\frac{\rho_{+}^{\var,\tau}-\rho_{-}^{\var,\tau}}{\alpha^{\tau}_{-}\rho_{+}^{\var,\tau}\rho_{-}^{\tau}+\alpha^{\tau}_{+}\rho_{+}^{\tau}\rho_{-}^{\var,\tau}} \big( \rho^{\var,\tau}\rho^{\tau} \delta Y-\alpha_{+}^{\tau}\alpha_{-}^{\tau}\rho_{-}^{\tau}\delta\rho_{+}+\alpha_{+}^{\tau}\alpha_{-}^{\tau}\rho_{+}^{\tau}\delta \rho_{-}\big)\\
    &\quad+\alpha_{+}^{\tau}\delta \rho_{+}+\alpha_{-}^{\tau}\delta\rho_{-}.
    \end{aligned}
\end{equation}
Moreover, we have
\begin{equation}\label{deltaPpmrhopm}
\begin{aligned}
&\delta P_{\pm}=\delta \rho_{\pm}\int_{0}^{1} P_{\pm}'(\theta \rho^{\var,\tau}_{\pm}+(1-\theta)\rho^{\tau}_{\pm})d\theta\quad \andf \quad \delta P=\alpha_{+}^{\var,\tau}(P_{+}^{\var,\tau}-P_{-}^{\var,\tau})+\delta P_{-}.
\end{aligned}
\end{equation}
Using the previous uniform estimates \eqref{uniform1} and \eqref{uniformK}, the product laws \eqref{uv1}-\eqref{uv2} and the composition estimates \eqref{F1}-\eqref{F11}, for some constant states $\bar{\Gamma_{i}}>0$ ($i=1,2,3$), we have
\begin{equation}\label{Gammaes}
\begin{aligned}
& \sum_{i=1}^{3}(\|\Gamma_{i}-\bar{\Gamma_{i}}\|_{\widetilde{L}^{\infty}_{t}(\dot{B}^{\frac{d}{2}-1}\cap\dot{B}^{\frac{d}{2}+1})}+\|\partial_{t}\Gamma_{i}\|_{L^1_{t}(\dot{B}^{\frac{d}{2}})})={o}(1).
\end{aligned}
\end{equation}
Therefore, \eqref{estimatedeltaalpha} follows from \eqref{Qvar}, \eqref{deltaalpha0}-\eqref{Gammaes}, the product laws \eqref{uv1}-\eqref{uv2} and the fact $\delta \alpha_{+}=-\delta \alpha_{+}$. 
\end{proof}

The next lemma pertains to $\mathcal{O}(\sqrt{\varepsilon\tau})$ bounds for $P_{+}(\rho_{+}^{\var,\tau})-P_{-}(\rho_{-}^{\var,\tau})$, which leads to the convergence rate $\sqrt{\varepsilon\tau}$.
\begin{lemma} For $d\geq3$, under the assumption \eqref{errora1}, the following estimate is valid{\rm:}
\begin{equation}\label{deltaP+P-1}
\begin{aligned}
&\|P_{+}(\rho_{+}^{\var,\tau})-P_{-}(\rho_{-}^{\var,\tau})\|_{\widetilde{L}^{\infty}_{t}(\dot{B}^{\frac{d}{2}-2}\cap\dot{B}^{\frac{d}{2}-1} )}\\
&\quad\quad+\frac{1}{\sqrt{\var}}\|P_{+}(\rho_{+}^{\var,\tau})-P_{-}(\rho_{-}^{\var,\tau})\|_{\widetilde{L}^{2}_{t}(\dot{B}^{\frac{d}{2}-1}\cap\dot{B}^{\frac{d}{2}-1} )}\lesssim \sqrt{\var \tau}.
\end{aligned}
\end{equation}
\end{lemma}

\begin{proof}
It is easy to verify from \eqref{BN} that $P_{+}(\rho_{+}^{\var,\tau})-P_{-}(\rho_{-}^{\var,\tau})$ satisfies the damped equation
\begin{equation}\label{P+P-}
    \begin{aligned}
    &\partial_{t}\big(P_{+}(\rho_{+}^{\var,\tau})-P_{-}(\rho_{-}^{\var,\tau})\big)+u^{\var,\tau}\cdot\nabla\big(P_{+}(\rho_{+}^{\var,\tau})-P_{-}(\rho_{-}^{\var,\tau})\big)\\
    &\quad\quad+\frac{c_{*}}{\var} \big(P_{+}(\rho_{+}^{\var,\tau})-P_{-}(\rho_{-}^{\var,\tau})\big)\\
    &\quad= ((\gamma_{+}\alpha^{\var,\tau}_{-}P_{+}(\rho_{+}^{\var,\tau})+\gamma_{-}\alpha^{\var,\tau}_{+}P_{-}(\rho_{-}^{\var,\tau}))-c_{*})\frac{1}{\var}\big( P_{+}(\rho_{+}^{\var,\tau})-P_{-}(\rho_{-}^{\var,\tau}) \big) \\
    &\quad\quad-\big(\gamma_{+}P_{+}(\rho_{+}^{\var,\tau})-\gamma_{-}P_{-}(\rho_{-}^{\var,\tau})\big)\div u^{\var,\tau}:=W_1+W_2,
    \end{aligned}
\end{equation}
with $c_{*}:=(\gamma_{+}\bar{\alpha}_{-}+\gamma_{-}\bar{\alpha}_{+})\bar{P}.$ Thence the $L^2$-in-time type estimate in Lemma \ref{maximaldamped} for the damped transport equation \eqref{P+P-} leads to
\begin{equation}\nonumber
\begin{aligned}
&\|P_{+}(\rho_{+}^{\var,\tau})-P_{-}(\rho_{-}^{\var,\tau})\|_{\widetilde{L}^{\infty}_{t}(\dot{B}^{\frac{d}{2}-2}\cap\dot{B}^{\frac{d}{2}-1} )}+\frac{1}{\sqrt{\var}}\|P_{+}(\rho_{+}^{\var,\tau})-P_{-}(\rho_{-}^{\var,\tau})\|_{\widetilde{L}^{2}_{t}(\dot{B}^{\frac{d}{2}-2}\cap\dot{B}^{\frac{d}{2}-1} ) }\\
&\quad\lesssim {\rm{exp}}\Big(\|u^{\var,\tau}\|_{L^1_{t}(\dot{B}^{\frac{d}{2}+1})}\Big) \Big(\sqrt{\var\tau}+\sqrt{\var}\|(W_{1},W_{2})\|_{\widetilde{L}^{2}_{t}(\dot{B}^{\frac{d}{2}-2}\cap\dot{B}^{\frac{d}{2}-1} )} \Big).
\end{aligned}
\end{equation}
 By \eqref{uniform1} and \eqref{uv2}, there holds
\begin{equation}\nonumber
\begin{aligned}
\sqrt{\var}\|W_{1}\|_{\widetilde{L}^{2}_{t}(\dot{B}^{\frac{d}{2}-2}\cap\dot{B}^{\frac{d}{2}-1} )}&\lesssim\|(\alpha_{\pm}^{\var,\tau}-\bar{\alpha}_{\pm},\rho_{\pm}^{\var,\tau}-\bar{\rho}_{\pm})\|_{\widetilde{L}^{\infty}_{t}(\dot{B}^{\frac{d}{2}})}\frac{1}{\sqrt{\var}}\|P_{+}(\rho_{+}^{\var,\tau})-P_{-}(\rho_{-}^{\var,\tau})\|_{\widetilde{L}^{2}_{t}(\dot{B}^{\frac{d}{2}-2}\cap\dot{B}^{\frac{d}{2}-1} )}\\
&\lesssim {o}(1) \frac{1}{\sqrt{\var}}\|P_{+}(\rho_{+}^{\var,\tau})-P_{-}(\rho_{-}^{\var,\tau})\|_{\widetilde{L}^{2}_{t}(\dot{B}^{\frac{d}{2}-2}\cap\dot{B}^{\frac{d}{2}-1} )},
\end{aligned}
\end{equation}
and
\begin{equation}\nonumber
\begin{aligned}
&\sqrt{\var} \|W_{2}\|_{\widetilde{L}^{2}_{t}(\dot{B}^{\frac{d}{2}-2}\cap \dot{B}^{\frac{d}{2}-1} )}\lesssim \sqrt{\var}\|u^{\var,\tau}\|_{\widetilde{L}^{2}_{t}(\dot{B}^{\frac{d}{2}-1}\cap\dot{B}^{\frac{d}{2}} )}\lesssim \sqrt{\var \tau}.
\end{aligned}
\end{equation}
Therefore,  we gain \eqref{deltaP+P-1}.
\end{proof}

We are going to estimate $(\delta Q,\delta u)$. By virtue of \eqref{BN}, \eqref{K}  and \eqref{Ptau}-\eqref{eqQ}, $(\delta Q,\delta u)$ satisfies the following equations of damped Euler type with rough coefficients:
\begin{equation}\label{deltaK}
\left\{
\begin{aligned}
&\partial_{t} \delta Q+u^{\var,\tau}\cdot\nabla  \delta Q+\Gamma_{2}^{\var,\tau}  \div \delta u=\delta F_1 ,\\
&\partial_{t} \delta u+ u^{\var,\tau}\cdot \nabla \delta u+\frac{1}{\bar{\rho}}\nabla \delta Q+(\frac{1}{\rho^{\var,\tau}}-\frac{1}{\rho^{\tau}}) \nabla P^{\tau}+\frac{ \delta u}{\tau}=\delta F_2,
\end{aligned}
\right.
\end{equation}
with the nonlinear terms
\begin{equation}\nonumber
\left\{
\begin{aligned}
\delta F_1&=-\delta u\cdot\nabla P^{\tau}-\Big( \Gamma_{2}^{\var,\tau}-\frac{\gamma_{+}\gamma_{-}P^{\tau}}{\gamma_{+}\alpha_{-}^{\tau}+\gamma_{-}\alpha_{+}^{\tau}}\Big)\div u^{\tau}\\
&\quad-\Gamma_{3}^{\var,\tau}\big(P_{+}(\rho_{+}^{\var,\tau})-P_{-}(\rho_{-}^{\var,\tau})\big) \div u^{\var,\tau}+(\partial_{t}\Gamma_{1}^{\var,\tau}+u^{\var,\tau}\cdot\nabla \Gamma_{1}^{\var,\tau})\big(P_{+}(\rho_{+}^{\var,\tau})-P_{-}(\rho_{-}^{\var,\tau})\big),\\
\delta F_2&:=-\delta u\cdot\nabla u^{\tau}-\frac{1}{\rho^{\var,\tau}}\nabla \Big(\Gamma_{1}^{\var,\tau}\big(P_{+}(\rho_{+}^{\var,\tau})-P_{-}(\rho_{-}^{\var,\tau})\big)\Big).
\end{aligned}
\right.
\end{equation} 
In order to establish the uniform-in-$\tau$ convergence estimates, we follow the ideas in Section \ref{section3} to overcome the issue caused by the overdamping phenomenon.

\begin{lemma}\label{lemmadeltaQu}
Let $d\geq3$, $0<\var\leq\tau\leq1$, and the threshold $J_{\tau}$ be given by \eqref{J}. Then under the assumption \eqref{errora1}, there holds
 \begin{equation}\label{errordeltaQu}
    \begin{aligned}
    &\|(\delta Q,\delta u)\|_{\widetilde{L}^{\infty}_{t}(\dot{B}^{\frac{d}{2}-2}\cap\dot{B}^{\frac{d}{2}-1})}+ \tau\|\delta Q\|_{L^1_{t}(\dot{B}^{\frac{d}{2}})}^{\ell}+\|\delta Q\|_{L^1_{t}(\dot{B}^{\frac{d}{2}-1})}^{h}\\
&\quad\quad+ \sqrt{\tau}\| \delta Q\|_{\widetilde{L}^{2}_{t}(\dot{B}^{\frac{d}{2}-1})}+ \frac{1}{\sqrt{\tau}} \|\delta u\|_{\widetilde{L}^{2}_{t}(\dot{B}^{\frac{d}{2}-2}\cap\dot{B}^{\frac{d}{2}-1})}+\| \delta u\|_{L^1_t(\dot{B}^{\frac{d}{2}-1})}\\
&\quad\lesssim \sqrt{\var\tau}+{o}(1)\|\delta Y\|_{\widetilde{L}^{\infty}_{t}(\dot{B}^{\frac{d}{2}-2}\cap\dot{B}^{\frac{d}{2}-1})}.
\end{aligned}
\end{equation}
\end{lemma}

\begin{proof}
As in Section \ref{section3}, we split the proof into three parts:
\begin{itemize}
\item \textbf{Step 1: $\dot{B}^{\frac{d}{2}-2}$-estimates in low frequencies}
\end{itemize}
We introduce the new damped mode (effective flux)
$$
\delta z:=\delta u+\frac{\tau}{\rho^{\var,\tau}}\nabla \delta Q+\tau(\frac{1}{\rho^{\var,\tau}}-\frac{1}{\rho^{\tau}}) \nabla P^{\tau},
$$
so that \eqref{deltaK} is rewritten as
\begin{equation}\label{deltalowe}
\left\{
    \begin{aligned}
    &\partial_{t}\delta Q-\frac{\bar{\Gamma}_{2}\tau}{\bar{\rho}} \Delta \delta Q=-\bar{\Gamma}_{2} \div z+\delta F_3,\\
    &\partial_{t}\delta z+\frac{\delta z}{\tau}=\frac{\tau}{\bar{\rho}}\nabla(\frac{\bar{\Gamma}_{2}\tau}{\bar{\rho}} \Delta \delta Q-\bar{\Gamma}_{2} \div z)+\delta F_4,
    \end{aligned}
    \right.
\end{equation}
where $\bar{\Gamma}_{2}>0$ is the constant state of $\Gamma^{\var,\tau}_{2}$, and $\delta F_{i}$ ($i=3,4$) is defined by
\begin{equation}\nonumber
\left\{
    \begin{aligned}
\delta F_{3}:&=-u^{\var,\tau}\cdot\nabla  \delta Q-(\Gamma_{2}^{\var,\tau}-\bar{\Gamma}_{2})\div\delta u\\
&\quad+\bar{\Gamma}_{2}\tau\div \Big((\frac{1}{\rho^{\var,\tau}}-\frac{1}{\bar{\rho}}) \nabla \delta Q+(\frac{1}{\rho^{\var,\tau}}-\frac{1}{\rho^{\tau}}) \nabla P^{\tau} \Big)+\delta F_1,\\
\delta F_{4}:&=-u^{\var,\tau}\cdot \nabla \delta u+\frac{\tau}{\bar{\rho}}\nabla \delta F_3\\
&\quad+\tau(\frac{1}{\rho^{\var,\tau}}-\frac{1}{\bar{\rho}})\nabla\partial_{t}\delta Q-\tau  \partial_{t}(\frac{1} {\rho^{\var,\tau}})\nabla \delta Q+\tau\partial_{t}\big((\frac{1}{\rho^{\var,\tau}}-\frac{1}{\rho^{\tau}}) \nabla P^{\tau}\big)+\delta F_2.
    \end{aligned}
    \right.
\end{equation}
Then by similar arguments used to get \eqref{rd21}-\eqref{wrzd21}, we deduce from \eqref{deltalowe} and the choice $\eqref{J}$ of the threshold $J_{\tau}$ that
\begin{equation}\label{deltalowee}
    \begin{aligned}
&\|(\delta Q,\delta z)\|_{\widetilde{L}^{\infty}_{t}(\dot{B}^{\frac{d}{2}-2})}^{\ell}+\tau\|\delta Q\|_{L^{1}_{t}(\dot{B}^{\frac{d}{2}})}^{\ell}+\tau\|\partial_{t}\delta Q\|_{L^{1}_{t}(\dot{B}^{\frac{d}{2}-2})}^{\ell}\\
&\quad\quad+\sqrt{\tau}\|\delta Q\|_{\widetilde{L}^{2}_{t}(\dot{B}^{\frac{d}{2}-1})}^{\ell}+\frac{1}{\tau}\|\delta z\|_{L^{1}_{t}(\dot{B}^{\frac{d}{2}-2})}^{\ell}\\
&\quad\lesssim \sqrt{\var\tau}+\|(\delta F_{3}, \delta F_{4})\|_{L^1_t(\dot{B}^{\frac{d}{2}-2})}^{\ell}.
    \end{aligned}
\end{equation}

We first estimate $\delta F_{3}$. From \eqref{uniform1}, \eqref{Gammaes} and the product map $\dot{B}^{\frac{d}{2}-2}\times \dot{B}^{\frac{d}{2}}\rightarrow \dot{B}^{\frac{d}{2}-2}$ for $d\geq3$, one obtains
\begin{equation}\label{unablaldetalQ}
    \begin{aligned}
    &\|u^{\var,\tau}\cdot\nabla  \delta Q +(\Gamma_{2}^{\var,\tau}-\bar{\Gamma}_{2})\div\delta u\|_{L^1_t(\dot{B}^{\frac{d}{2}-2})}\\
    &\quad\lesssim\frac{1}{\sqrt{\tau}}\|u^{\var,\tau}\|_{\widetilde{L}^{2}_{t}(\dot{B}^{\frac{d}{2}})}\sqrt{\tau}\|\delta Q\|_{\widetilde{L}^{2}_{t}(\dot{B}^{\frac{d}{2}-1})}+\|\Gamma_{2}^{\var,\tau}-\bar{\Gamma}_{2}\|_{\widetilde{L}^{\infty}_{t}(\dot{B}^{\frac{d}{2}})}\|\delta u\|_{L^1_t(\dot{B}^{\frac{d}{2}-1})}\\
    &\quad\lesssim {o}(1)\Big( \sqrt{\tau}\|\delta Q\|_{\widetilde{L}^{2}_{t}(\dot{B}^{\frac{d}{2}-1})}+\|\delta u\|_{L^1_t(\dot{B}^{\frac{d}{2}-1})}\Big).
    \end{aligned}
\end{equation}
By virtue of \eqref{uniform1}, \eqref{lhl}, \eqref{estimatedeltaalpha}, \eqref{deltaP+P-1} and \eqref{uv2}, we also have
\begin{equation}\nonumber
    \begin{aligned}
&\tau\Big\|\div\big((\frac{1}{\rho^{\var,\tau}}-\frac{1}{\rho^{\tau}}) \nabla P^{\tau}\big)\Big\|_{L^1_t(\dot{B}^{\frac{d}{2}-2})}^{\ell}\\
&\quad\lesssim \|\delta \rho\|_{\widetilde{L}^{\infty}_{t}(\dot{B}^{\frac{d}{2}-1})}\tau\|P^{\tau}-\bar{P}\|_{L^1_{t}(\dot{B}^{\frac{d}{2}+1})}\\
&\quad\lesssim {o}(1)\|(\delta Y,\delta Q)\|_{\widetilde{L}^{\infty}_{t}(\dot{B}^{\frac{d}{2}-1})}+\sqrt{\var\tau}.
\end{aligned}
\end{equation}
As in the previous analysis \eqref{wrzd213}, the tricky nonlinear term can be estimated as
\begin{equation}\nonumber
    \begin{aligned}
&\tau\Big\|\div \Big((\frac{1}{\rho^{\var,\tau}}-\frac{1}{\bar{\rho}}) \nabla \delta Q\Big)\Big\|_{L^1_t(\dot{B}^{\frac{d}{2}-2})}^{\ell}\\
&\quad\lesssim \tau\Big\|(\frac{1}{\rho^{\var,\tau}}-\frac{1}{\bar{\rho}}) \nabla \delta Q^{\ell}\Big\|_{L^1_t(\dot{B}^{\frac{d}{2}-1})}^{\ell}+\Big\|(\frac{1}{\rho^{\var,\tau}}-\frac{1}{\bar{\rho}}) \nabla \delta Q^{h}\Big\|_{L^1_t(\dot{B}^{\frac{d}{2}-2})}^{\ell}\\
&\quad\lesssim {o}(1) \Big(\tau\|\delta Q\|_{L^1_{t}(\dot{B}^{\frac{d}{2}})}^{\ell}+\|\delta Q\|_{L^1_{t}(\dot{B}^{\frac{d}{2}-1})}^{h}\Big).
\end{aligned}
\end{equation}
Similarly, one can show
\begin{align}
    \|\delta F_{1}\|_{L^1_t(\dot{B}^{\frac{d}{2}-2})}&\lesssim \frac{1}{\sqrt{\tau}} \|\delta u\|_{\widetilde{L}^{2}_{t}(\dot{B}^{\frac{d}{2}-2})} \sqrt{\tau} \|Q^{\tau}\|_{\widetilde{L}^{2}_{t}(\dot{B}^{\frac{d}{2}})}\nonumber\\
    &\quad+\Big\|\Big( \Gamma_{2}^{\var,\tau}-\frac{\gamma_{+}\gamma_{-}P^{\tau}}{\gamma_{+}\alpha_{-}^{\tau}+\gamma_{-}\alpha_{+}^{\tau}}\Big)\Big\|_{\widetilde{L}^{\infty}_{t}(\dot{B}^{\frac{d}{2}-2})}\|u^{\tau}\|_{L^1_{t}(\dot{B}^{\frac{d}{2}+1})}\nonumber\\
    &\quad+ \|P_{+}(\rho_{+}^{\var,\tau})-P_{-}(\rho_{-}^{\var,\tau})\|_{\widetilde{L}^{\infty}_{t}(\dot{B}^{\frac{d}{2}-2})}\|u^{\var,\tau}\|_{L^1_{t}(\dot{B}^{\frac{d}{2}+1})} \nonumber\\
    &\quad+(\|\partial_{t}\Gamma_{1}^{\var,\tau}\|_{L^1_{t}(\dot{B}^{\frac{d}{2}})}+ \|u^{\var,\tau}\|_{L^1_{t}(\dot{B}^{\frac{d}{2}})}\|\nabla \Gamma_{1}^{\var,\tau}\|_{\widetilde{L}^{\infty}_{t}(\dot{B}^{\frac{d}{2}})})\|P_{+}(\rho_{+}^{\var,\tau})-P_{-}(\rho_{-}^{\var,\tau})\|_{\widetilde{L}^{\infty}_{t}(\dot{B}^{\frac{d}{2}-2})}\nonumber,
    \end{align}
    which implies
\begin{equation}\label{deltaF1}
    \begin{aligned}
      \|\delta F_{1}\|_{L^1_t(\dot{B}^{\frac{d}{2}-2})}&\lesssim{o}(1) \Big(\|(\delta Y,\delta Q)\|_{\widetilde{L}^{\infty}_{t}(\dot{B}^{\frac{d}{2}-2})}+ \frac{1}{\sqrt{\tau}} \|\delta u\|_{\widetilde{L}^{2}_{t}(\dot{B}^{\frac{d}{2}-2})}\Big)+\sqrt{\var\tau}.
      \end{aligned}
\end{equation}
Therefore, we have
\begin{equation}\label{deltaF3}
    \begin{aligned}
    \|\delta F_{3}\|_{L^1_t(\dot{B}^{\frac{d}{2}-2})}^{\ell}&\lesssim {o}(1) \Big( \|(\delta Y,\delta Q)\|_{\widetilde{L}^{\infty}_{t}(\dot{B}^{\frac{d}{2}-2})}+ \| \delta Q\|_{L^{1}_{t}(\dot{B}^{\frac{d}{2}})}+\sqrt{\tau}\|\delta Q\|_{\widetilde{L}^{2}_{t}(\dot{B}^{\frac{d}{2}-1} )}\\
    &\quad\quad+ \frac{1}{\sqrt{\tau}} \|\delta u\|_{\widetilde{L}^{2}_{t}(\dot{B}^{\frac{d}{2}-2})}+\| \delta u\|_{L^1_t(\dot{B}^{\frac{d}{2}-1})} \Big)+ \sqrt{\var\tau}.
    \end{aligned}
\end{equation}

We turn to the estimate of $\delta F_{4}$. Similar calculations give
\begin{equation}\nonumber
    \begin{aligned}
    \Big\|-u^{\var,\tau}\cdot \nabla \delta u+\frac{\tau}{\rho^{\var,\tau}}\nabla \delta F_3\Big\|_{L^1_t(\dot{B}^{\frac{d}{2}-2})}&\lesssim {o}(1)\|\delta u\|_{L^1_{t}(\dot{B}^{\frac{d}{2}-1} )}+\|\delta F_{3}\|_{L^1_t(\dot{B}^{\frac{d}{2}-2})},
    \end{aligned}
\end{equation}
and
\begin{equation}\nonumber
    \begin{aligned}
\Big\|\tau  \partial_{t}(\frac{1} {\rho^{\var,\tau}})\nabla \delta Q\Big\|_{L^1_t(\dot{B}^{\frac{d}{2}-2})}&\lesssim \|\partial_{t}\rho^{\var,\tau}\|_{L^1_t(\dot{B}^{\frac{d}{2}})} \|\delta Q\|_{\widetilde{L}^{\infty}_{t}(\dot{B}^{\frac{d}{2}-1})}\lesssim {o}(1) \|\delta Q\|_{\widetilde{L}^{\infty}_{t}(\dot{B}^{\frac{d}{2}-1})}.
    \end{aligned}
\end{equation}
For the third difficult term in $\delta F_{4}$, we apply \eqref{lhl}, the product law \eqref{uv2} for $d\geq3$ and the fact
$$\tau(\frac{1}{\rho^{\var,\tau}}-\frac{1}{\bar{\rho}})\nabla\partial_{t}\delta Q=\tau\nabla\big( (\frac{1}{\rho^{\var,\tau}}-\frac{1}{\bar{\rho}})\partial_{t}\delta Q \big)-\tau\nabla (\frac{1}{\rho^{\var,\tau}}-\frac{1}{\bar{\rho}})\partial_{t}\delta Q
$$
to have
\begin{equation}\nonumber
    \begin{aligned}
    \Big\| \tau(\frac{1}{\rho^{\var,\tau}}-\frac{1}{\bar{\rho}})\nabla\partial_{t}\delta Q\Big\|_{L^1_t(\dot{B}^{\frac{d}{2}-2})}^{\ell}&\lesssim \tau\Big\| \nabla \big( (\frac{1}{\rho^{\var,\tau}}-\frac{1}{\bar{\rho}})\partial_{t}\delta Q \big)\Big\|_{L^1_t(\dot{B}^{\frac{d}{2}-2})}^{\ell}+\tau\Big\| \nabla (\frac{1}{\rho^{\var,\tau}}-\frac{1}{\bar{\rho}})\partial_{t}\delta Q\Big\|_{L^1_t(\dot{B}^{\frac{d}{2}-2})}^{\ell}\\
        &\lesssim \Big\|  (\frac{1}{\rho^{\var,\tau}}-\frac{1}{\bar{\rho}})\partial_{t}\delta Q\Big\|_{L^1_t(\dot{B}^{\frac{d}{2}-2})}^{\ell}+\Big\| \nabla (\frac{1}{\rho^{\var,\tau}}-\frac{1}{\bar{\rho}})\partial_{t}\delta Q\Big\|_{L^1_t(\dot{B}^{\frac{d}{2}-2})}^{\ell}\\
&\lesssim \|\rho^{\var,\tau}-\bar{\rho}\|_{\widetilde{L}^{\infty}_{t}(\dot{B}^{\frac{d}{2}}\cap\dot{B}^{\frac{d}{2}+1})} \|\partial_{t}\delta Q\|_{L^1_t(\dot{B}^{\frac{d}{2}-2})}\\
&\lesssim o(1)\|\partial_{t}\delta Q\|_{L^1_t(\dot{B}^{\frac{d}{2}-2})}.
        \end{aligned}
\end{equation}
Similarly, the term $\delta F_{2}$ can be easily estimated as follows:
\begin{equation}\nonumber
    \begin{aligned}
    \|\delta F_{2}\|_{L^1_t(\dot{B}^{\frac{d}{2}-2})}&\lesssim \|\delta u\|_{L^1_{t}(\dot{B}^{\frac{d}{2}-1})}\|u^{\tau}\|_{\widetilde{L}^{\infty}_{t}(\dot{B}^{\frac{d}{2}})}+\|P_{+}(\rho_{+}^{\var,\tau})-P_{-}(\rho_{-}^{\var,\tau})\|_{L^1_t(\dot{B}^{\frac{d}{2}-1})}\\
    &\lesssim {o}(1)\| \delta u\|_{L^1_t(\dot{B}^{\frac{d}{2}-1})}+\sqrt{\var\tau}+\var\\
    &\lesssim {o}(1)\| \delta u\|_{L^1_t(\dot{B}^{\frac{d}{2}-1})}+\sqrt{\var\tau}.
    \end{aligned}
\end{equation}
To bound the term  $\tau\partial_{t}\big((\frac{1}{\rho^{\var,\tau}}-\frac{1}{\rho^{\tau}}) \nabla P^{\tau}\big)$, noticing that
$$
\partial_{t}\delta\rho=-\div \big(\delta \rho u^{\var,\tau}+\rho^{\tau}\delta u),
$$
we use \eqref{uniform1}, \eqref{estimatedeltaalpha}, \eqref{deltaP+P-1} and \eqref{F1}-\eqref{F11} that
\begin{equation}\nonumber
    \begin{aligned}
   &\Big\| \tau\partial_{t}\big((\frac{1}{\rho^{\var,\tau}}-\frac{1}{\rho^{\tau}}) \nabla P^{\tau}\big)\Big\|_{L^1_t(\dot{B}^{\frac{d}{2}-2})}\\
   &\quad\lesssim  \|\partial_{t}\delta \rho\|_{L^1_t(\dot{B}^{\frac{d}{2}-2})}\|\nabla P^{\tau}\|_{\widetilde{L}^{\infty}_{t}(\dot{B}^{\frac{d}{2}})}+\|\delta \rho\|_{\widetilde{L}^{\infty}_{t}(\dot{B}^{\frac{d}{2}-2})}\tau \|\nabla P^{\tau}\|_{L^1_t(\dot{B}^{\frac{d}{2}})}\\
   &\quad\lesssim {o}(1)\Big( \|(\delta Y,\delta Q)\|_{\widetilde{L}^{\infty}_{t}(\dot{B}^{\frac{d}{2}-2}\cap\dot{B}^{\frac{d}{2}-1})}+\|\delta u\|_{L^1_t(\dot{B}^{\frac{d}{2}-1})}\Big)+\sqrt{\var\tau}.
    \end{aligned}
\end{equation}
We thence get
\begin{equation}\label{deltaF4}
    \begin{aligned}
    \|\delta F_{3}\|_{L^1_t(\dot{B}^{\frac{d}{2}-2})}^{\ell}&\lesssim {o}(1) \Big(\|(\delta Y,\delta Q)\|_{\widetilde{L}^{\infty}_{t}(\dot{B}^{\frac{d}{2}-2})}+ \| \delta Q\|_{L^{1}_{t}(\dot{B}^{\frac{d}{2}})}+\sqrt{\tau}\|\delta Q\|_{\widetilde{L}^{2}_{t}(\dot{B}^{\frac{d}{2}-1} )}\\
    &\quad\quad+ \frac{1}{\sqrt{\tau}} \|\delta u\|_{\widetilde{L}^{2}_{t}(\dot{B}^{\frac{d}{2}-2})}+\| \delta u\|_{L^1_t(\dot{B}^{\frac{d}{2}-1})}\Big)+ \sqrt{\var\tau}.
    \end{aligned}
\end{equation}

Substituting the above estimates \eqref{deltaF3}-\eqref{deltaF4} into \eqref{deltalowee} and taking advantage of $\delta u=\delta z-\frac{\tau}{\rho^{\var,\tau}}\nabla \delta Q-\tau(\frac{1}{\rho^{\var,\tau}}-\frac{1}{\rho^{\tau}}) \nabla P^{\tau}$, we obtain
\begin{equation}\label{step1}
    \begin{aligned}
&\|(\delta Q,\delta z)\|_{\widetilde{L}^{\infty}_{t}(\dot{B}^{\frac{d}{2}-2})}^{\ell}+\tau\|\delta Q\|_{L^{1}_{t}(\dot{B}^{\frac{d}{2}})}^{\ell}+\sqrt{\tau}\|\delta Q\|_{\widetilde{L}^{2}_{t}(\dot{B}^{\frac{d}{2}-1})}^{\ell}+\frac{1}{\tau}\|\delta z\|_{L^{1}_{t}(\dot{B}^{\frac{d}{2}-2})}^{\ell}\\
&\quad\quad+\|\delta u\|_{\widetilde{L}^{\infty}_{t}(\dot{B}^{\frac{d}{2}-2})}^{\ell}+\frac{1}{\sqrt{\tau}} \|\delta u\|_{\widetilde{L}^{2}_{t}(\dot{B}^{\frac{d}{2}-2})}^{\ell}+\| \delta u\|_{L^1_t(\dot{B}^{\frac{d}{2}-1})}^{\ell}\\
&\quad\lesssim {o}(1) \Big(\|(\delta Y,\delta Q)\|_{\widetilde{L}^{\infty}_{t}(\dot{B}^{\frac{d}{2}-2}\cap\dot{B}^{\frac{d}{2}-1})}+ \tau\|\delta Q\|_{L^1_{t}(\dot{B}^{\frac{d}{2}})}^{\ell}\\
&\quad\quad+\|\delta Q\|_{L^1_{t}(\dot{B}^{\frac{d}{2}-1})}^{h}+ \sqrt{\tau}\| \delta Q\|_{\widetilde{L}^{2}_{t}(\dot{B}^{\frac{d}{2}-1})}\\
&\quad\quad+ \frac{1}{\sqrt{\tau}} \|\delta u\|_{\widetilde{L}^{2}_{t}(\dot{B}^{\frac{d}{2}-2})}+\| \delta u\|_{L^1_t(\dot{B}^{\frac{d}{2}-1})}+\|\partial_{t}\delta Q\|_{L^1_t(\dot{B}^{\frac{d}{2}-2})}\Big)+\sqrt{\var\tau}.
    \end{aligned}
\end{equation}

\begin{itemize}
\item \textbf{Step 2: $\dot{B}^{\frac{d}{2}-2}$-estimates of $(\delta Q,\delta u)$ in high frequencies}
\end{itemize}

Applying $\Delta_{j}$ to $\eqref{deltaK}$, one gets
\begin{equation}\nonumber
\left\{
\begin{aligned}
&\partial_{t} \dot{\Delta}_{j}\delta Q+u^{\var,\tau}\cdot\nabla  \dot{\Delta}_{j} \delta Q+\Gamma^{\var,\tau}_{2} \div \dot{\Delta}_{j} \delta u=\dot{\Delta}_{j}\delta F_1+\delta R_{1,j} ,\\
&\partial_{t} \dot{\Delta}_{j} \delta u+ u^{\var,\tau}\cdot \nabla \dot{\Delta}_{j}\delta u+\frac{1}{\rho^{\var,\tau}}\nabla \dot{\Delta}_{j}\delta Q+\frac{ \dot{\Delta}_{j}\delta u}{\tau}\\
&\quad\quad=-\dot{\Delta}_{j}\big((\frac{1}{\rho^{\var,\tau}}-\frac{1}{\rho^{\tau}}) \nabla P^{\tau}\big)+\dot{\Delta}_{j}\delta F_2+\delta R_{2,j}+\delta R_{3,j},
\end{aligned}
\right.
\end{equation}
with 
\begin{equation}\nonumber
\left\{
\begin{aligned}
&\delta R_{1,j}:=[u^{\var,\tau},\dot{\Delta}_{j}]\nabla \delta Q+[\Gamma_{2}^{\var,\tau},\dot{\Delta}_{j}]\div \dot{\Delta}_{j} \delta u,\\
&\delta R_{2,j}:=[u^{\var,\tau},\dot{\Delta}_{j}]\nabla \delta u,\\
&\delta R_{3,j}:=[\frac{1}{\rho^{\var,\tau}},\dot{\Delta}_{j}]\nabla \delta Q.
\end{aligned}
\right.
\end{equation}

Similarly to the high-frequency analysis in Section \ref{subsectionhigh}, one gains
\begin{equation}\label{deltaL1}
    \begin{aligned}
    &\frac{d}{dt}\int_{\mathbb{R}^{d}} (\frac{1}{ \rho^{\var,\tau}} |\dot{\Delta}_{j}\delta Q|^2+\Gamma^{\var,\tau}_{2} |\dot{\Delta}_{j} \delta u|^2 )dx+\frac{1}{\tau}\|\dot{\Delta}_{j} \delta u\|_{L^2}^2\\
    &\lesssim \Big\|\Big(\div u^{\var,\tau},\nabla \Gamma^{\var,\tau}_{2},\nabla \frac{1}{ \rho^{\var,\tau}},\partial_{t}\frac{1}{\rho^{\var,\tau}},\partial_{t}\Gamma^{\var,\tau}_{2}\Big)\Big\|_{\widetilde{L}^{\infty}} \|\dot{\Delta}_{j}\delta Q\|_{L^2}\|\dot{\Delta}_{j}\delta u\|_{L^2}\\
    &\quad+\Big\|\dot{\Delta}_{j}(\frac{1}{\rho^{\var,\tau}}-\frac{1}{\rho^{\tau}}) \nabla P^{\tau}\Big\|_{L^2}\|\dot{\Delta}_{j}\delta u\|_{L^2}+\|\dot{\Delta}_{j}(\delta F_1,\delta F_2)\|_{L^2}\|\dot{\Delta}_{j}(\delta Q,\delta u)\|_{L^2}\\
    &\quad+\|\delta R_{1,j}\|_{L^2}\|\dot{\Delta}_{j}\delta Q\|_{L^2}+\|(\delta R_{2,j},\delta R_{3,j})\|_{L^2}\|\dot{\Delta}_{j}\delta u\|_{L^2},
    \end{aligned}
\end{equation}
and the cross term
\begin{equation}\label{deltaL2}
\begin{aligned}
    &\frac{d}{dt}\int_{\mathbb{R}^{d}} \dot{\Delta}_{j} \delta u\cdot \nabla \dot{\Delta}_{j}\nabla \delta P dx\\
    &\quad+\int_{\mathbb{R}^{d}} \Big(\frac{1}{\rho^{\var,\tau}}|\nabla  \dot{\Delta}_{j}\delta Q|^2-\Gamma^{\var,\tau}_{2} |\div \dot{\Delta}_{j}\delta u_{j}|^2+\frac{1}{\tau}\dot{\Delta}_{j}\delta u\cdot \nabla \dot{\Delta}_{j}\nabla \delta P \Big)dx\\
    &\lesssim \bigg(\|u^{\var,\tau}\|_{\widetilde{L}^{\infty}}\|\nabla \dot{\Delta}_{j} \delta u\|_{L^2}+ \Big\|\dot{\Delta}_{j}\big((\frac{1}{\rho^{\var,\tau}}-\frac{1}{\rho^{\tau}}) \nabla P^{\tau}\big)\Big\|_{L^2}\\
    &\quad+\|(\dot{\Delta}_{j}\delta F_2,\delta R_{2,j},\delta R_{3,j})\|_{L^2}\bigg)\| \nabla \dot{\Delta}_{j}\nabla \delta P\|_{L^2}\\
    &\quad+(\|u^{\var,\tau}\|_{\widetilde{L}^{\infty}}\|\nabla \dot{\Delta}_{j} \delta Q\|_{L^2}+\|\dot{\Delta}_{j}\delta F_1\|_{L^2}+\|\delta R_{1,j}\|_{L^2})\|\delta u\|_{L^2}.
    \end{aligned}
    \end{equation}
For all $j\geq J_{\tau}$, multiplying \eqref{deltaL2} by a suitable small constant and adding the resulting inequality and \eqref{deltaL1} together, we can derive the Lyapunov inequality similar to \eqref{ujh}-\eqref{sim} and then show the following $L^1$-in-time type estimate:
\begin{equation}\nonumber
    \begin{aligned}
    &\tau\|(\delta Q,\delta u)\|_{\widetilde{L}^{\infty}_{t}(\dot{B}^{\frac{d}{2}-1})}^{h}+\|(\delta Q,\delta u)\|_{L^1_{t}(\dot{B}^{\frac{d}{2}-1})}^{h}\\
    &\quad\lesssim \sqrt{\tau\var}+(\|u^{\var,\tau}\|_{L^1_t(\dot{B}^{\frac{d}{2}}\cap\dot{B}^{\frac{d}{2}+1})}+\|(\partial_{t}\rho^{\var,\tau},\partial_{t}\Gamma^{\var,\tau}_{2})\|_{L^1_{t}(\dot{B}^{\frac{d}{2}})})\tau\|(\delta Q,\delta u)\|_{\widetilde{L}^{\infty}_{t}(\dot{B}^{\frac{d}{2}-1})}^{h}\\
    &\quad\quad+\|( \rho^{\var,\tau}-\bar{\rho},\Gamma^{\var,\tau}_{2}-\bar{\Gamma}_{2})\|_{\widetilde{L}^{\infty}_{t}(\dot{B}^{\frac{d}{2}+1})}\|(\delta Q,\delta u)\|_{L^1_{t}(\dot{B}^{\frac{d}{2}-1})}^{h}+\tau\|(\delta F_1,\delta F_2)\|_{L^1_{t}(\dot{B}^{\frac{d}{2}-1})}^{h}\\
    &\quad\quad+\tau\Big\|(\frac{1}{\rho^{\var,\tau}}-\frac{1}{\rho^{\tau}}) \nabla P^{\tau}\Big\|_{L^1_{t}(\dot{B}^{\frac{d}{2}-1})}^{h}\\
     &\quad\quad+\tau\sum_{j\geq J_{\tau}-1}2^{(\frac{d}{2}-1)j}\|(\delta R_{1,j}, \delta  R_{2,j},\delta R_{3,j})\|_{L^2}.
    \end{aligned}
\end{equation}
By \eqref{uniform1}, \eqref{estimatedeltaalpha}, \eqref{deltaP+P-1}, the product laws \eqref{uv2} and the commutator estimate \eqref{commutator}, it is easy to show
\begin{equation}\nonumber
    \begin{aligned}
   & \tau\Big\|(\frac{1}{\rho^{\var,\tau}}-\frac{1}{\rho^{\tau}}) \nabla P^{\tau}\Big\|_{L^1_{t}(\dot{B}^{\frac{d}{2}-1})}^{h}\\
   &\quad\lesssim \|\delta \rho\|_{\widetilde{L}^{\infty}_{t}(\dot{B}^{\frac{d}{2}-1})} \tau\|P^{\tau}-\bar{P}\|_{L^1_{t}(\dot{B}^{\frac{d}{2}+1})}\lesssim {o}(1)\|(\delta Y,\delta Q)\|_{\widetilde{L}^{\infty}_{t}(\dot{B}^{\frac{d}{2}-1})}+\sqrt{\var\tau},
    \end{aligned}
\end{equation}
and
\begin{equation}\nonumber
    \begin{aligned}
    &\tau\sum_{j\geq J_{\tau}-1}2^{(\frac{d}{2}-1)j}\|(\delta R_{1,j}, \delta  R_{2,j})\|_{L^2}\\
    &\quad\lesssim \|\nabla u^{\var,\tau}\|_{L^1_{t}(\dot{B}^{\frac{d}{2}})}\|(\delta Q,\delta u)\|_{\widetilde{L}^{\infty}_{t}(\dot{B}^{\frac{d}{2}-1})}\lesssim {o}(1)\|(\delta Q,\delta u)\|_{\widetilde{L}^{\infty}_{t}(\dot{B}^{\frac{d}{2}-1})}.
    \end{aligned}
\end{equation}
For the tricky commutator term $R_{3,j}$, we have
\begin{equation}\nonumber
    \begin{aligned}
   &\tau\sum_{j\geq J_{\tau}-1}2^{(\frac{d}{2}-1)j}\| R_{3,j}\|_{L^1_{t}(L^2)}\\
   &\quad\lesssim \tau\sum_{j\geq J_{\tau}-1}2^{\frac{d}{2}j}\Big\| [\frac{1}{\rho^{\var,\tau}},\dot{\Delta}_{j}]\nabla \delta Q^{\ell}\Big\|_{L^1_{t}(L^2)}+\sum_{j\geq J_{\tau}-1}2^{(\frac{d}{2}-1)j}\Big\| [\frac{1}{\rho^{\var,\tau}},\dot{\Delta}_{j}]\nabla \delta Q^{h}\Big\|_{L^1_{t}(L^2)}\\
   &\quad\lesssim \|\nabla \rho^{\var,\tau}\|_{\widetilde{L}^{\infty}_{t}(\dot{B}^{\frac{d}{2}})}( \tau\|\delta Q^{\ell}\|_{L^1_{t}(\dot{B}^{\frac{d}{2}})}+\|\delta Q^{h}\|_{L^1_{t}(\dot{B}^{\frac{d}{2}-1})})\\
   &\quad\lesssim {o}(1)\Big( \tau\|\delta Q\|_{L^1_{t}(\dot{B}^{\frac{d}{2}})}^{\ell}+\|\delta Q\|_{L^1_{t}(\dot{B}^{\frac{d}{2}-1})}^{h}\Big).
    \end{aligned}
\end{equation}
For $\delta F_{1}$ and $\delta F_{2}$, similar computations give rise to
\begin{equation}\label{F1F2d21}
    \begin{aligned}
    &\|(\delta F_{1},\delta F_{2})\|_{L^1_{t}(\dot{B}^{\frac{d}{2}-1})}\lesssim \|\delta u\|_{L^1_{t}(\dot{B}^{\frac{d}{2}-1})}\|( P^{\tau}-\bar{P}, u^{\tau})\|_{\widetilde{L}^{\infty}_{t}(\dot{B}^{\frac{d}{2}+1})}\\
    &\quad\quad+\Big\|\Gamma_{2}^{\var,\tau}-\frac{\gamma_{+}\gamma_{-}P^{\tau}}{\gamma_{+}\alpha_{-}^{\tau}+\gamma_{-}\alpha_{+}^{\tau}}\Big\|_{\widetilde{L}^{\infty}_{t}(\dot{B}^{\frac{d}{2}-1})}\|u^{\tau}\|_{L^{1}_{t}(\dot{B}^{\frac{d}{2}+1})}\\
    &\quad+(\|u^{\var,\tau}\|_{L^1_{t}(\dot{B}^{\frac{d}{2}+1})}+\|\partial_{t}\Gamma_{1}^{\var,\tau}+u^{\var,\tau}\cdot\nabla \Gamma_{1}^{\var,\tau}\|_{L^1_{t}(\dot{B}^{\frac{d}{2}})})\|P_{+}(\rho_{+}^{\var,\tau})-P_{-}(\rho_{-}^{\var,\tau})\|_{\widetilde{L}^{\infty}_{t}(\dot{B}^{\frac{d}{2}-1})}\\
    &\quad\quad+\|P_{+}(\rho_{+}^{\var,\tau})-P_{-}(\rho_{-}^{\var,\tau})\|_{L^{1}_{t}(\dot{B}^{\frac{d}{2}-1}\cap\dot{B}^{\frac{d}{2}})}\\
    &\quad\lesssim {o}(1)\Big(\|(\delta Y, \delta Q)\|_{\widetilde{L}^{\infty}_{t}(\dot{B}^{\frac{d}{2}-1})}+  \| \delta u\|_{L^1_{t}(\dot{B}^{\frac{d}{2}-1})} \Big)+\sqrt{\var\tau}+\var.
    \end{aligned}
\end{equation}
We thus get
\begin{equation}\label{step2}
    \begin{aligned}
    &\|(\delta Q,\delta u)\|_{\widetilde{L}^{\infty}_{t}(\dot{B}^{\frac{d}{2}-2})}^{h}+\tau\|(\delta Q,\delta u)\|_{\widetilde{L}^{\infty}_{t}(\dot{B}^{\frac{d}{2}-1})}^{h}\\
    &\quad\quad+\|(\delta Q,\delta u)\|_{L^1_{t}(\dot{B}^{\frac{d}{2}-1})}^{h}+\sqrt{\tau}\|\delta Q\|_{\widetilde{L}^2_{t}(\dot{B}^{\frac{d}{2}-1})}^{h}+\frac{1}{\sqrt{\tau}}\|u\|_{\widetilde{L}^2_{t}(\dot{B}^{\frac{d}{2}-2})}^{h}\\
    &\quad\lesssim {o}(1)\Big(\|(\delta Q,\delta u)\|_{\widetilde{L}^{\infty}_{t}(\dot{B}^{\frac{d}{2}-1})}+\tau \|\delta Q\|_{L^1_{t}(\dot{B}^{\frac{d}{2}})}^{\ell}\\
    &\quad\quad+\|\delta Q\|_{L^1_{t}(\dot{B}^{\frac{d}{2}-1})}^{h}+ \tau \|\delta u\|_{L^1_{t}(\dot{B}^{\frac{d}{2}-1})} \Big)+\sqrt{\var\tau}.
    \end{aligned}
    \end{equation}

\begin{itemize}
\item \textbf{Step 3: $\dot{B}^{\frac{d}{2}-1}$-estimates of $(\delta Q,\delta u)$ in all frequencies}
\end{itemize}

We need to further establish the uniform $\dot{B}^{\frac{d}{2}-1}$-bounds. To this end, owing to \eqref{deltaL1}, we obtain the $L^2$-in-time estimate
\begin{equation}\nonumber
    \begin{aligned}
    &\|(\delta Q,\delta u)\|_{\widetilde{L}^{\infty}_{t}(\dot{B}^{\frac{d}{2}-1})}+\frac{1}{\sqrt{\tau}}\|\delta u\|_{\widetilde{L}^2_{t}(\dot{B}^{\frac{d}{2}-1})}\\
    &\quad\lesssim \sqrt{\var\tau}+\Big(\|u^{\var,\tau}\|_{L^1_t(\dot{B}^{\frac{d}{2}})}+\|(\partial_{t}\rho^{\var,\tau},\partial_{t}\Gamma^{\var,\tau}_{2})\|_{L^1_{t}(\dot{B}^{\frac{d}{2}})}\Big)^{\frac{1}{2}}\|(\delta Q,\delta u)\|_{\widetilde{L}^{\infty}_{t}(\dot{B}^{\frac{d}{2}-1})}^{h}\\
    &\quad\quad+\|( \rho^{\var,\tau}-\bar{\rho},\Gamma^{\var,\tau}_{2}-\bar{\Gamma}_{2})\|_{\widetilde{L}^{\infty}_{t}(\dot{B}^{\frac{d}{2}+1})}^{\frac{1}{2}}\Big(\sqrt{\tau}\|\delta Q\|_{\widetilde{L}^2_{t}(\dot{B}^{\frac{d}{2}-1})}^{h}\Big)^{\frac{1}{2}}\Big(\frac{1}{\sqrt{\tau}}\|\delta u\|_{\widetilde{L}^2_{t}(\dot{B}^{\frac{d}{2}-1})}^{h}\Big)^{\frac{1}{2}}\\
    &\quad\quad+\|(\delta F_1,\delta F_2)\|_{L^1_{t}(\dot{B}^{\frac{d}{2}-1})}^{\frac{1}{2}}\|(\delta Q,\delta u)\|_{\widetilde{L}^{\infty}_{t}(\dot{B}^{\frac{d}{2}-1})}^{\frac{1}{2}}+\Big\|(\frac{1}{\rho^{\var,\tau}}-\frac{1}{\rho^{\tau}}) \nabla P^{\tau}\Big\|_{\widetilde{L}^{2}_{t}(\dot{B}^{\frac{d}{2}-1})}^{\frac{1}{2}}\|\delta u\|_{\widetilde{L}^2_{t}(\dot{B}^{\frac{d}{2}-1})}^{\frac{1}{2}}\\
     &\quad\quad+\sum_{j\in\mathbb{Z}}2^{(\frac{d}{2}-1)j}\Big(\int_{0}^{t}\big(\|\delta R_{1,j}\|_{L^2} \|\dot{\Delta}_{j}\delta Q\|_{L^2}+\|(\delta R_{2,j},\delta R_{3,j})\|_{L^2}\|\dot{\Delta}_{j}u\|_{L^2}\big)ds \Big)^{\frac{1}{2}}.
    \end{aligned}
\end{equation}
One has
\begin{equation}\nonumber
    \begin{aligned}
&\Big\|(\frac{1}{\rho^{\var,\tau}}-\frac{1}{\rho^{\tau}}) \nabla P^{\tau}\Big\|_{\widetilde{L}^{2}_{t}(\dot{B}^{\frac{d}{2}-1})}^{\frac{1}{2}}\|\delta u\|_{\widetilde{L}^2_{t}(\dot{B}^{\frac{d}{2}-1})}^{\frac{1}{2}}\\
&\quad\lesssim \|\delta\rho\|_{\widetilde{L}^{\infty}_{t}(\dot{B}^{\frac{d}{2}-1})}\Big( \sqrt{\tau}\|P^{\tau}-\bar{P}\|_{\widetilde{L}^2_{t}(\dot{B}^{\frac{d}{2}+1})}\Big)^{\frac{1}{2}}\Big(\frac{1}{\sqrt{\tau}}\|\|\delta u\|_{\widetilde{L}^2_{t}(\dot{B}^{\frac{d}{2}-1})}\Big)^{\frac{1}{2}}\\
&\quad\lesssim {o}(1) \Big( \|(\delta Y,\delta Q)\|_{\widetilde{L}^{\infty}_{t}(\dot{B}^{\frac{d}{2}-1})}+\frac{1}{\sqrt{\tau}}\|\|\delta u\|_{\widetilde{L}^2_{t}(\dot{B}^{\frac{d}{2}-1})} \Big)+\sqrt{\var\tau}.
    \end{aligned}
\end{equation}
Concerning the commutator terms, we have
\begin{equation}\nonumber
    \begin{aligned}
    &\sum_{j\in\mathbb{Z}}2^{(\frac{d}{2}-1)j}\Big(\int_{0}^{t}\Big(\|\delta R_{1,j}\|_{L^2} \|\dot{\Delta}_{j}\delta Q\|_{L^2}+\| \delta  R_{2,j}\|_{L^2}\|\dot{\Delta}_{j}u\|_{L^2}\big)ds \Big)^{\frac{1}{2}}\\
    &\quad\lesssim \Big(\frac{1}{\sqrt{\tau}}\|\nabla u^{\var,\tau}\|_{\widetilde{L}^2_{t}(\dot{B}^{\frac{d}{2}})}\|\delta Q\|_{\widetilde{L}^{\infty}_{t}(\dot{B}^{\frac{d}{2}-1})}\sqrt{\tau} \|\delta Q\|_{\widetilde{L}^2_{t}(\dot{B}^{\frac{d}{2}-1})}\Big)^{\frac{1}{2}}\\
    &\quad\quad+\Big(\frac{1}{\sqrt{\tau}}\|\delta u \|_{\widetilde{L}^2_{t}(\dot{B}^{\frac{d}{2}-1})}\|\nabla \Gamma^{\var,\tau}_{2}\|_{\widetilde{L}^{\infty}_{t}(\dot{B}^{\frac{d}{2}+1})} \sqrt{\tau} \|\delta Q\|_{\widetilde{L}^2_{t}(\dot{B}^{\frac{d}{2}-1})}\Big)^{\frac{1}{2}}+\big(\|\nabla u^{\var,\tau}\|_{L^1_{t}(\dot{B}^{\frac{d}{2}})}\|\delta u\|_{\widetilde{L}^{\infty}_{t}(\dot{B}^{\frac{d}{2}-1})}^2 \big)^{\frac{1}{2}}\\
    &\quad\quad+\Big(\|\nabla \rho^{\var,\tau}\|_{\widetilde{L}^{\infty}_{t}(\dot{B}^{\frac{d}{2}})}\sqrt{\tau}\|\delta Q\|_{\widetilde{L}^{\infty}_{t}(\dot{B}^{\frac{d}{2}-1})} \frac{1}{\sqrt{\tau}}\|\delta u\|_{\widetilde{L}^{2}_{t}(\dot{B}^{\frac{d}{2}-1})} \Big)^{\frac{1}{2}}\\
    &\quad\lesssim {o}(1)\Big( \|(\delta Q,\delta u)\|_{\widetilde{L}^{\infty}_{t}(\dot{B}^{\frac{d}{2}-1})}+\sqrt{\tau}\|\delta Q\|_{\widetilde{L}^2_{t}(\dot{B}^{\frac{d}{2}-1})}+ \frac{1}{\sqrt{\tau}}\|\delta u \|_{\widetilde{L}^2_{t}(\dot{B}^{\frac{d}{2}-1})} \Big).
        \end{aligned}
\end{equation}
Gathering \eqref{F1F2d21} and the above three estimates,  we have
\begin{equation}\label{step3}
    \begin{aligned}
    &\|(\delta Q,\delta u)\|_{\widetilde{L}^{\infty}_{t}(\dot{B}^{\frac{d}{2}-1})}+\frac{1}{\sqrt{\tau}}\|\delta u\|_{\widetilde{L}^2_{t}(\dot{B}^{\frac{d}{2}-1})}\\
    &\quad\lesssim {o}(1)\Big( \|(\delta Y,\delta Q,\delta u)\|_{\widetilde{L}^{\infty}_{t}(\dot{B}^{\frac{d}{2}-1})}+\|\delta Q\|_{\widetilde{L}^2_{t}(\dot{B}^{\frac{d}{2}-1})}+\frac{1}{\sqrt{\tau}}\|\delta u\|_{\widetilde{L}^2_{t}(\dot{B}^{\frac{d}{2}-1})} \Big)+\sqrt{\var\tau}.
            \end{aligned}
\end{equation}
    
\begin{itemize}
\item \textbf{Step 4: Proof of convergence rate}
\end{itemize}

    Finally, as required in \eqref{step1}, one needs to estimate $\partial_{t}\delta Q$. We make use of the equation $\eqref{deltaK}_{1}$, \eqref{unablaldetalQ} and \eqref{deltaF1} to get
    \begin{equation}\label{step4}
    \begin{aligned}
    \|\partial_{t}\delta Q\|_{L^1_{t}(\dot{B}^{\frac{d}{2}-2})}&\lesssim \|u^{\var,\tau}\cdot\nabla  \delta Q +(\Gamma_{2}^{\var,\tau}-\bar{\Gamma}_{2})\div\delta u\|_{L^1_t(\dot{B}^{\frac{d}{2}-2})}\\
    &\quad+\|\div \delta u\|_{L^1_t(\dot{B}^{\frac{d}{2}-2})}+\|\delta F_{1}\|_{L^1_t(\dot{B}^{\frac{d}{2}-2})}\\
    &\lesssim o(1)\Big( \|(\delta Y,\delta Q)\|_{\widetilde{L}^{\infty}_{t}(\dot{B}^{\frac{d}{2}-2})}+ \frac{1}{\sqrt{\tau}} \|\delta u\|_{\widetilde{L}^{2}_{t}(\dot{B}^{\frac{d}{2}-2})}\Big)\\
    &\quad+\| \delta u\|_{L^1_t(\dot{B}^{\frac{d}{2}-1})}+\sqrt{\var\tau}.
    \end{aligned}
\end{equation} 
    Combining \eqref{deltaalpha}, \eqref{step1}, \eqref{step2}, \eqref{step3} and \eqref{step4} together, we end up with \eqref{errordeltaQu} which completes the proof of Lemma \ref{lemmadeltaQu}.
\end{proof}

\subsection{Time-relaxation limit: System \eqref{Keta} to System \eqref{PM}}\label{sectionrelaxationKPM}


This section is devoted to the proof of \eqref{error2} in Theorem \ref{theorem12}. Define the error variables
$$
(\delta \beta_{\pm}, \delta \varrho_{\pm}, \delta \varrho, \delta \Pi,\delta v):=(\beta^{\tau}_{\pm}-\beta_{\pm},  \varrho^{\tau}_{\pm}-\varrho_{\pm},\varrho^{\tau}-\varrho, \Pi^{\tau}-\Pi, v^{\tau}-v).
$$
First, similarly to \eqref{sectionrelaxationBNK}, instead of $\delta\beta$, we need to estimate the variable 
$$
\delta Z:=\frac{\beta_{+}^{\tau}\varrho_{+}^{\tau}}{\varrho^{\tau}}-\frac{\beta_{+}\varrho_{+}}{\varrho},
$$
where the initial data of $\delta z$ is
\begin{align}
\delta Z|_{t=0}=Z^{\tau}_{0}-Z_{0},\quad Z^{\tau}_{0}:=\frac{\alpha_{+,0}^{\tau}\rho_{+,0}^{\tau}}{\alpha_{+,0}^{\tau}\rho_{+,0}^{\tau}+\alpha_{-,0}^{\tau}\rho_{-,0}^{\tau}},\quad Z_{0}:=\frac{\beta_{+,0}\varrho_{+,0}}{\beta_{+,0}\varrho_{+,0}+\beta_{-,0}\varrho_{-,0}}.\label{deltaZ0}
\end{align}
Indeed, arguing similarly as in Lemma \ref{lemmarecover}, we obtain from \eqref{Keta} and \eqref{PM} that
\begin{equation}\label{deltarho000eta}
\left\{
    \begin{aligned}
    \delta \beta_{+}&=\frac{1}{\beta^{\tau}_{-}\varrho_{+}^{\tau}\varrho_{-}+\beta_{+}\varrho_{+}\varrho_{-}^{\tau}} \big( \varrho^{\tau}\varrho \delta Z-\beta_{+}\beta_{-}\varrho_{-}\delta\varrho_{+}+\beta_{+}^{\tau}\beta_{-} \varrho_{+}\delta \varrho_{-}\big),\\
    \delta\varrho^{\tau}&=\frac{\varrho_{+}^{\tau}-\varrho_{-}^{\tau}}{\beta_{-}\varrho_{+}^{\tau}\varrho_{-}+\beta_{+}\varrho_{+}\varrho_{-}^{\tau}} \big( \varrho^{\tau}\varrho \delta Z-\beta_{+}\beta_{-}\varrho_{-}\delta\varrho_{+}+\beta_{+}\beta_{-}\varrho_{+}\delta\varrho_{-}\big)+\beta_{+}\delta\varrho_{+}+\beta_{-}\delta\varrho_{-}^{\tau},\\
\delta \Pi&=\delta \varrho_{+}\int_{0}^{1} P_{+}'(\theta \varrho^{\tau}_{+}+(1-\theta)\varrho_{+})d\theta =\delta \varrho_{-}\int_{0}^{1} P_{-}'(\theta \varrho^{\tau}_{-}+(1-\theta)\varrho_{-})d\theta,
\end{aligned}
\right.
\end{equation}
which leads to
\begin{equation}\label{deltarho0eta}
\left\{
\begin{aligned}
&\|\delta \varrho_{\pm}\|_{\widetilde{L}^{\infty}_{t}(\dot{B}^{\frac{d}{2}-1})}+\|\delta \varrho_{\pm}\|_{L^1_{t}(\dot{B}^{\frac{d}{2}+1})}\sim \|\delta \Pi\|_{\widetilde{L}^{\infty}_{t}(\dot{B}^{\frac{d}{2}-1})}+\|\delta \Pi\|_{L^1_{t}(\dot{B}^{\frac{d}{2}+1})},\\
&\|\delta \beta_{\pm}\|_{\widetilde{L}^{\infty}_{t}(\dot{B}^{\frac{d}{2}-1})}\lesssim  \|(\delta Z,\delta\Pi)\|_{\widetilde{L}^{\infty}_{t}(\dot{B}^{\frac{d}{2}-1})}.
\end{aligned}
\right.
\end{equation}
It is therefore sufficient to estimate $(\delta \Pi,\delta v, \delta Z)$ to recover the information on all the error unknowns.

Next, note that  $\delta Z$ satisfies the transport equation 
\begin{equation}\label{deltaY0tau0eta}
    \begin{aligned}
        &\partial_{t}\delta Z+v^{\tau}\cdot\nabla \delta Z=-\delta v\cdot\nabla \frac{\beta_{+}\varrho_{+}}{\varrho}.
    \end{aligned} 
\end{equation}
Using Lemma \ref{maximaldamped}, \eqref{uniformK} and the product law \eqref{uv2}, we get
\begin{equation}\label{deltaalpha0eta}
\begin{aligned}
\|\delta Z\|_{\widetilde{L}^{\infty}_{t}(\dot{B}^{\frac{d}{2}-1}\cap\dot{B}^{\frac{d}{2}})}&\lesssim {\rm{exp}}\big(\|v^{\tau}\|_{L^1_{t}(\dot{B}^{\frac{d}{2}+1})}\big)\|\delta v\|_{L^1_{t}(\dot{B}^{\frac{d}{2}})}  \| \nabla \frac{\beta_{+}\varrho_{+}}{\varrho}\|_{\widetilde{L}^{\infty}_{t}(\dot{B}^{\frac{d}{2}}\cap\dot{B}^{\frac{d}{2}+1})}\\
&\lesssim  {o}(1)\|\delta v\|_{L^1_{t}(\dot{B}^{\frac{d}{2}})}.
\end{aligned}
\end{equation}

Then, we perform the key estimates of $\delta\Pi$. From \eqref{Keta}, it is easy to see
\begin{equation}\nonumber
    \begin{aligned}
&\partial_{t}\Pi^{\tau}+v^{\tau}\cdot\nabla \Pi^{\tau}=\frac{\gamma_{+}\gamma_{-}\Pi^{\tau}}{\gamma_{+}\beta_{-}^{\tau}+\gamma_{-}\beta_{+}^{\tau}}\div\Big (\frac{\nabla \Pi^{\tau}}{\varrho^{\tau}}\Big)-\frac{\gamma_{+}\gamma_{-}\Pi^{\tau}}{\gamma_{+}\beta_{-}^{\tau}+\gamma_{-}\beta_{+}^{\tau}}\div z^{\tau},\quad\quad z^{\tau}:=v^{\tau}+\frac{\nabla \Pi^{\tau}}{\varrho^{\tau}}.
    \end{aligned}
\end{equation}
Thence by the above equation and \eqref{PM}, $\delta\Pi$ satisfies
\begin{equation}\label{deltaP0tau0eta}
\begin{aligned}
&\partial_{t}\delta \Pi-\bar{c}\Delta \delta \Pi=-v^{\tau}\cdot\nabla \delta \Pi-\delta v\cdot \nabla \Pi\\
&\quad+\Big( \frac{\gamma_{+}\gamma_{-}\Pi^{\tau}}{\gamma_{+}\beta_{+}^{\tau}+\gamma_{-}\beta_{-}^{\tau}}-\frac{\gamma_{+}\gamma_{-} \Pi}{\gamma_{+}\beta_{+}+\gamma_{-} \beta_{-} }\Big)\div \Big(\frac{ \nabla \Pi^{\tau}}{\varrho^{\tau}}\Big)\\
&\quad+\frac{\gamma_{+}\gamma_{-} \Pi}{\gamma_{+} \beta_{+}+\gamma_{-}\beta_{-} }\div\big( (\frac{1}{\varrho^{\tau}}-\frac{1}{\varrho})\nabla \Pi^{\tau}\big)-\frac{\gamma_{+}\gamma_{-}\Pi^{\tau}}{\gamma_{+}\beta_{-}^{\tau}+\gamma_{-}\beta_{+}^{\tau}}\div z^{\tau},
\end{aligned}
\end{equation}
with the constant $
\bar{c}:=\frac{\gamma_{+}\gamma_{-} P_{+}(\bar{\rho}_{+})}{(\gamma_{+}\bar{\alpha}_{-}+\gamma_{-}\bar{\alpha}_{+})\bar{\rho}}>0.
$
We mention that the convergence rate $\tau$ is $\mathcal{O}(\tau)$ bound comes from the uniform estimate \eqref{uniformK} of the effective unknown $z^{\tau}$. Indeed, by Lemma \ref{maximalheat}, the uniform estimate \eqref{uniformK}, the smallness of the initial data \eqref{a1}, the product laws \eqref{uv1} and the composition estimates \eqref{F1}-\eqref{F11}, one obtains
\begin{equation}\label{deltaP00eta}
\begin{aligned}
&\|\delta \Pi\|_{\widetilde{L}^{\infty}_{t}(\dot{B}^{\frac{d}{2}-1})}+\|\delta \Pi\|_{L^1_{t}(\dot{B}^{\frac{d}{2}+1})}\lesssim \tau+\|v^{\tau}\|_{\widetilde{L}^2_{t}(\dot{B}^{\frac{d}{2}})}\|\delta \Pi\|_{\widetilde{L}^{2}_{t}(\dot{B}^{\frac{d}{2}-1})}+\|\delta v\|_{L^1_{t}(\dot{B}^{\frac{d}{2}})}\|\Pi-\bar{P}\|_{\widetilde{L}^{\infty}_{t}(\dot{B}^{\frac{d}{2}})}\\
&\quad\quad+\|\delta \Pi\|_{\widetilde{L}^{2}_{t}(\dot{B}^{\frac{d}{2}})}\|\Pi^{\tau}-\bar{P}\|_{\widetilde{L}^2_{t}(\dot{B}^{\frac{d}{2}+1})}+  \Big\|(\frac{1}{\varrho^{\tau}}-\frac{1}{\varrho})\nabla \Pi^{\tau}\Big\|_{L^1_{t}(\dot{B }^{\frac{d}{2}})}+\|z^{\tau}\|_{L^1_{t}(\dot{B}^{\frac{d}{2}})}\\
&\quad\lesssim {o}(1)\Big( \|\delta Z\|_{\widetilde{L}^{\infty}_{t}(\dot{B}^{\frac{d}{2}})}+\|(\delta \varrho_{\pm},\delta \Pi)\|_{\widetilde{L}^{\infty}_{t}(\dot{B}^{\frac{d}{2}-1})}+\|(\delta \varrho_{\pm},\delta \Pi)\|_{L^1_{t}(\dot{B}^{\frac{d}{2}+1})} \Big) +\tau,
\end{aligned}
\end{equation}
where we have used the key fact
\begin{equation}\label{keyr0eta}
\begin{aligned}
\Big\|(\frac{1}{\varrho^{\tau}}-\frac{1}{\varrho})\nabla \Pi^{\tau}\Big\|_{L^1_{t}(\dot{B }^{\frac{d}{2}})}&\lesssim  \|\delta Z\|_{\widetilde{L}^{\infty}_{t}(\dot{B}^{\frac{d}{2}})}\|\Pi^{\tau}-\bar{P}\|_{L^1_{t}(\dot{B }^{\frac{d}{2}+1})}+\|\delta \varrho_{\pm}\|_{\widetilde{L}^2_{t}(\dot{B}^{\frac{d}{2}})}\|\Pi^{\tau}-\bar{P}\|_{\widetilde{L}^2_{t}(\dot{B }^{\frac{d}{2}+1})} \\
&\quad\lesssim {o}(1) \Big( \|\delta Z\|_{\widetilde{L}^{\infty}_{t}(\dot{B}^{\frac{d}{2}})}+\|\delta \varrho_{\pm}\|_{\widetilde{L}^2_{t}(\dot{B}^{\frac{d}{2}})}\Big),
\end{aligned}
\end{equation}
derived from \eqref{uv2}, \eqref{uniformK} and \eqref{deltarho000eta}. Gathering \eqref{deltarho0eta} and \eqref{deltaP00eta} together, we get
\begin{equation}\label{deltaP0eta}
\begin{aligned}
&\|\delta \Pi\|_{\widetilde{L}^{\infty}_{t}(\dot{B}^{\frac{d}{2}-1})}+\|\delta \Pi\|_{L^1_{t}(\dot{B}^{\frac{d}{2}+1})}\lesssim {o}(1) \Big( \|\delta Z\|_{\widetilde{L}^{\infty}_{t}(\dot{B}^{\frac{d}{2}})}+\|\delta \Pi\|_{L^1_{t}(\dot{B}^{\frac{d}{2}+1})} \Big)+\tau.
\end{aligned}
\end{equation}

For the error unknown $\delta v$, in view of \eqref{uniformK}, \eqref{keyr0eta} and $\delta v=(\frac{1}{\varrho^{\tau}}-\frac{1}{\varrho})\nabla \Pi^{\tau} -\frac{1}{\varrho}\nabla \delta \Pi+z^\tau$, it can be bounded by
\begin{equation}\label{deltau0eta}
\begin{aligned}
\|\delta v\|_{L^1_{t}(\dot{B}^{\frac{d}{2}})}&\lesssim \|\delta \Pi\|_{L^1_{t}(\dot{B}^{\frac{d}{2}+1})}+\Big\|(\frac{1}{\varrho^{\tau}}-\frac{1}{\varrho})\nabla \Pi^{\tau}\Big\|_{L^1_{t}(\dot{B }^{\frac{d}{2}})}+\|z^{\tau}\|_{L^1_{t}(\dot{B}^{\frac{d}{2}})}\\
&\lesssim \|\delta Z\|_{\widetilde{L}^{\infty}_{t}(\dot{B}^{\frac{d}{2}})}+\|\delta \Pi\|_{\widetilde{L}^{\infty}_{t}(\dot{B}^{\frac{d}{2}-1})}+\|\delta \Pi\|_{L^1_{t}(\dot{B}^{\frac{d}{2}+1})}+\tau.
\end{aligned}
\end{equation}

The combination of estimate \eqref{deltarho0eta} and inequalities \eqref{deltaalpha0eta}-\eqref{deltaP0eta} gives rise to estimate \eqref{error2}, which completes the proof of Theorem \ref{theorem12}.

\vspace{2ex}
\textbf{Acknowledgments}
The authors are indebted to the anonymous referees for their valuable suggestions and comments on the manuscript. T. Crin-Barat has been funded by the Alexander von Humboldt-Professorship program and the Transregio 154 Project “Mathematical Modelling, Simulation and Optimization Using the Example of Gas Networks” of the DFG. L.-Y. Shou is supported by the National Natural Science Foundation of China (12301275) and the China Postdoctoral Science Foundation  (2023M741694). J. Tan is partially supported by the ANR project BORDS (ANR-16-CE40-0027-01),  by the Labex MME-DII and the CY Initiative of Excellence, project CYNA (CY Nonlinear Analysis).
  \bigbreak
 
\small
\bibliographystyle{abbrv} 
	\parskip=0pt
	\small
	
\bibliography{reference}

\begin{thebibliography}{10}

\bibitem{baer1}
M.~R. Baer and J.~W. Nunziato.
\newblock A two-phase mixture theory for the deflagration-to-detonation
  transition {(DDT)} in reactive granular materials.
\newblock {\em International journal of multiphase flow}, \textbf{12}(6),
  861-889, 1986.

\bibitem{bahouri1}
H.~Bahouri, J.-Y. Chemin, and R.~Danchin.
\newblock {\em {F}ourier {A}nalysis and {N}onlinear {P}artial {D}ifferential
  {E}quations}, volume 343.
\newblock Springer, Grundlehren der Mathematischen Wissenschaften, 2011.

\bibitem{bea1}
K.~Beauchard and E.~Zuazua.
\newblock Large time asymptotics for partially dissipative hyperbolic systems.
\newblock {\em Arch. Rational Mech. Anal.}, \textbf{199}(1), 177-227, 2011.

\bibitem{serre1}
S.~Benzoni-Gavage and D.~Serre.
\newblock {\em Multidimensional Hyperbolic Partial Differential Equations:
  First-Order Systems and Applications, Oxford Math. Monogr.}
\newblock The Clarendon Press, Oxford University Press, Oxford, 2007.

\bibitem{BCBP}
R.~Bianchini, T.~Crin-Barat, and M.~Paicu.
\newblock Relaxation approximation and asymptotic stability of stratified
  solutions to the $\textsc{IPM}$ equation.
\newblock {\em Arch Rational Mech Anal}, \textbf{248}, 2, 2024.

\bibitem{bresch1}
D.~Bresch, C.~Burtea, and F.~Lagouti\'ere.
\newblock Mathematical justification of a compressible bifluid system with
  different pressure laws: a continuous approach.
\newblock {\em Applicable Analysis}, 101:4235--4266, 2022.

\bibitem{breschCos1}
D.~Bresch, C.~Burtea, and F.~Lagouti\'ere.
\newblock Mathematical justification of a compressible bi-fluid system with
  different pressures laws: a semi-discrete approach and numerical
  illustrations.
\newblock {\em J. Comp. Phys.}, \textbf{490}, 112259, 2023.

\bibitem{breschhand}
D.~Bresch, B.~Desjardins, J.~M. Ghidaglia, G.~E., and M.~Hilliairet.
\newblock Multifluid models including compressible fluids.
\newblock {\em in : Giga Y., Novotn$\acute{y}$ (eds) Handbook of Mathematical
  Analysis in Mechanics of Viscous Fluids. Springer, International Publishing
  Switzerland}, 2018.

\bibitem{bresch2}
D.~Bresch, B.~Desjardins, J.-M. Ghidaglia, and E.~Grenier.
\newblock Global weak solutions to a generic two-fluid model.
\newblock {\em Arch. Rational Mech. Anal.}, \textbf{196}(2), 599-629, 2010.

\bibitem{BH1}
D.~Bresch and M.~Hillairet.
\newblock Note on the derivation of multicomponent flow systems.
\newblock {\em Proc. Amer. Math. Soc.}, \textbf{143}, 3429-3443, 2015.

\bibitem{bresch3}
D.~Bresch and M.~Hillairet.
\newblock A compressible multifluid system with new physical relaxation terms.
\newblock {\em Annales ENS}, \textbf{52}(1), 255-295, 2019.

\bibitem{bresch5}
D.~Bresch and X.~Huang.
\newblock A multi-fluid compressible system as the limit of weak solutions of
  the isentropic compressible {N}avier–{S}tokes equations.
\newblock {\em Arch. Rational Mech. Anal.}, \textbf{201}(2), 647-680, 2011.

\bibitem{BreschHuangLi}
D.~Bresch, X.~Huang, and J.~Li.
\newblock Global weak solutions to one-dimensional non-conservative viscous
  compressible two-phase system.
\newblock {\em Commun. Math. Phys.}, \textbf{309}, 737-755, 2012.

\bibitem{bresch4}
D.~Bresch, P.-B. Mucha, and E.~Zatorska.
\newblock Finite-energy solutions for compressible two-fluid {S}tokes system.
\newblock {\em Arch. Rational Mech. Anal.}, \textbf{232}, 987-1029, 2019.

\bibitem{burtea1}
C.~Burtea, T.~Crin-Barat, and J.~Tan.
\newblock Pressure-relaxation limit for a damped one-velocity {B}aer-{N}unziato
  model to a {K}appila model.
\newblock {\em Math. Models Methods Appl. Sci.}, \textbf{33}(4), 687-753, 2023.

\bibitem{chen1994}
G.-Q. Chen, C.~D. Levermore, and T.-P. Liu.
\newblock Hyperbolic conservation laws with stiff relaxation terms and entropy.
\newblock {\em Com. Pure Appl. Math.}, \textbf{47}(6), 787-830, 1994.

\bibitem{cou1}
J.-F. Coulombel and T.~Goudon.
\newblock The strong relaxation limit of the multidimensional isothermal
  {E}uler equations.
\newblock {\em Trans. Amer. Math. Soc.}, \textbf{359}, 637-648, 2007.

\bibitem{c0}
T.~Crin-Barat and R.~Danchin.
\newblock Partially dissipative hyperbolic systems in the critical regularity
  setting: {T}he multi-dimensional case.
\newblock {\em J. Math. Pures Appl. (9)}, \textbf{165}, 1-41, 2022.

\bibitem{c1}
T.~Crin-Barat and R.~Danchin.
\newblock Global existence for partially dissipative hyperbolic systems in the
  {$L^{p}$} framework, and relaxation limit.
\newblock {\em Math. Ann.}, \textbf{386}, 2159-2206, 2023.

\bibitem{c2}
T.~Crin-Barat, Q.~He, and L.-Y. Shou.
\newblock The hyperbolic-parabolic chemotaxis system for vasculogenesis:
  {G}lobal dynamics and relaxation limit toward a {K}eller-{S}egel model.
\newblock {\em SIAM J. Math. Anal.}, \textbf{55}(5), 4445-4492, 2023.

\bibitem{jinxin}
T.~Crin-Barat and L.-Y. Shou.
\newblock Diffusive relaxation limit of the multi-dimensional {J}in-{X}in
  system.
\newblock {\em J. Differential Equations}, \textbf{357}, 302-331, 2023.

\bibitem{danchin1}
R.~Danchin.
\newblock Global existence in critical spaces for compressible
  {N}avier-{S}tokes equations.
\newblock {\em Invent. Math.}, \textbf{141}(3), 579-614, 2000.

\bibitem{danchin4}
R.~Danchin.
\newblock Fourier analysis methods for the compressible {N}avier-{S}tokes
  equations.
\newblock {\em in : Giga Y., Novotn$\acute{y}$ (eds) Handbook of Mathematical
  Analysis in Mechanics of Viscous Fluids. Springer, International Publishing
  Switzerland}, 2018.

\bibitem{danchin5}
R.~Danchin.
\newblock Partially dissipative systems in the critical regularity setting, and
  strong relaxation limit.
\newblock {\em EMS Surv. Math. Sci.}, 2023, DOI: 10.4171/EMSS/55.

\bibitem{evje11}
S.~Evje and K.~H. Karlsen.
\newblock Global existence of weak solutions for a viscous two-phase model.
\newblock {\em J. Differential Equations}, \textbf{245}, 2660-2703, 2008.

\bibitem{evje1}
S.~Evje, W.~Wang, and H.~Wen.
\newblock Global well-posedness and decay rates of strong solutions to a
  non-conservative compressible two-fluid model.
\newblock {\em Arch. Rational Mech. Anal.}, \textbf{221}, 1285-1316, 2016.

\bibitem{ForestierGavrilyuk}
A.~Forestier and S.~Gavrilyuk.
\newblock Criterion of hyperbolicity for non-conservative quasilinear systems
  admitting a partially convex conservation law.
\newblock {\em Math. Meth. Appl. Sci.}, \textbf{34}, 2148-2158, 2011.

\bibitem{Gio1}
V.~Giovangigli and W.-A. Yong.
\newblock Volume viscosity and internal energy relaxation: symmetrization and
  {C}hapman-{E}nskog expansion.
\newblock {\em Kinet. Relat. Models}, \textbf{8}, 79-116, 2014.

\bibitem{Gio2}
V.~Giovangigli and W.-A. Yong.
\newblock Volume viscosity and internal energy relaxation: error estimates.
\newblock {\em Nonlinear Anal. Real World Appl.}, \textbf{43}, 213-244, 2018.

\bibitem{guo1}
Z.~Guo, J.~Yang, and L.~Yao.
\newblock Global strong solution for a three-dimensional viscous liquid-gas
  two-phase flow model with vacuum.
\newblock {\em J. Math. Phy.}, \textbf{52}, 093102, 2011.

\bibitem{hao1}
C.~Hao and H.-L. Li.
\newblock Well-posedness for a multidimensional viscous liquid-gas two-phase
  flow model.
\newblock {\em SIAM J. Math. Anal.}, \textbf{44}(3), 1304-1332, 2012.

\bibitem{ishii2}
M.~Ishii.
\newblock {\em Thermo-fluid dynamics of two-phase flow}.
\newblock Springer-Verlag, New York, 2006.

\bibitem{junca1}
S.~Junca and M.~Rascle.
\newblock Strong relaxation of the isothermal {E}uler system to the heat
  equation.
\newblock {\em Z. Angew. Math. Phys.}, \textbf{53}, 239-264, 2002.

\bibitem{kapila1}
A.~Kapila, R.~Menikoff, J.~Bdzil, S.~Son, and D.~Stewart.
\newblock Two-phase modeling of deflagration-to-detonation transition in
  granular materials: {R}educed equations.
\newblock {\em Physics of fluids}, \textbf{13}(10), 3002-3024, 2001.

\bibitem{kracmar1}
S.~Kra\u{c}mar, Y.-S. Kwon, {\u{S}}.~Ne\u{c}asov\'a, and
  A.~Novotn$\rm{\acute{y}}$.
\newblock Weak solutions for a bifluid model for a mixture of two compressible
  noninteracting fluids with general boundary data.
\newblock {\em SIAM J. Math. Anal.}, \textbf{54}(1), 818-871, 2022.

\bibitem{lhlshou2}
H.-L. Li and L.-Y. Shou.
\newblock Global existence and optimal time-decay rates of the compressible
  {N}avier-{S}tokes-{E}uler system.
\newblock {\em SIAM J. Math. Anal.}, \textbf{55}(3), 1810-1846, 2023.

\bibitem{li2}
H.-L. Li and L.-Y. Shou.
\newblock Global existence of weak solutions to the drift-flux system for
  general pressure laws.
\newblock {\em Sci. China. Math.}, \textbf{66}(2), 251-284, 2023.

\bibitem{li010}
J.~Li, Y.~Yu, and C.~Zhu.
\newblock Ill-posedness for the {B}urgers equation in {S}obolev spaces.
\newblock {\em Indian J Pure Appl Math}, 2022,
  https://doi.org/10.1007/s13226-022-00357-z.

\bibitem{linares1}
F.~Linares, D.~Pilod, and J.-C. Saut.
\newblock Dispersive perturbations of {B}urgers and hyperbolic equations {I}:
  {L}ocal theory.
\newblock {\em SIAM J. Math. Anal.}, \textbf{46}(2), 1505-1537, 2014.

\bibitem{Majda}
A.~Majda.
\newblock {\em Compressible Fluid Flow and Systems of Conservation Laws in
  Several Space Variable}.
\newblock Springer, New-York, 1984.

\bibitem{mats1}
A.~Matsumura and T.~Nishida.
\newblock The {C}auchy problem for the equations of motion of compressible
  viscous and heat-conductive fluids.
\newblock {\em Proc. Japan Acad. Ser. A Math. Sci.}, \textbf{55}, 337-342,
  1979.

\bibitem{novotny2}
A.~Novotn$\rm{\acute{y}}$ and M.~Pokorn$\rm{\acute{y}}$.
\newblock Weak solutions for some compressible multicomponent fluid models.
\newblock {\em Arch. Rational Mech. Anal.}, \textbf{235}(1), 355-403, 2020.

\bibitem{runst1}
T.~Runst and W.~Sickel.
\newblock {\em Sobolev spaces of fractional order, {N}emytskij operators, and
  nonlinear partial differential equations}.
\newblock Nonlinear Analysis and Applications, Walter de Gruyter \& Co.,
  Berlin, 1996.

\bibitem{shi1}
S.~Shizuta and S.~Kawashima.
\newblock Systems of equations of hyperbolic-parabolic type with applications
  to the discrete {B}oltzmann equation.
\newblock {\em Hokkaido Math. J.}, \textbf{14}, 249-275, 1985.

\bibitem{vasseur1}
A.~Vasseur, H.~Wen, and C.~Yu.
\newblock Global weak solution to the viscous two-fluid model with finite
  energy.
\newblock {\em J. Math. Pures Appl. (9)}, \textbf{125}, 247-282, 2019.

\bibitem{wallis1}
G.~Wallis.
\newblock {\em One-dimensional two-fluid flow}.
\newblock McGraw-Hill, New York, 1979.

\bibitem{wen0}
H.~Wen.
\newblock On global solutions to a viscous compressible two-fluid model with
  unconstrained transition to single-phase flow in three dimensions.
\newblock {\em Calc. Var. Partial Differential Equations}, \textbf{60}, 158,
  2021.

\bibitem{wen1}
H.~Wen, L.~Yao, and C.~Zhu.
\newblock Review on mathematical analysis of some two-phase flow models.
\newblock {\em Acta Math. Sci.}, \textbf{38}, 1617-1636, 2018.

\bibitem{xu1}
J.~Xu and S.~Kawashima.
\newblock Global classical solutions for partially dissipative hyperbolic
  system of balance laws.
\newblock {\em Arch. Rational Mech. Anal.}, \textbf{211}(2), 513-553, 2014.

\bibitem{xu00}
J.~Xu and Z.~Wang.
\newblock Relaxation limit in {B}esov spaces for compressible {E}uler
  equations.
\newblock {\em J. Math. Pures Appl. (9)}, \textbf{99}, 43-61, 2013.

\bibitem{yao1}
L.~Yao, T.~Zhang, and C.~Zhu.
\newblock Existence of asymptotic behavior of global weak solutions to a 2d
  viscous liquid-gas two-phase flow model.
\newblock {\em SIAM J. Math. Anal.}, \textbf{42}(4), 1874-1897, 2010.

\bibitem{yao3}
L.~Yao and C.~Zhu.
\newblock Existence and uniqueness of global weak solution to a two-phase flow
  model with vacuum.
\newblock {\em Math. Ann.}, \textbf{349}, 903-928, 2011.

\bibitem{zhangying1}
Y.~Zhang.
\newblock Decay of the 3d inviscid liquid–gas two-phase flow model.
\newblock {\em Z. Angew. Math. Phys.}, \textbf{67}, 54, 2016.

\bibitem{zhang1}
Y.~Zhang and C.~Zhu.
\newblock Global existence and optimal convergence rates for the strong
  solutions in {$H^{2}$} to the 3d viscous liquid-gas two-phase flow model.
\newblock {\em J. Differential Equations}, \textbf{258}, 2315-2338, 2015.

\bibitem{SlidesZuazua}
E.~Zuazua.
\newblock Decay of partially dissipative hyperbolic systems.
\newblock {\em https://www.dm.uniba.it/it/ricerca/convegni/2022
  /acipdif22/decaypartiallydiss-1dfocus.pdf/view}, 2022.

\end{thebibliography}


\newpage

(T. Crin-Barat)\par\nopagebreak
\noindent\textsc{Chair for Dynamics, Control, Machine Learning and Numerics, Alexander Von Humboldt- Professorship, Department of
Mathematics, Friedrich-Alexander-Universität Erlangen-Nürnberg, 91058 Erlangen, Germany.}

Email address: {\tt timotheecrinbarat@gmail.com}

\vspace{3ex}

(L.-Y. Shou)\par\nopagebreak
\noindent\textsc{School of Mathematics and Key Laboratory of Mathematical MIIT, Nanjing University of Aeronautics and
Astronautics, Nanjing, 211106, P. R. China}

Email address: {\tt shoulingyun11@gmail.com}

\vspace{3ex}

(J. Tan)\par\nopagebreak
\noindent\textsc{Laboratoire de Math\'{e}matiques AGM, UMR CNRS 8088, Cergy Paris Universit\'{e}, 2 Avenue Adolphe Chauvin, 95302     Cergy-Pontoise Cedex, France}\par\nopagebreak
Email address: {\tt jin.tan@cyu.fr}

\end{document}